\pdfoutput=1
\documentclass[11pt]{article}

\usepackage{mymacros}

\begin{document}

\title{Componentwise Polish groupoids and equivalence relations}
\author{Ruiyuan Chen}
\date{}
\maketitle

\begin{abstract}
We study Borel equivalence relations equipped with a uniformly Borel family of Polish topologies on each equivalence class, and more generally, standard Borel groupoids equipped with such a family of topologies on each connected component.
Such ``componentwise Polish topologies'' capture precisely the topological information determined by the Borel structure of a Polish group action, by the Becker--Kechris theorem.
We prove that conversely, every abstract such Borel componentwise Polish groupoid obeying suitable axioms admits a Borel equivalence of groupoids to a global open Polish groupoid.
Together with known results, this implies that every such groupoid is Borel equivalent to an action groupoid of a Polish group action; in particular, the induced equivalence relations are Borel bireducible.

Our results are also valid for Borel groupoids with componentwise quasi-Polish topologies; and under stronger uniformity assumptions, we show that such groupoids in fact themselves admit global quasi-Polish topologies.
As a byproduct, we also generalize several standard tools for Polish groups and their actions to the setting of componentwise quasi-Polish groupoids, including Vaught transforms, Effros's theorem on orbits, and the open mapping theorem.

\let\thefootnote=\relax
\end{abstract}

\tableofcontents

\section{Introduction}
\label{sec:intro}

This paper is a contribution to the descriptive set theory and dynamics of Polish groups, groupoids, and Borel equivalence relations.
A \defn{Polish group} is a topological group whose topology is second-countable and completely metrizable.
Over the past several decades, Polish groups and their actions have played a central role in descriptive set theory and its connections to other areas, such as model theory, ergodic theory, representation theory, coarse geometry, Ramsey theory, etc.
See
\cite{becker-kechris:polgrp},
\cite{hjorth:book},
\cite{kechris-pestov-todorcevic:ramsey},
\cite{gao:idst},
\cite{kechris:ergodic},
\cite{rosendal:coarse}.

A fundamental tool in the study of Polish group actions $G \actson X$ is to consider, on each orbit $G \cdot x \subseteq X$, the quotient topology induced from the group topology on $G$ via the action $G ->> G \cdot x$.
The family of such quotient topologies, which we call the \defn{orbitwise topology}, forms a local topological structure on $X$ that is intermediate between the Borel sigma-algebra and a global Polish topology rendering the action continuous, and serves as the basis for many important techniques and arguments.
These include \emph{Vaught transforms} \cite{vaught:invariant}, which enable working with the orbitwise topology in a ``uniformly Borel'' manner, and have been extended to certain contexts without a group action \cite{benyaacov-doucha-nies-tsankov:scott}; Hjorth's ``orbit continuity lemma'' \cite{hjorth:book} (see also \cite{lupini-panagiotopoulos:games}), which shows that the orbitwise topology is preserved by Borel homomorphisms modulo Baire category; and the theory of \emph{idealistic equivalence relations} \cite{kechris:loccpt}, \cite{kechris-louveau:hypersmooth}, which formalizes the properties of the orbitwise meager ideals induced by the orbitwise topology.

The canonicity and significance of the orbitwise topology on a Polish group action is made precise by the \emph{Becker--Kechris topological realization theorem} \cite{becker-kechris:polgrp}, which shows, in one formulation, that the Borel orbitwise open sets are precisely those which can be made open in some finer Polish topology rendering the action continuous.
In other words, the orbitwise topology encodes precisely the topological rigidity of a Polish group action that is determined already by the Borel sigma-algebra.
In \cite{chen:beckec}, we gave a new proof of the Becker--Kechris theorem that emphasized this point of view; the techniques developed in that paper form the precursor to the present work.

\subsection{Classwise Polish equivalence relations}
\label{sec:intro-bocper}

The goal of this paper is to isolate the key features of the orbitwise topology on a Polish group action that enable many of the aforementioned techniques, without reference to the action itself.
That is, we propose an axiomatic notion of a ``Borel, locally Polish'' structure on a space, and seek to develop an analog of the orbit equivalence theory of Polish group actions, to the extent possible in this more abstract context.

\begin{definition}
\label{intro:def:bocper}
A \defn{Borel-overt classwise Polish equivalence relation} consists of a Borel equivalence relation $E \subseteq X^2$ on a standard Borel space $X$, together with a Polish topology on each $E$-class, such that these topologies are ``uniformly Borel'' in the sense of the following axioms:
\begin{enumerate}[roman]
\item \label{intro:def:bocper:basis}
there exists a countable family of Borel sets $U_i \subseteq E$ such that for each $x \in X$, the fibers $(U_i)_x = \set{y | (x,y) \in U_i}$ form a basis for the topology on $[x]_E$
(``uniform second-countability'');
\item \label{intro:def:bocper:overt}
for each such $U_i$, its projection $\pi_1(U_i) = \set{x | (U_i)_x \ne \emptyset} \subseteq X$ onto the first coordinate is a Borel set
(``uniformly testable nonemptiness of basic open sets'').
\end{enumerate}
Equivalently, \cref{intro:def:bocper:overt} means that $\pi_1(U)$ is Borel for \emph{any} Borel set $U \subseteq E$ whose vertical fibers are open.
This is the intended connotation of ``overt'', a term borrowed from constructive topology \cite{spitters:overt}.
(See \cref{def:fib-bor-qpol,rmk:fibqpolgpd}.)
\end{definition}

The motivating example is an orbit equivalence relation $E$ of a Borel action of a Polish group $G \actson X$, equipped with the orbitwise topology described above.
Note however that not every Polish group action induces a Borel equivalence relation.
Indeed, those that do are precisely those for which the orbitwise topology obeys the uniform second-countability axiom \cref{intro:def:bocper:basis} above, by another result of Becker--Kechris \cite[7.1.2]{becker-kechris:polgrp}.
See \cref{ex:polgrp-orbit}.

Our main result for Borel-overt classwise Polish equivalence relations is a representation theorem showing that, in at least one sense, all examples are reducible to this one:

\begin{theorem}[\cref{thm:bocqper-polact}; see also \cref{rmk:fibpolgpd-realiz}]
\label{intro:thm:bocper-polact}
Every Borel-overt classwise Polish equivalence relation $(X,E)$ is Borel bireducible with the orbit equivalence relation of a free Borel action of a Polish group $G \actson Y$.
\end{theorem}

Recall that a \defn{Borel reduction} $f : (X,E) -> (Y,F)$ between equivalence relations on standard Borel spaces is a Borel map $f : X -> Y$ which descends to an injection between quotients $X/E `-> Y/F$.
For Borel-overt classwise Polish equivalence relations, this implies that the image saturation $[f(X)]_F \subseteq Y$ is Borel and there is a Borel reduction $g : ([f(X)]_F, F) -> (X,E)$ such that $f, g$ descend to inverses of each other.
This is essentially due to idealisticity of such equivalence relations, as witnessed by the classwise meager ideal; see \cref{thm:fibqpolgpd-er-idl,thm:fibqpolgpd-equiv}.

We note that \cref{intro:thm:bocper-polact}, as well as idealisticity, fail for ``Borel classwise Polish equivalence relations'' not obeying the overtness axiom \labelcref{intro:def:bocper}\cref{intro:def:bocper:overt}.
This is analogous to the distinction between orbit equivalence relations of Polish group actions, versus arbitrary Borel equivalence relations: for example, a closed equivalence relation can have a quotient space which is analytic but not Borel, but this cannot happen for a Polish group action, due to idealisticity.

\subsection{Fiberwise Polish groupoids}
\label{sec:intro-bofpg}

\Cref{intro:def:bocper} is a special case of an analogous notion for groupoids, which are a more natural context in which to study such ``local topologies''.
Indeed, the proof of \cref{intro:thm:bocper-polact} works by treating $E$ as a groupoid, and is not made substantially easier by $E$ being an equivalence relation.

A \defn{groupoid} $(X,G)$ is, in short, a category with inverses, and consists concretely of two sets $X$ (of \emph{objects}) and $G$ (of \emph{morphisms}), a \emph{source} or ``domain'' map $\sigma : G -> X$, a \emph{target} or ``codomain'' map $\tau : G -> X$, as well as multiplication, identity, and inverse operations on morphisms, satisfying the usual axioms such as associativity and matching endpoint conditions.

A groupoid $(X,G)$ can be regarded as the common generalization of a group (when $X = \set{*}$ is a singleton), an equivalence relation (when $G \subseteq X^2$ and $\sigma, \tau$ are the projections), and a group action (when $G$ consists of morphisms $x -> gx$ ``labeled'' by acting group elements $g$).
For this reason, many notions in dynamics admit natural generalizations from group actions and/or equivalence relations to groupoids.
In particular, a descriptive set theory of Polish groupoids has been developed \cite{ramsay:polgpd}, \cite{lupini:polgpd}, paralleling and generalizing that of Polish groups and their actions.

Similarly to \cref{intro:def:bocper}, we propose an axiomatic notion of ``Borel, locally Polish'' groupoid, where the topology lives on the connected components of the groupoid.
Note however that the objects in a connected component do not in general capture the algebraic structure of the groupoid on that component (in the extreme case of a group, the objects are trivial).
Thus in contrast to \cref{intro:def:bocper}, where we took the classwise topology as primitive and then lifted it to the fibers of $E$ to express the uniformity condition, here we must take the fiberwise topology as primitive and then explicitly impose invariance between the fibers within a component.

\begin{definition}[see \cref{def:fibqpolgpd}]
\label{intro:def:fibpolgpd}
A \defn{Borel-overt fiberwise Polish groupoid} consists of a standard Borel groupoid $(X,G)$, together with a Polish topology on each fiber of the source map $\sigma^{-1}(x)$, i.e., on the set of morphisms starting at each $x \in X$; such that these topologies are invariant under right translation: for any morphism $g : x -> y \in G$, the bijection
\begin{eqaligned*}
\sigma^{-1}(y) &--> \sigma^{-1}(x) \\
h &|--> hg
\end{eqaligned*}
is a homeomorphism; and the topologies are ``uniformly Borel'' over all fibers, meaning
\begin{enumerate}[roman]
\item \label{intro:def:fibpolgpd:basis}
there exists a countable family of Borel sets $U_i \subseteq G$ such that for each $x \in X$, the fibers $U_i \cap \sigma^{-1}(x)$ form a basis for the topology on $\sigma^{-1}(x)$
(``uniform second-countability'');
\item \label{intro:def:fibpolgpd:overt}
for each such $U_i$, $\sigma(U_i) \subseteq X$ is a Borel set
(``uniformly testable nonemptiness of open sets'').
\end{enumerate}
\end{definition}

The motivating example is the action groupoid induced by an arbitrary Borel action of a Polish group (not necessarily with a Borel orbit equivalence relation); see \cref{ex:polgrp-action}.
We again prove that all examples are ``equivalent'' to this one:

\begin{theorem}[\cref{thm:fibqpolgpd-polact}; see also \cref{rmk:fibpolgpd-realiz}]
\label{intro:thm:fibpolgpd-polact}
Every Borel-overt fiberwise Polish groupoid admits a Borel equivalence of groupoids to an action groupoid of a Polish group action.
\end{theorem}

An \defn{equivalence of groupoids} (an instance of an \emph{equivalence of categories}) $F : (X,G) -> (Y,H)$ is a loosening of the natural notion of \emph{isomorphism of groupoids} that also takes morphisms \emph{within} $G, H$ into account.
A quick-and-dirty definition is that $F$ becomes an isomorphism after restricting to the full subgroupoids on subsets of objects $X' \subseteq X$ and $Y' \subseteq Y$ meeting every connected component.
Morally speaking, this means $(X,G), (Y,H)$ are the same up to containing redundant isomorphic copies of the same objects, analogous to a bireduction between equivalence relations.
See \cref{def:ftr-equiv}.
Note that the property of being an equivalence relation is preserved by an equivalence of groupoids; thus, \cref{intro:thm:fibpolgpd-polact} implies \cref{intro:thm:bocper-polact} as a special case.

As with equivalence relations, the axioms of \emph{Borel-overt fiberwise Polish groupoid} imply idealisticity of the connectedness relation, which in turn implies ``automatic existence of Borel inverse'' to a Borel equivalence of groupoids; see again \cref{thm:fibqpolgpd-er-idl,thm:fibqpolgpd-equiv}.

\subsection{Topological realization of componentwise quasi-Polish groupoids}
\label{sec:intro-comqpolgpd}

Our underlying main result on fiberwise topological groupoids does not mention groups at all:

\begin{theorem}[\cref{thm:fibqpolgpd-realiz-equiv-pol}]
\label{intro:thm:fibqpolgpd-realiz-equiv-pol}
Every Borel-overt fiberwise quasi-Polish groupoid $(X,G)$ admits a Borel equivalence of groupoids to a (globally) open Polish groupoid $(Y,H)$.
\end{theorem}

\defn{Quasi-Polish spaces} \cite{debrecht:qpol} are a canonical non-Hausdorff generalization of Polish spaces, obeying nearly all of the usual descriptive set-theoretic properties, and robust under some additional constructions that are particularly useful when working with Polish group actions and groupoids.
For instance, one characterization of quasi-Polish spaces is that they are precisely the $T_0$ quotients of Polish group actions on Polish spaces (see \cite[2.2]{chen:cbermd}).

An \defn{open topological groupoid} is a groupoid whose spaces of objects $X$ and morphisms $G$ are topological spaces, whose operations are continuous maps, and whose source map $\sigma : G -> X$ is open.
Every action groupoid of a continuous group action is an open topological groupoid; and the distinction between an open versus arbitrary (quasi-)Polish groupoid is analogous to the role of the overtness axiom \cref{intro:def:fibpolgpd:overt} in \cref{intro:def:fibpolgpd}.

In \cite{chen:polgpdrep}, we proved that every open Polish groupoid is Borel equivalent to a Polish group action.
This, together with \cref{intro:thm:fibqpolgpd-realiz-equiv-pol}, implies \cref{intro:thm:fibpolgpd-polact}, generalized to fiberwise \emph{quasi-}Polish groupoids.
Note that this generalization is new even for (globally) open quasi-Polish groupoids:

\begin{corollary}
Every open quasi-Polish groupoid is Borel equivalent to the action groupoid of a Polish group action.
(Hence in particular, the induced equivalence relations are Borel bireducible.)
\end{corollary}

The added flexibility of the quasi-Polish context is essential for the intermediate constructions used in the proofs of the above results, even in the Polish case, as we now explain.

Note that the statements of \cref{intro:thm:bocper-polact,intro:thm:fibpolgpd-polact,intro:thm:fibqpolgpd-realiz-equiv-pol} seem somewhat suboptimal, in that the input data of a ``Borel, locally topological'' structure need not be preserved by the resulting Borel equivalence.
It is natural to ask, for instance in \cref{intro:thm:fibqpolgpd-realiz-equiv-pol}, whether it is possible to find a global groupoid topology on the original $(X,G)$, restricting to the given fiberwise topology.

In \cref{ex:fibqpolgpd-disctscocy} we show that this is not always possible.
In order for a fiberwise topological groupoid $(X,G)$ to admit a compatible global topology, the ``difference'' operation on morphisms, $(g,h) |-> g^{-1}h$, must be continuous when restricted to each fiber of the source map $\sigma : G -> X$.
This turns out to be the only ``local'' obstruction (see \cref{thm:comqpolgpd-qpol-connected}); however, we are unable to prove that this condition suffices for a global topological realization.
Therefore in \cref{def:ucomqpolgpd} we formulate a further strengthening: we say a Borel-overt fiberwise quasi-Polish groupoid $(X,G)$ has \defn{uniformly fiberwise continuous differences} if, roughly speaking, the difference map $(g,h) |-> g^{-1}h$ is continuous on $\sigma$-fibers, and this can be witnessed ``in a Borel way''; we also say that $G$ is a \defn{Borel-overt uniformly componentwise quasi-Polish groupoid} in this case.

\begin{theorem}[\cref{thm:ucomqpolgpd-qpol}]
\label{intro:thm:ucomqpolgpd-qpol}
Let $(X,G)$ be a Borel-overt uniformly componentwise quasi-Polish groupoid.
Then there exist compatible quasi-Polish topologies on $X, G$ turning $G$ into an open quasi-Polish groupoid, and restricting to the originally given fiberwise topology on $G$.
\end{theorem}

This is our main topological realization theorem for ``Borel, locally topological'' groupoids, to which all aforementioned results reduce.
To derive \cref{intro:thm:fibqpolgpd-realiz-equiv-pol} from it, there are two additional steps needed, both of which involve passing to an equivalent ``componentwise comeager'' subgroupoid.
Namely, we must improve the original fiberwise quasi-Polish groupoid to a uniformly componentwise quasi-Polish groupoid; and we must make the resulting quasi-Polish groupoid Polish.

To handle the first issue, we use a groupoid version of an argument of Solecki--Srivastava \cite{solecki-srivastava:lpolgrp}.
Note that the special case of \cref{intro:thm:fibqpolgpd-realiz-equiv-pol} for a group $G$ says that a standard Borel group equipped with a compatible right-translation-invariant (quasi-)Polish topology is already a Polish group; this is a weak form of the main result of \cite{solecki-srivastava:lpolgrp}.
Running an analog of the argument of \cite{solecki-srivastava:lpolgrp} in the context of \cref{intro:thm:fibqpolgpd-realiz-equiv-pol} yields instead a ``componentwise comeager'' subgroupoid with uniformly fiberwise continuous differences, on which we may then apply \cref{intro:thm:ucomqpolgpd-qpol}.

In order to obtain a global Polish groupoid topology, one might hope for an analog of \cref{intro:thm:ucomqpolgpd-qpol} that assumes the given fiberwise topology is fiberwise Polish.
In \cref{ex:fibpolgpd-nonpol}, we show that this is insufficient to ensure global Polishability.
Instead, in \cref{thm:ucompolgpd} we give a more involved characterization of the Borel-overt uniformly componentwise quasi-Polish groupoids $(X,G)$ which admit a compatible global Polish topology.
Roughly speaking, the criterion amounts to the existence of ``fundamental sequences of identity neighborhoods $W \subseteq G$'' containing \emph{all} identity morphisms $1_x$; the connection between this condition and metrizability was first shown by \cite{ramsay:polgpd}, and was used as a key ingredient in \cite{chen:polgpdrep}.
We then show in \cref{thm:fibqpolgpd-pol} that this condition can always be achieved, again after passing to a componentwise comeager subgroupoid.

\subsection{Componentwise topologies}
\label{sec:intro-orbtop}

We now give an overview of our approach to proving our main topological realization theorem \labelcref{intro:thm:ucomqpolgpd-qpol}, which involves systematically developing the theory of Borel-overt fiberwise quasi-Polish groupoids in a manner that we believe to be of independent interest.

Given a topological groupoid $(X,G)$, for each object $x \in X$, the target map $\tau : G -> X$ restricts to a continuous map from the set $\sigma^{-1}(x)$ of morphisms starting at $x$ onto the connected component $[x]_G \subseteq X$ of $x$.
We define the \defn{componentwise topology on objects} $\@O_G(X)$ of a fiberwise topological groupoid to be the disjoint union of the quotient topologies induced by all such $\tau : \sigma^{-1}(x) ->> [x]_G$; this serves as an upper bound on any compatible global groupoid topology on $X$.
We write $\@{BO}_G(X) \subseteq \@O_G(X)$ for the Borel, componentwise open sets.
When $G$ is already an open quasi-Polish groupoid, it follows from the version of the Becker--Kechris theorem from \cite{chen:beckec} that any $B \in \@{BO}_G(X)$ can in fact be made open in a finer compatible quasi-Polish topology.

\begin{figure}
\centering
\begin{tikzcd}[
    column sep=3em,
    supseteq/.style={hookleftarrow},
    preimage/.style={dashed},
    exists*/.style={dotted},
]
&[-1em]&[-1em]&[1em]
\@{BO}_\sigma(G)
    \drar[supseteq]
    \ar[dddr, exists*, two heads, "\exists^*_\tau"'{inner sep=0pt, pos=.2}]
&[1em]
\\[-1em]
\text{morphisms} &
\@B(G)
    \rar[supseteq]
    \ar[dd, exists*, shift right, "\exists^*_\sigma"'{pos=.2}]
    \ar[dd, exists*, shift left, "\exists^*_\tau"{pos=.2}]
    &
\@{BO}_\Box(G)
    \urar[supseteq]
    \drar[supseteq]
    &&
\@{BO}_G(G)
    \rar[supseteq]
    &
\@O(G)
    \ar[dd, exists*, looseness=1.5, shift right, bend right=45, "\sigma"'{}]
    \ar[dd, exists*, looseness=1.5, shift left, bend left=45, "\tau"{}]
\\[-1em]
&&& \@{BO}_\tau(G)
    \urar[supseteq]
    \ar[dr, exists*, two heads, "\exists^*_\sigma"'{inner sep=0pt}]
\\
\text{objects} &
\@B(X)
    \ar[rrr,supseteq]
    \ar[uuurr, preimage, "\sigma^{-1}"{pos=.2, inner sep=0pt}]
    \ar[urr, preimage, "\tau^{-1}"'{inner sep=0pt}]
    &&&
\@{BO}_G(X) \rar[supseteq]
    \ar[uu, preimage, shift left, "\sigma^{-1}"]
    \ar[uu, preimage, shift right, "\tau^{-1}"']
    &
\@O(X)
    \ar[uu, preimage, shift left, "\sigma^{-1}"]
    \ar[uu, preimage, shift right, "\tau^{-1}"']
\end{tikzcd}
\caption{Fiberwise and componentwise sigma-topologies on spaces of morphisms $G$ and objects $X$, for a Borel-overt fiberwise quasi-Polish groupoid $(X,G)$, or an open quasi-Polish groupoid (for $\@O(G), \@O(X)$).
Solid arrows are inclusions; dashed arrows are preimage; dotted arrows are category quantifiers for the respective fiberwise topology.
See \cref{sec:orbtop} for precise definitions.}
\label{fig:topologies}
\end{figure}
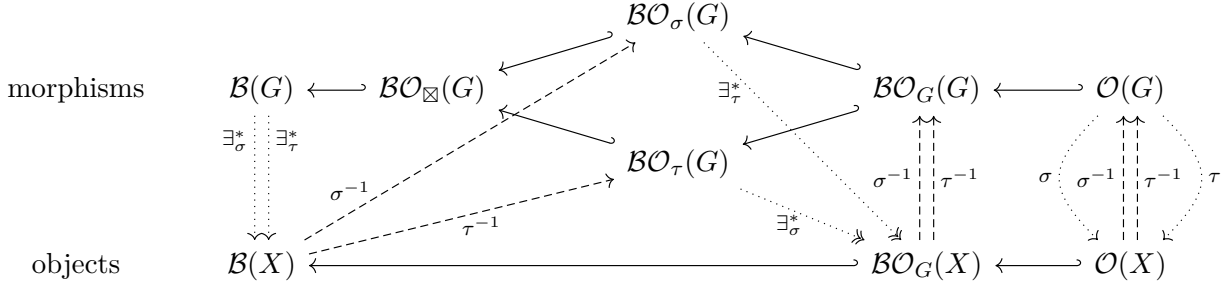

Via an analogous but more involved computation, one may determine, from the given fiberwise topology, a canonical upper bound on the topology on the space of morphisms $G$,
which we call the \defn{componentwise topology on morphisms} $\@O_G(G)$.
See \cref{sec:orbtop} for details.

\Cref{fig:topologies} depicts the Borel parts of these canonical topologies on a Borel-overt fiberwise quasi-Polish groupoid $(X,G)$, defined starting from the $\sigma$-fiberwise topology which we denote by $\@O_\sigma(G)$.
By the Becker--Kechris argument noted above, if \cref{intro:thm:ucomqpolgpd-qpol} were true, then $\@{BO}_G(X), \@{BO}_G(G)$ would be the directed unions of all compatible open quasi-Polish groupoid topologies $\@O(X), \@O(G)$.
Note that $\@{BO}_G(X), \@{BO}_G(G)$ are not themselves topologies, but rather \emph{sigma-topologies}, i.e., closed under countable unions only.
Nonetheless, we may infer that $X, G$ equipped with $\@{BO}_G(X), \@{BO}_G(G)$ should form an ``open sigma-topological groupoid''.
Our core technical result, \cref{thm:ucomqpolgpd}, shows that this is indeed the case.
Proving ``sigma-continuity'' of multiplication is the hardest part, and requires the uniform fiberwise continuity of differences assumption in \labelcref{intro:thm:ucomqpolgpd-qpol}.

\subsection{Point-free methods}
\label{sec:intro-localic}

Our approach to producing compatible Polish and quasi-Polish topologies follows the approach to the Becker--Kechris theorem from \cite{chen:beckec}, and is ultimately inspired by connections to \emph{point-free topology} (also known as \emph{locale theory}); see \cite{chen:borloc} for a comprehensive survey.

A topology may be seen as an algebraic structure equipped with the operations of finite $\cap$ and arbitrary $\bigcup$ (called a \emph{frame}).
The quasi-Polish topologies may be characterized as precisely those which can be presented by countably many generators and relations \cite{heckmann:qpol}.
By general universal algebra, a retract of a countably presented algebra remains so; because $\cap$ is finitary, it may be absorbed into the presentation, so that it is enough to have a $\bigcup$-preserving retract \cite{joyal-tierney:locgpd}, \cite{johnstone:stone}.
One instance of this is classical: Sierpiński's theorem \cite[8.19]{kechris:cdst}, \cite[Theorem~40]{debrecht:qpol} shows that an open $T_3$ ($T_0$) quotient of a (quasi-)Polish space is still (quasi-)Polish, and can be proved by treating the quotient topology as a union-preserving retract via image and preimage.\break
In \cite[2.7.5]{chen:beckec}, we gave a generalization of Sierpiński's theorem using ``Baire-categorical image'', i.e., Vaught transforms in the context of groupoids; we use this result (\cref{thm:sierpinski-exists*}) to derive \cref{intro:thm:ucomqpolgpd-qpol} from the aforementioned \cref{thm:ucomqpolgpd} for sigma-topological groupoids.

In this paper, compared to \cite{chen:beckec}, we are less pedantic about doing everything point-free, and use point-based arguments when they are simpler.
Nonetheless, we often find point-free reasoning more convenient, manipulating open and Borel sets algebraically without mentioning their elements.
Indeed, we will often manipulate \emph{collections} of sets as a whole; this allows for cleaner and more concise calculations when dealing with the various topologies that we introduce.

Because this writing style may be less familiar, we give a brief illustration here, by recalling a classical result, Pettis's automatic continuity theorem for Polish groups.
The flavor of its proof is emblematic of several in this paper (see e.g., \cref{thm:vaught-fibact,thm:fibtop-meager-exists*,thm:ucomqpolgpd}).

\begin{theorem}[Pettis]
\label{intro:thm:pettis}
The topology of a Polish group $G$ is uniquely determined by its group structure and its Borel sigma-algebra.
Namely, $U \subseteq G$ is open iff its preimage under the group multiplication $\mu : G^2 -> G$ is a countable union of Borel rectangles.
\end{theorem}
\begin{proof}
$\Longrightarrow$ is clear.
For $\Longleftarrow$, suppose $\mu^{-1}(U) = \bigcup_{i \in \#N} (A_i \times B_i)$ where $A_i, B_i \subseteq G$ are Borel.
By the Baire property, find open $V_i, W_i \subseteq G$ which are equal to $A_i, B_i$ respectively mod meager.
Then
\begin{eqaligned*}
g \in U
&\iff  \set{h \in G | (gh^{-1},h) \in \mu^{-1}(U)} \text{ is nonmeager} \\
&\iff  \exists i\, (\set{h \in G | (gh^{-1},h) \in A_i \times B_i} = A_i^{-1} g \cap B_i \text{ is nonmeager}) \\
&\iff  \exists i\, (V_i^{-1} g \cap W_i \text{ is nonmeager})
\iff  \exists i\, (g \in V_i \cdot W_i),
\end{eqaligned*}
so $U = \bigcup_i (V_i \cdot W_i)$ is open.

In the algebraic notation we will be using, this argument can be written as follows:
\begin{equation*}
\exists^*_\mu(\@B(G) \otimes \@B(G))
= \@B(G) \oast \@B(G)
= \@O(G) \oast \@O(G)
= \@O(G) \odot \@O(G)
= \@O(G),
\end{equation*}
where $\@O(G), \@B(G)$ denote the topology and Borel sigma-algebra respectively (\cref{def:top}),
$\exists^*_\mu$ denotes Baire category quantifier under $\mu$ (\cref{def:fib-top}),
and $\otimes, \oast, \odot$ denote countable unions of rectangles, Vaught transforms, and products respectively (\cref{def:ostar}).
\end{proof}

\subsection{Other connections and open questions}
\label{sec:intro-misc}

At the end of the paper in \cref{sec:misc}, we take this opportunity to generalize several known results, that fit naturally into the ``fiberwise topological'' framework of this paper.

Effros's theorem on orbits of Polish group actions \cite{effros:groups} was generalized to open Polish groupoids by \cite{ramsay:polgpd}, \cite{lupini:polgpd}.
We further extend this to the quasi-Polish setting:

\begin{theorem}[\cref{thm:effros}]
\label{intro:thm:effros}
Let $(X,G)$ be a Borel-overt componentwise quasi-Polish groupoid.
Suppose $G$ is connected.
Then the componentwise topologies $\@O_G(X), \@O_G(G)$ form the unique open quasi-Polish topology on $(X,G)$ such that $\@O_G(G)$ restricts to the given fiberwise topology.
\end{theorem}

Using this, we prove an open mapping theorem for componentwise quasi-Polish groupoids, that also extends the characterization from \cite{becker-kechris:polgrp} of Polish $G$-spaces with Borel orbit equivalence relation (previously generalized to open Polish groupoids by \cite{lupini:polgpd}); see \cref{thm:openmap,thm:ucomqpolgpd-erquot}.

It is also interesting to note that we may formally deduce the Becker--Kechris theorems for groupoids from \cite{chen:beckec} and \cite{lupini:polgpd} from our main \cref{intro:thm:ucomqpolgpd-qpol}; see \cref{thm:beckec,thm:beckec-pol}.
Thus in some sense, topological realization for groupoids ``subsumes'' that for actions.

As noted in \cref{sec:intro-comqpolgpd}, \cref{intro:thm:ucomqpolgpd-qpol} necessarily requires the assumption that the difference map $(g,h) |-> g^{-1}h$ is $\sigma$-fiberwise continuous; but we are only able to prove the result assuming a ``uniform'' strengthening of this assumption.
We do not know if the uniform and non-uniform versions are equivalent.
If true, this would be particularly useful for classwise quasi-Polish equivalence relations, which trivially have (non-uniform) fiberwise continuity of differences.

\Cref{ex:fibpolgpd-nonpol} shows that \cref{intro:thm:ucomqpolgpd-qpol} can in general produce only an open quasi-Polish groupoid, not a Polish groupoid, even when given a fiberwise Polish groupoid as input.
However, we do not know if such a failure can happen for a classwise Polish equivalence relation, which trivially admit an open Polish groupoid topology on each component, similarly to in the preceding question.
If both questions could be solved positively for equivalence relations, then \cref{intro:thm:bocper-polact} can likely be strengthened to yield a classwise topological embedding into the resulting free Polish group action with the orbitwise topology, since it is known that the last step in the proof, namely \cite{chen:polgpdrep}, can be modified to produce a Borel embedding (see Remark 9.11 in that paper).

As mentioned at the start of this introduction, a well-known counterpart to the componentwise topologies we are studying is given by the theory of \emph{idealistic equivalence relations}, originally introduced by \cite{kechris:loccpt}, \cite{kechris-louveau:hypersmooth}.
There is a family of long-standing conjectures regarding the precise connection between the notions of \emph{Polish group action}, \emph{idealistic equivalence relation}, and equivalence relations not Borel reducing the relation $\#E_1$ of eventual equality on $\#R^\#N$; see \cite{hjorth-kechris:recent}, as well as \cite{calderoni-mottoros:idealistic}, \cite{solecki:filtration} for some recent related work.
In the context of this paper, a natural counterpart to these conjectures would be to generalize \cref{intro:def:fibpolgpd} to a notion of \emph{Borel-idealistic groupoid}, consisting of a Borel groupoid with a family of sigma-ideals on the $\sigma$-fibers, and to ask whether every such groupoid is Borel equivalent to a Polish group action.
See \cref{rmk:idlgpd}.

Another interesting direction would be to generalize the notion of \emph{Borel-overt classwise (quasi-\nobreak)\linebreak[0]Polish equivalence relation} (\cref{intro:def:bocper}) to analytic equivalence relations with Borel classes.
This would offer an alternative approach, different from \cref{intro:def:fibpolgpd}, to extending the scope of our framework to encompass orbitwise topologies on arbitrary Polish group actions.

\subsection*{Acknowledgments}

Research supported by NSF grant DMS-2224709.

\section{Topological preliminaries}
\label{sec:prelim}

The basic topological framework and notation of this paper follows \cite{chen:beckec}.
In particular, we will be working mostly in the topological context of de~Brecht's \emph{quasi-Polish spaces} \cite{debrecht:qpol}; and we will make extensive use of the theory of \emph{Borel bundles of quasi-Polish spaces} as developed in \cite[\S2]{chen:beckec}.
In this section, we will summarize the main definitions and facts, as well as prove some additional facts we will need which were not present in \cite{chen:beckec}.

\subsection{Topologies and Borel structures}

\begin{notation}
\label{def:top}
For a topological space $X$, we denote its topology (i.e., collection of open sets) by $\@O(X) \subseteq \@P(X)$.
Similarly, for a Borel space $X$, we denote its Borel sigma-algebra%
\footnote{We will avoid using ``$\sigma$-'' as a prefix meaning ``countable-arity'', and will always spell out ``sigma-'' instead, because the letter $\sigma$ will be used extensively for the source map of a groupoid; see \cref{def:gpd}.}
by $\@B(X)$.

We will sometimes need to consider multiple topologies on the same set; these will be denoted by other symbols like $\@S, \@T, \@O'(X)$, etc.
The symbol $\@O(X)$ is reserved for a ``default'' topology that $X$ is considered to be ``equipped'' with.
\end{notation}

Since we will be working with non-Hausdorff spaces, we use Selivanov's definition of the Borel hierarchy \cite{selivanov:domdst}, which agrees with the classical definition (see \cite{kechris:cdst}) in metrizable spaces.

\begin{definition}[Selivanov \cite{selivanov:domdst}]
\label{def:bor}
The \defn{Borel hierarchy} of a general topological space $X$ is defined by letting
$\*\Sigma^0_1(X) := \@O(X)$ be the open sets and $\*\Pi^0_1(X)$ be the closed sets, and for $\xi \ge 2$,
\begin{eqaligned*}
\*\Sigma^0_\xi(X) &:= \set{\bigcup_i (U_i \setminus V_i) | U_i, V_i \in \*\Sigma^0_{\zeta_i}(X),\, \zeta_i < \xi}, \\
\*\Pi^0_\xi(X) &:= \set{\bigcap_i (U_i => V_i) | U_i, V_i \in \*\Sigma^0_{\zeta_i}(X),\, \zeta_i < \xi}
    = \set{\neg A | A \in \*\Sigma^0_\xi(X)},
\end{eqaligned*}
where $i$ ranges over a countable index set and $(U => V) := \neg U \cup V$ (where $\neg$ means complement).
As usual, put also $\*\Delta^0_\xi := \*\Sigma^0_\xi \cap \*\Pi^0_\xi$.
\end{definition}

\begin{definition}[de~Brecht \cite{debrecht:qpol}]
\label{def:qpol}
A \defn{quasi-Polish space} is a topological space homeomorphic to a $\*\Pi^0_2$ subspace of $\#S^\#N$,
where $\#S$ is the \defn{Sierpiński space} $\set{0,1}$ with $\set{1}$ open but not closed.
\end{definition}

We will freely use the following properties of quasi-Polish spaces.
For proofs, see \cite{debrecht:qpol}, \cite{chen:qpol}.
\begin{eqenumerate}

\item \label{it:qpol-t0-cb-baire}
All quasi-Polish spaces are $T_0$, second-countable, standard Borel, and completely Baire.

\item \label{it:qpol-pol}
A topological space is Polish iff it is quasi-Polish and regular ($T_3$).

\item \label{it:qpol-dis}
If $X$ is a quasi-Polish space, and $A_i \in \*\Sigma^0_\xi(X)$ are countably many sets, then there is a finer quasi-Polish topology containing each $A_i$ and contained in $\*\Sigma^0_\xi(X)$.
In more detail,
\begin{enumerate}[alph]
\item \label{it:qpol-dis:clopen}
adjoining a single $\*\Delta^0_2$ set to the topology of $X$ preserves quasi-Polishness;
\item \label{it:qpol-dis:join}
if the intersection of countably many quasi-Polish topologies contains a second-countable $T_0$ topology, then their union generates a quasi-Polish topology.
\end{enumerate}

\item \label{it:qpol-prod}
Countable products of quasi-Polish spaces are quasi-Polish.

\item \label{it:qpol-pi02}
A subspace of a quasi-Polish space is quasi-Polish iff it is $\*\Pi^0_2$ (in the sense of \cref{def:bor}).

\item \label{it:qpol-openquot}
Let $X$ be quasi-Polish space, $E \subseteq X^2$ be an equivalence relation such that the quotient map $X ->> X/E$ to the quotient topology is open, i.e., the $E$-saturation of every open set is open.
Then $X/E$ is quasi-Polish iff it is $T_0$, iff $E \subseteq X^2$ is $\*\Pi^0_2$, iff every $E$-class $[x]_E \subseteq X$ is $\*\Pi^0_2$.

\end{eqenumerate}
(The last result is a version of Sierpiński's theorem for quasi-Polish spaces \cite[Theorem~40]{debrecht:qpol}; see also \cite[10.1]{chen:qpol}, \cite[3.4.4]{gao:idst}.)

\begin{definition}
\label{def:baire}
Given a topological space $X$ and $A, B \subseteq X$, we write
\begin{eqaligned*}
A \subseteq^* B
&\coloniff  A \subseteq B \text{ mod meager}
\iff  A \setminus B \text{ is meager}, \\
A =^* B
&\coloniff  A = B \text{ mod meager}
\iff  A \subseteq^* B \subseteq^* A.
\end{eqaligned*}

Recall that $A \subseteq X$ has the \defn{Baire property} if $A =^* U_A$ for some open $U_A$.
Every Borel $A \subseteq X$ has the Baire property, where $U_A$ may be computed inductively using the formulas
\begin{eqaligned}[][\label{eq:bp}\left\{][\right.]
U_V &:= V
    \quad \text{if $V \in \@O(X)$}, \\
U_{\bigcup_i A_i} &:= \bigcup_i U_{A_i}, \\
U_{A \setminus B} &:= U_A \setminus \-{U_B} = \bigcup_{W \in \@W; W \cap U_B = \emptyset} (W \cap U_A)
\end{eqaligned}
for any open basis $\@W \subseteq \@O(X)$.

Recall also that every dense $\*\Pi^0_2$ set, in the sense of \cref{def:bor}, contains an intersection of dense open sets (see e.g., \cite[7.2]{chen:qpol}).
Thus dense $\*\Pi^0_2$ sets also generate the comeager sigma-filter.
\end{definition}

\subsection{Sigma-topologies, clubs of topologies, and operations on topologies}

As in \cite{chen:beckec}, the following common generalization of topologies and Borel sigma-algebras will serve as a useful way to organize change-of-topology arguments, by capturing a collection of open sets with desirable properties among which we eventually want to find a finer quasi-Polish topology, without requiring the specific topology to be chosen just yet.

\begin{definition}[{see \cite[2.2.10--13]{chen:beckec}}]
\label{def:stop}
A \defn{sigma-topology} $\@S$ on a set $X$ is a collection of subsets, closed under \emph{countable} unions and finite intersections.

The notions of \defn{product sigma-topology} on a product of sigma-topological spaces $X \times Y$, \defn{subspace sigma-topology} on a subset $Y \subseteq X$ of a sigma-topological space, and \defn{continuous map} and \defn{open map} between sigma-topological spaces are defined the usual way, in analogy with the corresponding topological notions.

We call a sigma-topological space $(X,\@S)$ \defn{regular} if every $U \in \@S$ is the union of \emph{countably} many $V_i \in \@S$ each of which is contained in some $F_i \subseteq U$ with $\neg F_i \in \@S$.
(Note: ``closure'' $\-{V_i}$ is not defined.)

Finally, for a standard Borel space $X$, we call a sigma-topology $\@S$ on $X$ \defn{compatible} (with the Borel structure) if any countably many sets $U_i \in \@S$ are contained in a single quasi-Polish topology which is contained in $\@S$.
\end{definition}

\begin{remark}
\label{rmk:stop-club}
If $\@S$ is a compatible sigma-topology on $X$, then for a desired property of quasi-Polish topologies $\@T \subseteq \@S$ (e.g., being Polish), the following assertions are equivalent:
\begin{itemize}
\item
Any countably many sets $U_i \in \@S$ are contained in a quasi-Polish $\@T \subseteq \@S$ with that property.
\item
Cofinally many quasi-Polish topologies $\@T \subseteq \@S$ have that property.
\end{itemize}
If cofinally many quasi-Polish $\@T \subseteq \@S$ have a given property, and also the topology generated by any countable increasing union $\bigcupup_n \@T_n$ of topologies $\@T_n \subseteq \@S$ with the property also has the property, then following \cite[5.1.4]{becker-kechris:polgrp} we say that a \defn{club of quasi-Polish topologies $\@T \subseteq \@S$} has the property.
It is straightforward to verify that a countable intersection of clubs is still a club (see \cite[3.4]{chen:cbermd}).
\end{remark}

\begin{lemma}
\label{thm:stop-pol}
If $\@S$ is a compatible \emph{regular} sigma-topology on a standard Borel space $X$, then the \emph{Polish} topologies $\@T \subseteq \@S$ form a club.
\end{lemma}
\begin{proof}
Clearly the collection of Polish topologies contained in $\@S$ is closed under countable increasing joins (see \cite[13.3]{kechris:cdst}), so it suffices to verify cofinality.
Let $\@T_0 \subseteq \@S$ be any quasi-Polish topology.
Inductively given $\@T_n$, let $\@T_n \subseteq \@T_{n+1} \subseteq \@S$ be a finer quasi-Polish topology containing a countable open basis for $\@T_n$, as well as for each $U$ in said basis, countably many witnesses $V_i, \neg F_i \in \@S$ to regularity for $U$, where $V_i \subseteq F_i \subseteq U$ and $U = \bigcup_i V_i$.
Finally take the topology generated by $\bigcupup_n \@T_n$, which is quasi-Polish by \cref{it:qpol-dis}\cref{it:qpol-dis:join} and regular by construction.
\end{proof}

\begin{remark}
The Borel hierarchy (\cref{def:bor}) also makes sense for a sigma-topological space.
Note that our definition of \emph{regular sigma-topology} $\@S$ implies that $\*\Pi^0_2 = \@G_\delta$ = countable intersections of sets in $\@S$, as in Polish spaces.
\end{remark}

\begin{example}
The following are compatible sigma-topologies:
\begin{itemize}
\item
A compatible second-countable sigma-topology is just a compatible quasi-Polish topology.
\item
If $X$ is a quasi-Polish space, then $\*\Sigma^0_\xi(X)$ is compatible for each $\xi < \omega_1$ (by \labelcref{it:qpol-dis}).
\item
If $X$ is a standard Borel space, then $\@B(X)$ is a compatible sigma-topology.
\item
For compatible sigma-topologies on $X, Y$, the product sigma-topology on $X \times Y$ is compatible.
\item
For a compatible sigma-topology $\@S$ on $X$, its restriction to any $\*\Pi^0_2(X,\@S)$ subset is compatible.
\end{itemize}
See also \cref{thm:fib-bor-qpol,thm:ucomqpolgpd} for further examples.
\end{example}

\begin{notation}[see {\cite[2.1.7]{chen:beckec}}]
\label{def:ostar}
If $\@S, \@T$ are two collections of sets, and $\star$ is a binary operation on sets such as $\cap$ or $\times$, we write
\begin{align*}
\@S \ostar \@T &:= \set{\bigcup_{i \in I} (A_i \star B_i) | I \text{ countable},\, A_i \in \@S,\, B_i \in \@T}.
\end{align*}
For example, if $\@S, \@T$ are second-countable topologies (or more generally sigma-topologies) on a set $X$, then $\@S \ocap \@T$ is that generated by $\@S \cup \@T$, while $\@S \otimes \@T$ is the product (sigma-)topology on $X^2$.

We will occasionally use the analogous notion with arbitrary unions instead of countable ones, which will be denoted $\ostar^\infty$ instead of $\ostar$.
For example, for arbitrary topologies $\@S, \@T$, $\@S \otimes^\infty \@T$ is the product topology (whereas $\@S \otimes \@T$ might not be closed under arbitrary unions).
\end{notation}

\subsection{Bundles of topological spaces}

\begin{definition}[{see \cite[\S2.3]{chen:beckec}}]
\label{def:fib-top}
By a \defn{bundle} over a set $Y$, we mean an arbitrary map $p : X -> Y$, where we regard $(X,p)$ as the family of \defn{fibers} $X_y := p^{-1}(y)$ parametrized over $y \in Y$.
When we say a property or structure occurs \defn{$p$-fiberwise} in $X$, we mean it occurs in each $X_y$.

In particular, a \defn{$p$-fiberwise topology} on $X$ is a family of topologies $\@O(X_y)$ on each $X_y$.
We may identify such a family of topologies with a single global topology, namely the disjoint union topology of the fiber topologies, consisting of the \defn{fiberwise open} sets, which we denote by $\@O_p(X)$.
We call $X$ equipped with a fiberwise topology a \defn{bundle of topological spaces} over $Y$.

Given a (global) topology $\@O(X)$ on $X$, we may \defn{fiberwise restrict} it to get a fiberwise topology, whose corresponding disjoint union topology $\@O_p(X)$ then refines the original topology $\@O(X)$.

Terms such as \defn{fiberwise dense} or \defn{fiberwise meager} set $A \subseteq X$ have the self-evident meaning.
We have the \defn{Baire category quantifiers} mapping $A \subseteq X$ to
\begin{eqaligned*}
\exists^*_p(A) &:= \set{y \in Y | A_y \subseteq X_y \text{ is nonmeager}}, \\
\forall^*_p(A) &:= \set{y \in Y | A_y \subseteq X_y \text{ is comeager}} = \neg \exists^*_p(\neg A).
\end{eqaligned*}
For $A, B \subseteq X$, we write
\begin{eqaligned*}
A \subseteq^*_p B
&\coloniff  A \subseteq B \text{ mod fiberwise meager}
\iff  \forall^*_p(A => B) = Y, \\
A =^*_p B
&\coloniff  A = B \text{ mod fiberwise meager}
\iff  A \subseteq^*_p B \subseteq^*_p A.
\end{eqaligned*}
\end{definition}

\begin{definition}
\label{def:fib-pb}
If $p : X -> Y$ is a bundle as above, and $f : Z -> Y$ is another map,
\begin{eqaligned*}
Z \times_Y X := Z \tensor[_f]\times{_p} Y := \set{(z,x) \in Z \times_Y X | f(z) = p(x)}
\end{eqaligned*}
is the \defn{pullback} or \defn{fiber product} of $Z, X$ over $Y$, fitting into a commutative square
\begin{eqtikzcd*}
Z \times_Y X \dar["\pi_1"'] \rar["\pi_2"] & X \dar["p"] \\
Z \rar["f"] & Y
\end{eqtikzcd*}
where $\pi_1, \pi_2$ are the projections.
While the definition is symmetric in $X, Z$, if we regard $X$ as a bundle over $Y$, then it is natural to regard $Z \times_Y X$ as a bundle over $Z$ via $\pi_1$, where each fiber $(Z \times_Y X)_z = \{z\} \times X_{f(z)}$ is canonically identified via $\pi_2$ with the fiber $X_{f(z)}$.

If $X$ has been equipped with a $p$-fiberwise topology $\@O_p(X)$ over $Y$, then by default we equip $Z \times_Y X$ with the $\pi_1$-fiberwise topology over $Z$ given by copying the topology on each $X_{f(z)}$ to $(Z \times_Y X)_z$ via the above identification.

If $Z$ has also been equipped with a fiberwise topology over $Y$, then we may also equip $Z \times_Y X$ with the fiberwise product topology over $Y$, which restricts to the $\pi_1$-fiberwise topology.

By virtue of the fiberwise nature of their definition, the category quantifiers obey
\begin{eqenumerate}

\item \label{it:fib-baire-bc}
(\defn{Beck--Chevalley condition})
For a pullback bundle of spaces as above, for $A \subseteq X$,
\begin{equation*}
f^{-1}(\exists^*_p(A)) = \exists^*_{\pi_1}(\pi_2^{-1}(A)).
\end{equation*}

\item \label{it:fib-baire-frob}
(\defn{Frobenius reciprocity})
In particular, for $Z \subseteq Y$,
\begin{equation*}
Z \cap \exists^*_p(A) = \exists^*_p(p^{-1}(Z) \cap A).
\end{equation*}

\item \label{it:fib-baire-surj}
In particular, for $Z \subseteq \exists^*_p(X)$,
\begin{equation*}
Z = \exists^*_p(p^{-1}(Z)).
\end{equation*}

\end{eqenumerate}
\end{definition}

\subsection{The Kunugi--Novikov theorems}

We are interested in bundles of well-behaved (Polish or quasi-Polish) spaces as above over a standard Borel base space $Y$, where the topologies are ``uniformly Borel''.
The key tool enabling such a theory is the \emph{Kunugi--Novikov uniformization theorem}, which says that every Borel fiberwise open set is automatically ``uniformly fiberwise open''; see \cite[28.7]{kechris:cdst}, which is the classical statement in the case of a product bundle $X \times Y -> Y$.
In \cite[8.14]{chen:polgpdrep} and \cite[2.4.1]{chen:beckec} it was pointed out that the same proof works verbatim for arbitrary Borel bundles.

In this paper, we will need even more general forms of the Kunugi--Novikov theorem, which nonetheless follow from the same argument.
For the sake of completeness, we will give several versions of the theorem.
The core result is a simple reflection argument:

\begin{theorem}[Kunugi--Novikov reflection]
\label{thm:kunugi-novikov-refl}
Let $p : X -> Y$ be a Borel map between standard Borel spaces, $(U_i)_{i \in I}$ be a countable family of Borel sets in $X$.
Let $A \subseteq X$ be $\*\Sigma^1_1$, and suppose
\begin{align*}
A &\subseteq \bigcup_i (p^{-1}(C_i) \cap U_i) \\
\intertext{for $\*\Pi^1_1$ sets $C_i \subseteq Y$.
Then there exist Borel $B_i \subseteq C_i$ such that}
A &\subseteq \bigcup_i (p^{-1}(B_i) \cap U_i).
\end{align*}
\end{theorem}
\begin{proof}
By the Novikov separation theorem (see \cite[28.5]{kechris:cdst}), there exist Borel
$V_i \subseteq p^{-1}(C_i) \cap U_i$
such that
$A \subseteq \bigcup_i V_i$.
It follows that
$p(V_i) \subseteq C_i$,
so by the Lusin separation theorem, there exist Borel
$p(V_i) \subseteq B_i \subseteq C_i$.
Then $V_i \subseteq p^{-1}(B_i)$, so
$A
\subseteq \bigcup_i V_i
\subseteq \bigcup_i (p^{-1}(B_i) \cap U_i)$.
\end{proof}

\begin{corollary}[Kunugi--Novikov separation]
\label{thm:kunugi-novikov-sep}
Let $p : X -> Y$ be a Borel map between standard Borel spaces, $(U_i)_{i \in I}$ be a countable family of Borel sets in $X$.
Let $A \subseteq X$ be $\*\Sigma^1_1$ and $C \subseteq X$ be $\*\Pi^1_1$, and let $(T_i)_{i \in I}$ be a family of auxiliary $\*\Pi^1_1$ ``test sets'' in $Y$.
Suppose every fiber $A_y$ of $A$ can be covered by $(U_i)_y$ for $i \in I$ such that $(U_i)_y \subseteq C$ and $y \in T_i$.
Then there exist Borel $B_i \subseteq T_i$ such that
\begin{eqaligned*}
A &\subseteq \bigcup_i (p^{-1}(B_i) \cap U_i) \subseteq C.
\end{eqaligned*}
In particular, if every fiber $A_y$ can be covered by $(U_i)_y \subseteq C$, then there exist Borel $B_i \subseteq Y$ such that the above inclusions hold.
\end{corollary}
\begin{proof}
Take $C_i := \set{y \in T_i | (U_i)_y \subseteq C}$ in \cref{thm:kunugi-novikov-refl}.
For the last statement, take $T_i := Y$.
\end{proof}

\begin{corollary}[Kunugi--Novikov uniformization]
\label{thm:kunugi-novikov}
Let $p : X -> Y$ be a Borel map between standard Borel spaces, $(U_i)_{i \in I}$ be a countable family of Borel sets in $X$.
Let $A \subseteq X$ be Borel, and let $(T_i)_{i \in I}$ be a family of auxiliary $\*\Pi^1_1$ ``test sets'' in $Y$.
Suppose every fiber $A_y$ of $A$ is a union of $(U_i)_y$ for $i \in I$ such that $y \in T_i$.
Then there exist Borel $B_i \subseteq T_i$ such that
\begin{eqaligned*}
A &= \bigcup_i (p^{-1}(B_i) \cap U_i).
\end{eqaligned*}
In particular, if $A$ is fiberwise a union of the $U_i$'s, then there are Borel $B_i \subseteq Y$ such that this holds.
\end{corollary}
\begin{proof}
Take $C := A$ above.
\end{proof}

See \cref{thm:fib-bor-pol-reg} for an example use of the ``test sets'' in these statements.

\subsection{Borel bundles of (quasi-)Polish spaces}

\begin{definition}[{see \cite[2.4.2, 2.4.5]{chen:beckec}}]
\label{def:fib-bor-qpol}
Let $p : X -> Y$ be a Borel map between standard Borel spaces, equipped with a fiberwise topology $\@O_p(X)$ on $X$.
We call this fiberwise topology a \defn{Borel fiberwise quasi-Polish topology} on $X$, and call $X$ a \defn{standard Borel bundle of quasi-Polish spaces} over $Y$, if
\begin{enumerate}[roman]
\item  the topology $\@O(X_y)$ on each fiber $X_y$ is quasi-Polish; and
\item  these topologies are ``uniformly second-countable'': there exists a countable family $\@U$ of Borel, fiberwise open sets, without loss of generality closed under finite intersections, which restricts to an open basis in each fiber (we call such $\@U$ a \defn{fiberwise open basis}).
\end{enumerate}
In that case, we write
\begin{eqaligned*}
\@{BO}_p(X) := \@B(X) \cap \@O_p(X)
\end{eqaligned*}
for the collection of Borel fiberwise open sets in $X$, which forms a sigma-topology (\cref{def:stop}).
Note that by Kunugi--Novikov \labelcref{thm:kunugi-novikov}, for any countable Borel fiberwise basis $\@U$, we then have
\begin{align}
\label{eq:fib-bor-qpol-basis}
\@{BO}_p(X) = p^{-1}(\@B(X)) \ocap \@U
= \set*{\bigcup_{U \in \@U} (p^{-1}(B_U) \cap U) | (B_U)_U \in \@B(X)^\@U}
\end{align}
(recall the meaning of $\ocap$ from \cref{def:ostar}).
More generally, by \cref{thm:kunugi-novikov-sep} we have
\begin{eqenumerate}
\item \label{it:fib-bor-qpol-sep}
For any $A \in \*\Sigma^1_1(X)$ and $C \in \*\Pi^1_1(X)$ such that $A$ is contained in the fiberwise interior of $C$, there is a $U \in \@{BO}_p(X)$ such that $A \subseteq U \subseteq C$.
\end{eqenumerate}

We call such a Borel fiberwise quasi-Polish topology on $p : X -> Y$ \defn{Borel-overt}, and call $X$ a \defn{standard Borel-overt bundle of quasi-Polish spaces} over $Y$, if moreover
\begin{eqaligned*}
p(\@{BO}_p(X)) \subseteq \@B(Y),
\end{eqaligned*}
i.e., the $p$-projection of every Borel fiberwise open set is Borel.
By Kunugi--Novikov via \cref{eq:fib-bor-qpol-basis}, it suffices to check this on any countable Borel fiberwise basis $\@U \subseteq \@{BO}_p(X)$.

We call $X$ a \defn{standard Borel(-overt) bundle of Polish spaces} if each fiber is Polish.
\end{definition}

The robustness of \cref{def:fib-bor-qpol} for quasi-Polish bundles is demonstrated by

\begin{proposition}[topological realization of Borel bundles of spaces {\cite[2.4.3, 2.4.5]{chen:beckec}}]
\label{thm:fib-bor-qpol}
\leavevmode
\begin{enumerate}[alph]

\item
For a standard Borel bundle of quasi-Polish spaces $p : X -> Y$, there exist compatible (global) quasi-Polish topologies on $X, Y$ making $p$ into a continuous map and such that $\@O(X)$ restricts fiberwise to the original fiberwise topology $\@O_p(X)$.

Moreover, $\@O(X)$ may be chosen to include any countably many sets in $\@{BO}_p(X)$.
In particular, $\@{BO}_p(X)$ is a compatible sigma-topology on $X$ (in the sense of \cref{def:stop}).

\item
If $X$ is moreover Borel-overt, then such topologies may be chosen to make $p$ open.

\end{enumerate}
\end{proposition}
\begin{proof}[Proof sketch]
Using the indicator functions of a countable Borel fiberwise basis $\@U$, we may find a Borel fiberwise continuous embedding into a product bundle $e : X -> Y \times \#S^\@U$, which has fiberwise $\*\Pi^0_2$ image by \cref{it:qpol-pi02}.
We now apply the following quasi-Polish generalization \cite[2.4.4]{chen:beckec} of Saint~Raymond's uniformization theorem \cite{saint-raymond:ksigma} for Borel fiberwise $\*\Pi^0_2$ sets (which is the analog of Kunugi--Novikov for level 2 of the Borel hierarchy; see also \cite[28.10, 35.45]{kechris:cdst}):

\begin{lemma}[Saint~Raymond uniformization for quasi-Polish bundles]
\label{thm:saint-raymond}
Let $p : X -> Y$ be a standard Borel bundle of quasi-Polish spaces, $A \subseteq X$ be Borel fiberwise $\*\Pi^0_2$.
Then for any countable Borel fiberwise basis $\@U \subseteq \@{BO}_p(X)$, there are Borel $B_{U,V} \subseteq Y$ for each $U, V \in \@U$ such that
\begin{eqaligned*}
A = \bigcap_{U \in \@U} \paren*{U => \bigcup_{V \in \@U} (p^{-1}(B_{U,V}) \cap V)}.
\end{eqaligned*}
\end{lemma}
\begin{proof}
First, consider the case where $X, Y$ are quasi-Polish, $p$ is continuous, and $\@U \subseteq \@O(X)$.
When $X = Y \times Z$ is a product bundle, where $Z$ is also quasi-Polish, and $p : Y \times Z -> Y$ is the first projection, this is \cite[2.4.4]{chen:beckec}.
For a general continuous $p : X -> Y$, apply the product bundle case to the graph of $\eval*{p}_A \subseteq X \times Y$ equipped with the second projection $\pi_2 : X \times Y -> Y$ and the fiberwise basis $\@U \times Y$; we get that $\eval*{p}_A$ may be written as
\begin{eqaligned*}
\eval*{p}_A = \bigcap_{U \in \@U} \paren*{U \times Y => \bigcup_{V \in \@U} (V \times B_{U,V})} \subseteq X \times Y,
\end{eqaligned*}
now take preimage under $(\id_X,p) : X -> X \times Y$.
Finally, the general case of a Borel bundle follows from \cref{thm:fib-bor-qpol} (whose proof only uses the product bundle case of this lemma).
\end{proof}

Applying this lemma to $e(X) \subseteq Y \times \#S^\@U$, equipped with the first projection $p : Y \times \#S^\@U -> Y$, shows that $e(X)$ is ``uniformly fiberwise $\*\Pi^0_2$''.
Now choose the topology on $Y$ to contain these $B_{U,V}$.
In the overt case, refine the topology to make the $p$-images of a countable basis in $X$ open in $Y$.
\end{proof}

We now show the analogous result in the Polish context.

\begin{lemma}
\label{thm:fib-bor-pol-reg}
Let $p : X -> Y$ be a standard Borel bundle of Polish spaces.
Then $\@{BO}_p(X)$ is a regular sigma-topology (recall again \cref{def:stop}).
\end{lemma}
\begin{proof}
Fix a countable Borel fiberwise basis $\@V \subseteq \@{BO}_p(X)$.
Let $U \in \@{BO}_p(X)$.
Then in each fiber over $y \in Y$, $U_y$ is the union of $V_y$ for all $V \in \@V$ such that $\-{V_y} \subseteq U_y$.
Note that the fiberwise closure
\begin{align}
\label{eq:fib-bor-closure}
\-V_p
:= \bigcup_{y \in Y} \-{V_y}
= \set{x \in X | \forall W \in \@V\, (x \in W \implies V_x \cap W_x \ne \emptyset)}
\end{align}
is $\*\Sigma^1_1$.
Thus by Kunugi--Novikov \labelcref{thm:kunugi-novikov} (with the ``test sets'' $T_V := \set{y \in Y | \-{V_y} \subseteq U_y}$), there are Borel $B_V \subseteq Y$ for each $V \in \@V$ such that
\begin{align*}
U &= \bigcup_{V \in \@V} (p^{-1}(B_V) \cap V), &
\-{(p^{-1}(B_V) \cap V)}_p = p^{-1}(B_V) \cap \-V_p \subseteq U.
\end{align*}
For each $V$, since the Borel set $\neg U$ is contained in the $\*\Pi^1_1$ fiberwise open $\neg \-{(p^{-1}(B_V) \cap V)}_p$, by Kunugi--Novikov again in the form of \labelcref{it:fib-bor-qpol-sep} there is a Borel fiberwise open $\neg F_V$ sandwiched in between.
Then the countable cover $U = \bigcup_{V \in \@V} (p^{-1}(B_V) \cap V)$ by Borel fiberwise open sets contained in Borel fiberwise closed sets $p^{-1}(B_V) \cap V \subseteq F_V \subseteq U$ witnesses regularity for $U$.
\end{proof}

\begin{corollary}[topological realization of Borel bundles of Polish spaces]
For a standard Borel(-overt) bundle of Polish spaces $p : X -> Y$, there exist compatible Polish topologies on $X, Y$ making $p$ into a continuous (open) map and such that $\@O(X)$ restricts fiberwise to the original fiberwise topology $\@O_p(X)$.
Moreover, $\@O(X)$ may be chosen to include any countably many sets in $\@{BO}_p(X)$.
\end{corollary}
\begin{proof}
By \cref{thm:fib-bor-qpol}, find compatible quasi-Polish topologies satisfying the required conditions.
By \cref{thm:stop-pol} and the preceding lemma, there is a finer Polish topology on $X$ still contained in $\@{BO}_p(X)$.
If necessary, also refine the topology on $Y$ to make it Polish (after first making the images of countably many basic open sets in $X$ open, in the overt case), and adjoin the preimages of the new open sets to $\@O(X)$.
\end{proof}

We have the following ways of deriving new Borel bundles of spaces from old ones:

\begin{remark}
\label{rmk:fib-bor-qpol-pb}
Given a standard Borel bundle of quasi-Polish spaces $p : X -> Y$, and a Borel map $f : Z -> Y$ from another standard Borel space $Z$, recall that we have the pullback fiberwise topology $\@O_{\pi_1}(Z \times_Y X)$ on $\pi_1 : Z \times_Y X -> Z$ (\cref{def:fib-pb}).

By definition, if $\@U \subseteq \@{BO}_p(X)$ is a countable Borel fiberwise basis, then $Z \times_Y \@U \subseteq \@{BO}_{\pi_1}(Z \times_Y X)$ is a countable Borel fiberwise basis.
Thus $\pi_1 : Z \times_Y X -> Z$ is a standard Borel bundle of quasi-Polish spaces, with (by Kunugi--Novikov \labelcref{eq:fib-bor-qpol-basis})
\begin{eqaligned*}
\@{BO}_{\pi_1}(Z \times_Y X) = \@B(Z) \otimes_Y \@U = \@B(Z) \otimes_Y \@{BO}_p(X).
\end{eqaligned*}
If $p : X -> Y$ is Borel-overt, then so is $\pi_1 : Z \times_Y X -> Z$, by Frobenius reciprocity \cref{it:fib-baire-frob}.

If in this situation $f$ is also equipped with a Borel fiberwise quasi-Polish topology, then the fiberwise product topology on $Z \times_Y X -> Y$ (see again \cref{def:fib-pb}) is also a Borel fiberwise quasi-Polish topology; a countable fiberwise basis is given by
$\@U \times \@V := \set{U \times V | U \in \@U,\, V \in \@V}$ for any countable fiberwise bases $\@U \subseteq \@{BO}_p(X)$ and $\@V \subseteq \@{BO}_f(Z)$.
If both $\@O_p(X)$ and $\@O_f(Z)$ are Borel-overt, then so is $\@O_{f\pi_1=p\pi_2}(Z \times_Y X)$.
\end{remark}

\begin{lemma}
\label{thm:fib-bor-subsp}
For a Borel fiberwise $\*\Pi^0_2$ set $X' \subseteq X$ in a standard Borel bundle of quasi-Polish spaces $p : X -> Y$, equipped with the fiberwise subspace topology, we have $\@{BO}_p(X') = X' \cap \@{BO}_p(X)$.
\end{lemma}
\begin{proof}
Given $U' \in \@{BO}_p(X')$, apply \cref{it:fib-bor-qpol-sep} to $U', X' \setminus U'$ to get $U \in \@{BO}_p(X)$ with $U' = X' \cap U$.
\end{proof}

\begin{corollary}
\label{thm:fib-bor-subsp-dense}
If $p : X -> Y$ is a standard Borel bundle of quasi-Polish spaces, and $X' \subseteq X$ is a Borel fiberwise dense $\*\Pi^0_2$ set, then $p$ is a Borel-overt bundle iff $\eval*{p}_{X'} : X' -> Y$ is.
\end{corollary}
\begin{proof}
Every Borel fiberwise open $U' \subseteq X'$ is by the preceding lemma $U' = X' \cap U$ for $U \in \@{BO}_p(X)$; and by fiberwise density, $p(U') = p(U)$.
\end{proof}

\subsection{Borel category quantifiers}

Borel-overt bundles $p : X -> Y$ are those in which the fiberwise Baire property and Baire category quantifiers $\exists^*_p, \forall^*_p$ (\cref{def:fib-top}) may be computed in a Borel way; see \cite[2.3.16, 2.4.7]{chen:beckec}.
We will recall the proof here, as we will need a bound on the Borel complexity of the witnesses.

\begin{proposition}
\label{thm:fib-baire-borel}
Let $p : X -> Y$ be a continuous open map between quasi-Polish spaces.
\begin{enumerate}[alph]
\item \label{thm:fib-baire-borel:bp}
(fiberwise Baire property)
For any $A \in \*\Sigma^0_\xi(X)$, there is a fiberwise open $U_A \subseteq X$, of the form
\tagsright
\begin{align*}
U_A = \bigcup_i (p^{-1}(B_i) \cap U_i)
\in p^{-1}(\*\Sigma^0_\xi(Y)) \ocap \@O(X)
\tag{$*$}
\end{align*}
where $B_i \in \*\Sigma^0_\xi(Y)$ and $U_i \in \@O(X)$,
such that $A =^*_f U_A$,
as witnessed by a fiberwise dense $\@G_\delta$
\begin{align*}
G_A = \bigcap_j \bigcup_k (p^{-1}(C_{j,k}) \cap V_{j,k})
\tag{$\dagger$}
\end{align*}
where $C_{j,k} \in \*\Delta^0_\xi(Y)$ and $V_{j,k} \in \@O(X)$,
on which $G_A \cap A = G_A \cap U_A$.
\item \label{thm:fib-baire-borel:exists*}
Thus, $\exists^*_f(\*\Sigma^0_\xi(X)) \subseteq \*\Sigma^0_\xi(Y)$.
\end{enumerate}
Similarly if $p : X -> Y$ is a standard Borel-overt bundle of quasi-Polish spaces, then these hold with $\@B$ in place of $\*\Sigma^0_\xi, \*\Delta^0_\xi$.
\end{proposition}
\begin{proof}
\cref{thm:fib-baire-borel:bp} follows by induction on the Borel complexity of $A$, using the fiberwise versions of \cref{eq:bp}:
\begin{eqaligned}[][\label{eq:fib-baire-borel:bp}\left\{][\right.]
U_V &:= V, &
G_V &:= X
    \quad \text{for $V \in \@O(X)$},
\\
U_{\bigcup_i A_i} &:= \bigcup_i U_{A_i}, &
G_{\bigcup_i A_i} &:= \bigcap_i G_{A_i},
\\
U_{A \setminus B} &:= \bigcup_{W \in \@W} \paren[\big]{W \cap U_A \setminus p^{-1}(p(W \cap U_B))}, &
G_{A \setminus B} &:= G_A \cap G_B \cap \paren[\big]{\-{(U_B)}_p => U_B}
\end{eqaligned}
where $\@W$ is any countable open basis for $X$.
Here $\-{(U_B)}_p$ is the fiberwise closure, as in \cref{eq:fib-bor-closure} which shows that $\-{(U_B)}_p => U_B$ (the complement of the fiberwise boundary of $U_B$) is of the form
\begin{eqaligned*}
\paren[\big]{\-{(U_B)}_p => U_B}
&= \paren*{\bigcap_{W \in \@W} \paren[\big]{W => p^{-1}(p(U_B \cap W))}} => U_B \\
&= \bigcup_{W \in \@W} \paren[\big]{W \setminus p^{-1}(p(U_B \cap W))} \cup U_B.
\end{eqaligned*}
When $B \in \*\Sigma^0_\xi(X)$, then by ($*$) (and Frobenius reciprocity \labelcref{it:fib-baire-frob}),
$p(U_B \cap W) \in \*\Sigma^0_\xi(Y)$;
this along with ($*$) verifies ($\dagger$) in the inductive case for successor ordinals.

\cref{thm:fib-baire-borel:exists*} follows from applying $\exists^*_p$ to ($*$), by Frobenius reciprocity \cref{it:fib-baire-frob}.
The formulas \cref{eq:fib-baire-borel:bp} become
\begin{eqaligned}[][\label{eq:fib-baire-borel:exists*}\left\{][\right.]
\exists^*_p(V) &= p(V)
    \quad \text{for $V \in \@O(X)$ (or more generally $V \in \@O_p(X)$)}, \\
\exists^*_p(\bigcup_i A_i) &= \bigcup_i \exists^*_p(A_i), \\
\exists^*_p(A \setminus B) &= \bigcup_{W \in \@W} (\exists^*_p(W \cap A) \setminus \exists^*_p(W \cap B)).
\end{eqaligned}

The Borel version now follows from \cref{thm:fib-bor-qpol}, or by the same proof.
We note that in \cref{eq:fib-baire-borel:bp}, hence also in \cref{eq:fib-baire-borel:exists*}, the set $\@W$ may in fact be any countable Borel \emph{fiberwise} basis for $X$.
\end{proof}

\begin{corollary}
\label{thm:qpol-openquot-bor}
Let $f : X ->> Y$ be a continuous open surjection between quasi-Polish spaces.
Then $B \subseteq Y$ is $\*\Sigma^0_\xi$ iff $f^{-1}(B) \subseteq X$ is.
\end{corollary}
\begin{proof}
$\Longleftarrow$ is clear; for $\Longrightarrow$, $B = \exists^*_f(f^{-1}(B))$.
\end{proof}

We also have the Kuratowski--Ulam theorem, which admits the following formulation in terms of bundles; see \cite[2.4.8]{chen:beckec}, \cite[7.6]{chen:qpol}, \cite[A.1]{melleray-tsankov:abelian}, \cite[8.41]{kechris:cdst}.
Since we will need a slightly more general form than what these references state, we briefly recall the proof.

\begin{theorem}[Kuratowski--Ulam]
\label{thm:kuratowski-ulam}
Let $p : X -> Z$ and $q : Y -> Z$ be two bundles of topological spaces, and $f : X -> Y$ be a fiberwise continuous open map over $Z$, so that $p = q \circ f$.
\begin{eqtikzcd*}
X \drar["p"'] \rar["f"] & Y \dar["q"] \\
& Z
\end{eqtikzcd*}
Suppose $X$ is $p$-fiberwise second-countable and $f$-fiberwise Baire.
Then for any $p$-fiberwise Baire-measurable $A \subseteq X$, we have that
\begin{enumerate}[alph]
\item
$A \cap f^{-1}(y) \subseteq f^{-1}(y)$ is Baire-measurable for $q$-fiberwise comeagerly many $y \in Y$;
\item
$\exists^*_f(A) \subseteq Y$ is $q$-fiberwise Baire-measurable;
\item
$\exists^*_p(A) = \exists^*_q(\exists^*_f(A))$.
\end{enumerate}
\end{theorem}
\begin{proof}[Proof sketch]
By restricting to the fibers over each $z \in Z$, it suffices to consider the case $Z = \set{*}$.
Since $X$ is $f$-fiberwise Baire, $\exists^*_f(W) = f(W)$ for open $W \subseteq X$.
If $F \subseteq X$ is closed nowhere dense, and $W \subseteq X$ is open, then it is easily seen that $f(W \setminus F) \subseteq f(W)$ is dense; fixing a countable basis $\@W \subseteq \@O(X)$, it follows that $\exists^*_f(F) = \bigcup_{W \in \@W} (f(W) \setminus f(W \setminus F))$ is meager.
Hence if $F \subseteq X$ is meager, then so is $\exists^*_f(F) \subseteq Y$.
Now for Baire-measurable $A \subseteq X$, find $A =^* U \in \@O(X)$, so that $A \triangle U$ is meager, hence so is $\exists^*_f(A \triangle U) \supseteq \exists^*_f(A) \triangle \exists^*(U) = \exists^*_f(A) \triangle f(U)$; all three claims easily follow.
\end{proof}

\begin{remark}
We will typically be interested in the above situation when $p : X -> Z$ and $q : Y -> Z$ are standard Borel-overt bundles of quasi-Polish spaces, and $f : X -> Y$ is Borel fiberwise continuous open over $Z$.
In such a situation, it does not seem possible to prove that $f$ is automatically ``uniformly fiberwise open'', in the sense that $f(\@{BO}_p(X)) \subseteq \@{BO}_q(Y)$.

However, we have the following ``approximate'' version of this:
\end{remark}

\begin{lemma}
\label{thm:fib-open-borel-reg}
Let $p : X -> Z$ and $q : Y -> Z$ be standard Borel-overt bundles of quasi-Polish spaces, and $f : X -> Y$ be Borel fiberwise continuous open over $Z$.
Then for any Borel $p$-fiberwise open $U \in \@{BO}_p(X)$, the $q$-fiberwise exterior of the image $Y \setminus \-{f(U)}_p \subseteq Y$ is Borel (and $q$-fiberwise open), thus so is the $q$-fiberwise regularized open image $(\-{f(U)}_p)^\circ_p$ (which is $=^*_q f(U)$).
\end{lemma}
\begin{proof}
Fixing any countable fiberwise basis $\@V \subseteq \@{BO}_q(Y)$, we have
\begin{eqaligned*}[b][\qedhere]
Y \setminus \-{f(U)}_p
&= \bigcup_{V \in \@V} \paren[\big]{V \setminus q^{-1}(p(U \cap f^{-1}(V)))}.
\end{eqaligned*}
\end{proof}

Finally, we have the following method of constructing quotient quasi-Polish topologies from \cite{chen:beckec}, generalizing Sierpiński's theorem \cref{it:qpol-openquot}:

\begin{theorem}[{\cite[2.7.5]{chen:beckec}}]
\label{thm:sierpinski-exists*}
Let $p : X ->> Y$ be a standard Borel-overt bundle of nonempty quasi-Polish spaces, and suppose given another (global) quasi-Polish topology $\@O(X)$ on $X$, whose fiberwise restriction is \emph{finer} than the fiberwise topology $\@O_p(X)$, and such that
\begin{eqaligned*}
p^{-1}(\exists^*_p(\@O(X))) \subseteq \@O(X).
\end{eqaligned*}
Then the quotient topology on $Y$ induced by $\@O(X)$ via $p$ is equal to $\exists^*_p(\@O(X))$, and is quasi-Polish.
\end{theorem}
\begin{proof}[Proof sketch]
Let $Z \subseteq X$ be the fiberwise $\@O(X)$-closed support of the $\@O_p(X)$-comeager sigma-filter, namely
$Z = \bigcap_U (U => p^{-1}(\exists^*_p(U)))$
where $U$ ranges over any countable basis for $\@O(X)$.
This is $\*\Pi^0_2$, hence quasi-Polish, and we have $\exists^*_p(U) = p(Z \cap U)$, hence the result follows from \cref{it:qpol-openquot} applied to $\eval*{p}_Z$.
(See \cite[\S2.6--7]{chen:beckec} for details.)
\end{proof}

We will also need the sigma-topological version of this:

\begin{corollary}
\label{thm:sierpinski-exists*-stop}
Let $p : X ->> Y$ be a surjective standard Borel-overt bundle of quasi-Polish spaces, and $\@S(X)$ be a compatible sigma-topology on $X$, which contains a countable Borel fiberwise basis for a finer fiberwise topology than $\@O_p(X)$, and such that
\begin{eqaligned*}
p^{-1}(\exists^*_p(\@S(X))) \subseteq \@S(X).
\end{eqaligned*}
Then $\exists^*_p(\@S(X))$ is a compatible sigma-topology on $Y$.
\end{corollary}
\begin{proof}
Given countably many $\exists^*_p(U_i) \in \exists^*_p(\@S(X))$, where each $U_i \in \@S(X)$, find a quasi-Polish topology $\@O_0(X) \subseteq \@S(X)$ containing each $U_i$ as well as a countable Borel fiberwise basis for a finer fiberwise topology than $\@O_p(X)$, and inductively find quasi-Polish topologies $\@O_n(X) \cup p^{-1}(\exists^*_p(\@O_n(X))) \subseteq \@O_{n+1}(X) \subseteq \@S(X)$; then the quasi-Polish topology $\@O(X)$ generated by $\bigcupup_n \@O_n(X)$ obeys $p^{-1}(\exists^*_p(\@O(X))) \subseteq \@O(X)$, whence by \cref{thm:sierpinski-exists*}, $\exists^*_p(\@O(X)) \subseteq \exists^*_p(\@S(X))$ is a quasi-Polish topology containing each $\exists^*_p(U_i)$.
\end{proof}

\section{Fiberwise quasi-Polish groupoids}
\label{sec:fibgpd}

\subsection{Preliminaries on groupoids and actions}

\begin{definition}
\label{def:gpd}
A \defn{groupoid} $(X,G) = (X,G,\sigma,\tau,\mu,\iota,\nu)$ consists of two sets $X$ (of \defn{objects}) and $G$ (of \defn{morphisms}) together with five structure maps:
\begin{itemize}
\item  $\sigma, \tau : G -> X$ (\defn{source} and \defn{target}); if $g \in G$ with $\sigma(g) = x$ and $\tau(g) = y$, then we also write $g : x -> y$ or $g \in G(x,y)$ where $G(x,y)$ is the \defn{hom-set} of all morphisms from $x$ to $y$;
\item
$\iota : X -> G$ (\defn{identity}), also denoted $\iota(x) =: 1_x$;
\item
$\nu : G -> G$ (\defn{inverse}), also denoted $\nu(g) =: g^{-1}$;
\item
$\mu : G \times_X G := G \tensor[_\sigma]\times{_\tau} G -> G$ (\defn{multiplication}), also denoted $\mu(g,h) =: g \cdot h =: gh$;
\end{itemize}
satisfying the usual axioms:
$\sigma(1_x) = \tau(1_x) = x$,
$\sigma(g^{-1}) = \tau(g)$,
$\tau(g^{-1}) = \sigma(g)$,
$\sigma(gh) = \sigma(h)$,
$\tau(gh) = \tau(g)$,
$(gh)k = g(hk)$,
$1_{\tau(g)}g = g = g1_{\sigma(g)}$,
$g^{-1}g = 1_{\sigma(g)}$, and
$gg^{-1} = 1_{\tau(g)}$.
See \cite{maclane:cats}.
\end{definition}

Note that for a groupoid $(X,G)$, there are two canonical maps $\sigma, \tau : G -> X$, leading to potential ambiguity in a fiber product such as $G \times_X G$ above.
To resolve this, we adopt

\begin{convention}
For a fiber product $G \times_X (-)$ or $(-) \times_X G$ over the space of objects $X$ of a groupoid, where one of the factors is the space of morphisms, we take the map $G -> X$ by default to be $\sigma$ if $G$ appears on the left, and $\tau$ if $G$ appears on the right.
\end{convention}

\begin{definition}
\label{def:gpd-ftr}
A \defn{functor} between groupoids $F : (X,G) -> (Y,H)$ is a homomorphism of groupoids regarded as 2-sorted first-order structures.
In other words, we have two maps $F : X -> Y$ and $F : G -> H$, preserving all of the groupoid operations.
\end{definition}

\begin{example}
A one-object groupoid $(*,G)$ is the same thing as a group.
A functor between two such groupoids is a group homomorphism.

In general, for any groupoid $(X,G)$, for each object $x \in X$, we have the \defn{automorphism group} (also known as \emph{isotopy group}) $\Aut_G(x) := G(x,x)$.
\end{example}

\begin{example}
A groupoid $(X,G)$ such that $\abs{G(x,y)} \le 1$ for all $x,y \in X$ (equivalently, $\Aut_G(x) = \set{1_x}$ for each $x \in X$) is just an equivalence relation $G \subseteq X^2$, up to isomorphism of groupoids, where we identify $g : x -> y \in G$ with the pair $(x,y) \in X^2$.

In general, for a groupoid $(X,G)$, the \defn{connectedness relation} $\#E_G \subseteq X^2$ is given by
\begin{eqaligned*}
x \mathrel{\#E_G} y \coloniff G(x,y) \ne \emptyset.
\end{eqaligned*}
We have a quotient functor $(X,G) ->> (X,\#E_G)$, collapsing parallel morphisms.
We write $X/G = X/\#E_G$ for the set of \defn{connected components} of $G$.
We call $G$ a \defn{connected} groupoid if $\abs{X/G} = 1$.
\end{example}

\begin{definition}
\label{def:action}
An \defn{action} of a groupoid $(X,G)$ on a bundle $p : M -> X$ is a map
\begin{eqaligned*}
\alpha : G \times_X M = G \tensor[_\sigma]\times{_p} M --> M,
\end{eqaligned*}
also denoted $\alpha(g,a) =: g \cdot a =: ga$, satisfying the axioms $p(ga) = \tau(g)$, $(gh)a = g(ha)$, and $1_{p(a)} \cdot a = a$.
Thus each $g : x -> y \in G$ acts as a bijection $M_x -> M_y$.

The \defn{action groupoid} of such an action has space of objects $M$ and space of morphisms
\begin{eqaligned*}
G \ltimes M := G \times_X M
\end{eqaligned*}
where each $(g,x) \in G \times_X M$ represents a morphism $x -> gx$ labeled by $g$, with the operations
\begin{align*}
\sigma_{G \ltimes M} &:= \pi_2 : G \times_X M -> M
    \quad \text{(the second coordinate projection)}, \\
\tau_{G \ltimes M} &:= \alpha : G \times_{X} M -> M, \\
\mu_{G \ltimes M} &:= \mu_G \times M :
    \begin{array}[t]{@{}r@{}>{{}}l<{{}}@{}c@{}>{{}}r<{{}}@{}l@{}}
    (G \times_{X} M) \tensor[_{\pi_2}]\times{_\alpha} (G \times_{X} M) &\cong& G \times_{X} G \times_{X} M &->& G \times_{X} M \\
    ((g,ha),(h,a)) &|->& (g,h,a) &|->& (gh,a),
    \end{array} \\
\iota_{G \ltimes M} &:= (\iota_G \circ p, \id_M) : \begin{zaligned}[t]
M &-> G \times_{X} M \\
a &|-> (1_{p(a)}, a),
\end{zaligned} \\
\nu_{G \ltimes M} &:= (\nu_G \circ \pi_1, \alpha) : \begin{zaligned}[t]
G \times_{X} M &-> G \times_{X} M \\
(g, a) &|-> (g^{-1}, ga).
\end{zaligned}
\end{align*}

The action groupoid is equipped with two canonical functors
\begin{eqtikzcd}[][\label{diag:action-gpd}]
& (M, G \ltimes M) \dlar["\phi"'] \drar[two heads] \\
(X, G) && (M, \#E_{G \ltimes M})
\end{eqtikzcd}
where the right leg is the quotient functor to the connectedness relation or \defn{orbit equivalence relation} of the action; and the left leg $\phi$, called the \defn{cocycle associated to the action}, takes each object $a \in M$ to its projection $p(a) \in X$ and each morphism $(g,a) \in G \ltimes M$ to its label $g$.
This $\phi$ is a \defn{discrete fibration} of groupoids, meaning that we have a bijection
\begin{eqaligned*}
(\phi, \sigma_H) : H := G \ltimes M &--> G \times_X M \\
h = (g,a) &|--> (\phi(h), \sigma_H(h)).
\end{eqaligned*}
In other words, for each $g : x -> y \in G$, and each lift of its source $a \in M_x$, there is a unique lift $h = (g,a)$ of the entire morphism with that source.
In fact, discrete fibrations are essentially the same thing as actions: given any groupoid $(M,H)$ with a discrete fibration $\phi : (M,H) -> (X,G)$,
\begin{equation}
\label{eq:action-fib}
\alpha : G \times_X M --->{(\phi,\sigma_H)^{-1}} H --->{\tau_H} M
\end{equation}
defines an action of $G$ on $M$, such that $H \cong G \ltimes M$.
\end{definition}

\begin{example}
\label{ex:action-trivial}
The \defn{trivial action} of $G$ on $\id_X : X -> X$ is given by $\alpha = \tau : G -> X$.
\end{example}

\begin{example}
\label{ex:action-transl}
The \defn{left translation action} of $G$ on $\tau : G -> X$ is given by $\mu : G \times_X G -> X$.
\end{example}

\begin{definition}
\label{def:gpd-top}
A \defn{topological groupoid} is a groupoid where the spaces of morphisms and of objects are topological spaces, and all 5 groupoid operations are continuous maps.
A \defn{(quasi-)Polish groupoid} is a topological groupoid where the spaces of morphisms and of objects are (quasi-)Polish spaces.
Similarly, a \defn{standard Borel groupoid} is a groupoid where the spaces are standard Borel and the maps are Borel.

A \defn{continuous action} of a topological groupoid $(X,G)$ on a bundle $p : M -> X$ which is a continuous map between topological spaces is an action $\alpha : G \times_X M -> M$ which is continuous (jointly in both variables); we also call such $M$ a \defn{topological $G$-space}.
Similarly, a \defn{Borel action} or \defn{Borel $G$-space} $p : M -> X$ for a standard Borel groupoid $(X,G)$ is a standard Borel bundle equipped with a Borel action $\alpha$.

An \defn{open topological $G$-space} $p : M -> X$ will mean one where $p$ is an open map.
This is equivalent to saying that $\pi_1 : G \times_X M -> G$ is open, since $\pi_1, p$ are pullbacks of each other (along $\sigma, \iota$ respectively), and a pullback of a continuous open map is open.

An \defn{open topological groupoid} $(X,G)$ is one where $\sigma : G -> X$ is open, or equivalently $\tau$ is, or equivalently $\pi_1$ or $\pi_2 : G \times_X G -> G$ is (by considering the left translation action), or equivalently $\mu : G \times_X G -> G$ is (since $\pi_2, \mu$ are the source and target maps for the action groupoid $G \ltimes G$).
\end{definition}

\subsection{Borel-overt fiberwise quasi-Polish groupoids}
\label{sec:fibqpolgpd}

The following is the central concept of this paper; we will spend the rest of this section developing their basic theory.

\begin{definition}
\label{def:fibqpolgpd}
A \defn{standard Borel-overt fiberwise quasi-Polish groupoid} $(X,G)$ is a standard Borel groupoid (\cref{def:gpd-top}), equipped with a Borel-overt $\sigma$-fiberwise quasi-Polish topology $\@O_\sigma(G)$ on $G$ (\cref{def:fib-bor-qpol}), which is fiberwise invariant under right translation by each $g : x -> y \in G$, i.e., the bijection
\begin{eqaligned*}
\sigma^{-1}(y) &--> \sigma^{-1}(x) \\
h &|--> hg
\end{eqaligned*}
is a homeomorphism between the respective fiberwise topologies.

Given such an $\@O_\sigma(G)$, we always put $\@O_\tau(G) := \@O_\sigma(G)^{-1}$, which is thus a Borel-overt $\tau$-fiberwise quasi-Polish topology invariant under left translation.

A \defn{standard Borel-overt fiberwise Polish groupoid} is such a groupoid where $\@O_\sigma(G)$ is a $\sigma$-fiberwise Polish topology, i.e., fiberwise regular (see again \cref{def:fib-bor-qpol}).

A \defn{Borel-overt classwise (quasi-)Polish equivalence relation} $E \subseteq X^2$ on a standard Borel space $X$ is a Borel equivalence relation equipped with such a fiberwise topology $\@O_\sigma(E)$ as above when $E$ is regarded as a groupoid.
(See \cref{rmk:bocqper} below for a simpler reformulation in this case, recovering \cref{intro:def:bocper} from the introduction.)
\end{definition}

\begin{remark}
\label{rmk:fibqpolgpd}
We now comment on our choice of terminology for the above notions.

First, the adjective ``overt'' is included, because the analogous concept without requiring overtness of the fiberwise topology is also meaningful and arises naturally.
For example, the kernel of a continuous map $f : X -> Y$ between Polish spaces whose image is $\*\Sigma^1_1$ but not Borel will be a Borel classwise Polish equivalence relation which is not overt.
However, in this paper we will largely confine our study to the overt case.

Some adjective such as ``fiberwise'' is also required in front of ``(quasi-)Polish'', in order to disambiguate from the global concept (\cref{def:gpd-top}).
We have chosen ``fiberwise'', rather than ``$\sigma$-fiberwise'' or ``$\tau$-fiberwise'', because the two fiberwise topologies determine each other, and will play complementary roles in developing the theory of such groupoids.
For the same reason, we do not use a term such as ``left-topological groupoid'' (as would be analogous to the term ``left-topological group'' commonly used in topological group and semigroup theory; see e.g., \cite{arhangelskii-tkachenko:topgrp}).

We will show below that starting from such a fiberwise topology on a groupoid $(X,G)$, we may define a canonical topology on the spaces of \emph{all} morphisms as well as objects in a single connected component; see \cref{def:orbtop}.
In the case of an equivalence relation $E \subseteq X^2$, this topology on an $E$-class $[x]_E$ is that induced via the canonical bijection with the $\sigma$-fiber $[x]_E \cong \{x\} \times [x]_E = \sigma^{-1}(x)$ (and the topology on the morphisms $[x]_E^2$ is just the product topology); see \cref{rmk:bocqper}.
For this reason, we use the simpler adjective ``classwise'' for an equivalence relation.

However, for a general groupoid, it turns out that the above axioms are too weak to imply that the topologies induced on each component form \emph{groupoid} topologies, i.e., the multiplication may be discontinuous within a component; see \cref{ex:fibqpolgpd-disctscocy}.
We thus say ``fiberwise quasi-Polish groupoid'' for the above notion, rather than ``componentwise quasi-Polish groupoid'' which will be reserved for strengthened axioms that do imply componentwise continuity; see \cref{def:comqpolgpd}.
\end{remark}

\subsection{Derived topologies}
\label{sec:orbtop}

Fix a standard Borel-overt fiberwise quasi-Polish groupoid $(X,G)$.
Starting from the $\sigma$-fiberwise topology $\@O_\sigma(G)$, we will define several other related canonical (sigma-)topologies on $G$ and $X$.

\begin{remark}
\label{rmk:fibqpolgpd-bor}
Following \cref{def:fib-bor-qpol}, $\@{BO}_\sigma(G)$ denotes the Borel $\sigma$-fiberwise open sets in $G$.
For such a set $U \subseteq G$, by right-invariance of $\@O_\sigma(G)$, $\mu^{-1}(U) \subseteq G \times_X G$ is $\pi_2$-fiberwise open, thus by Kunugi--Novikov \cref{eq:fib-bor-qpol-basis} applied to $\pi_1^{-1}$ of any countable Borel fiberwise basis for $\@O_\sigma(G)$,
\begin{align*}
\mu^{-1}(\@{BO}_\sigma(G)) &\subseteq \@{BO}_\sigma(G) \otimes_X \@B(G), &
\mu^{-1}(\@{BO}_\tau(G)) &\subseteq \@B(G) \otimes_X \@{BO}_\tau(G)
\end{align*}
(where as in \cref{def:ostar}, $\otimes_X$ denotes countable unions of rectangles).
It easily follows that
\begin{align*}
\@{BO}_\sigma(G) &= (\mu^{-1})^{-1}(\@{BO}_\sigma(G) \otimes_X \@B(G)), &
\@{BO}_\tau(G) &= (\mu^{-1})^{-1}(\@B(G) \otimes_X \@{BO}_\tau(G))
\end{align*}
(where $(\mu^{-1})^{-1}$ denotes preimage under $\mu^{-1}$),
i.e., $U \in \@{BO}_\sigma(G) \iff \mu^{-1}(U) \in \@{BO}_\sigma(G) \otimes_X \@B(G)$ and similarly for $\@{BO}_\tau$.
($\subseteq$ follows from the above; $\supseteq$ is by definition and the fact that $\@B(G) = (\mu^{-1})^{-1}(\@B(G \times_X G))$, by the Lusin separation theorem.)
\end{remark}

\begin{remark}
\label{rmk:homtop}
On each hom-set $G(x,y) \subseteq G$, we have \emph{a priori} two different subspace topologies, induced by the $\sigma$-fiberwise topology on $\sigma^{-1}(x)$ and $\tau$-fiberwise topology on $\tau^{-1}(y)$.

We do not know if these two topologies are the same in general, nor if they are (quasi-)Polish, i.e., if $G(x,y)$ is always $\*\Pi^0_2$ in $\sigma^{-1}(x)$ and $\tau^{-1}(y)$.
These questions turn out to be the same.

Indeed, note that any nonempty hom-set $G(x,y)$ is homeomorphic, in the $\sigma$-fiberwise topology, to $G(y,y)$, via right translation $(-)g^{-1}$ by any $g \in G(x,y)$.
Similarly, $G(x,y)$ is $\@O_\tau$-homeomorphic to $G(x,x)$.
Thus if $G(x,y)$ is Polish in $\@O_\sigma$ or $\@O_\tau$, then so is $G(y,y)$ or $G(x,x)$ respectively.

If some automorphism group $G(x,x)$ is Polish in $\@O_\sigma(\sigma^{-1}(x))$, say, then since the topology is right-translation-invariant, it follows from \cite{solecki-srivastava:lpolgrp} that $G(x,x)$ is a Polish group.
Hence, inversion is continuous, and so the two fiberwise topologies on $G(x,x)$ agree.
In fact, by the groupoid generalization of \cite{solecki-srivastava:lpolgrp} we will prove later, which works in particular for groups (see \cref{ex:qpolgrp-realiz}), it suffices here for $G(x,x) \subseteq \sigma^{-1}(x)$ to be quasi-Polish, i.e., $\*\Pi^0_2$.
This then implies that each $G(x,y)$ is $\@O_\tau$-homeomorphic to $G(x,x)$, and each $G(y,x)$ is $\@O_\sigma$-homeomorphic to $G(x,x)$.

Note also that if two automorphism groups $G(x,x), G(y,y)$ for $x, y$ in the same connected component are both Polish, then they are isomorphic Polish groups, being Borel conjugate by any $g \in G(x,y)$, hence continuously conjugate by Pettis's theorem.
It then follows that $\@O_\sigma, \@O_\tau$ agree on $G(x,y)$, being homeomorphic to $G(y,y), G(x,x)$ respectively via translation by $g$.

Thus, in each connected component, between any two objects $x, y$ whose automorphism groups have Polish ($\sigma$- or $\tau$-fiberwise) topologies, the hom-set $G(x,y)$ has a canonical Polish topology given by restricting either $\@O_\sigma$ or $\@O_\tau$, and is a homeomorphic copy of the Polish group $G(x,x)$ or $G(y,y)$.

We say that $G$ \defn{has $\*\Pi^0_2$ hom-sets} if each of its automorphism groups $G(x,x)$ is $\*\Pi^0_2$ (in either $\@O_\sigma(\sigma^{-1}(x))$ or $\@O_\tau(\tau^{-1}(x))$), hence a Polish group, hence every hom-set is $\*\Pi^0_2$ in both $\@O_\sigma$ and $\@O_\tau$.
(We will show in \cref{thm:homtop-dense} that this always holds for ``comeagerly many'' objects in each connected component, but will not use this fact in developing the basic theory in this section.)
\end{remark}

\begin{definition}
\label{def:orbtop}
The \defn{componentwise topology on objects} is the topology on $X$ given by
\begin{align*}
\@O_G(X)
:=& (\tau^{-1})^{-1}(\@O_\sigma(G)) = \set{A \subseteq X | \tau^{-1}(A) \in \@O_\sigma(G)} \\
=& (\sigma^{-1})^{-1}(\@O_\tau(G)) = \set{A \subseteq X | \sigma^{-1}(A) \in \@O_\tau(G)}.
\end{align*}
In other words, it is the quotient (strong) topology on $X$ induced by the surjection $\tau : G ->> X$ from the topology $\@O_\sigma(G)$ on $G$.
Note that it is a disjoint union of topologies on the components, i.e., it is a fiberwise topology for the quotient map $X ->> X/G$.
We thus have the usual notions of \defn{componentwise meager} or \defn{componentwise dense} set of objects $A \subseteq X$.
We write
\begin{eqaligned*}
A \subseteq^*_G B
&\coloniff  A \subseteq B \text{ mod componentwise meager}, \\
A =^*_G B
&\coloniff  A = B \text{ mod componentwise meager}.
\end{eqaligned*}

Similarly, the \defn{componentwise topology on morphisms} is defined by
\begin{align*}
\@O_G(G) :=& (\mu^{-1})^{-1}(\@O_\sigma(G) \otimes_X^\infty \@O_\tau(G)) \\
=& \set{U \subseteq G | \mu^{-1}(G) = \bigcup_i (U_i \times_X V_i),\, U_i \in \@O_\sigma(G),\, V_i \in \@O_\tau(G)}
\end{align*}
(where as in \cref{def:ostar}, $\otimes_X^\infty$ denotes arbitrary unions of rectangles),
which is the quotient topology induced by $\mu : G \times_X G ->> G$ from the fiberwise product topology on the ``middle vertex'' bundle $\sigma\pi_1 = \tau\pi_2 : G \times_X G -> X$ mapping $(g,h) |-> \sigma(g) = \tau(h)$.
\end{definition}

\begin{remark}
\label{rmk:orbtop-openquot}
We have continuous open surjections
\begin{zalign*}
\tau : && (G &, \@O_\sigma(G)) &&-->> (X,\@O_G(X)), \\
\mu : && (G \times_X G &, \@O_\sigma(G) \otimes_X^\infty \@O_\tau(G)) &&-->> (G, \@O_G(G)).
\end{zalign*}
Indeed, more is true: $(X,\@O_G(X))$ is the topological disjoint union of the components $C \in X/G$, over each of which $\tau^{-1}(C)$ splits as the disjoint union of the $\sigma$-fibers $\sigma^{-1}(x)$ for each $x \in C$, each of which is a quasi-Polish space with a continuous open surjection
\begin{align}
\label{eq:orbtop-orb}
\tau : (\sigma^{-1}(x),\@O_\sigma(\sigma^{-1}(x))) &-->> (C,\@O_G(C)),
\end{align}
since for $U \in \@O_\sigma(\sigma^{-1}(x))$, we have
$\tau^{-1}(\tau(U)) = \set{g : x -> y | \exists h : x -> y \in U} = U \cdot G$
(since any such $h$ is $g(g^{-1}h)$), which is $\sigma$-fiberwise open by right-invariance of $\@O_\sigma(G)$, hence $\tau(U) \in \@O_G(C)$.

Similarly for $U \in \@O_\sigma(G)$ and $V \in \@O_\tau(G)$, $\mu^{-1}(\mu(U \times_X V))$ is easily seen to be the saturation of $U \times_X V$ under the diagonal action of $G$ on the ``middle vertex'' bundle $G \times_X G -> X$, where $g \cdot (h,k) := (hg^{-1},gk)$; by right-invariance of $\@O_\sigma$ and left-invariance of $\@O_\tau$, $\mu^{-1}(\mu(U \times_X V)) \in \@O_\sigma(G) \otimes_X^\infty \@O_\tau(G)$, hence $\mu(U \times_X V) \in \@O_G(G)$.
As for $\@O_G(X)$, this in fact holds componentwise.
\end{remark}

\begin{remark}
\label{rmk:orbtop-t0}
We do \emph{not} know if the componentwise topologies $\@O_G(X), \@O_G(G)$ must be $T_0$.
By the preceding remark and \cref{it:qpol-openquot}, this happens iff each component is quasi-Polish, iff each fiber of the continuous open surjection \cref{eq:orbtop-orb} is $\*\Pi^0_2$, i.e., iff $G$ has $\*\Pi^0_2$ hom-sets (\cref{rmk:homtop}).

If this is the case, then by \cref{thm:qpol-openquot-bor} and Kuratowski--Ulam \labelcref{thm:kuratowski-ulam}, the componentwise Borel hierarchy as well as Baire category would be reflected along the map \cref{eq:orbtop-orb}.

Since we do not know this in general, we make the following \emph{ad hoc} definition as a workaround:
\end{remark}

\begin{definition}
\label{def:orbtop-weak}
We say that $A \subseteq X$ is \defn{weakly componentwise $\*\Sigma^0_2$} or $\*\Pi^0_2$ if $\tau^{-1}(A) \subseteq G$ is $\sigma$-fiberwise $\*\Sigma^0_2$ or $\*\Pi^0_2$, respectively.

Similarly, we say $A$ is \defn{weakly componentwise meager} if $\tau^{-1}(A) \subseteq G$ is $\sigma$-fiberwise meager.
Note that these form a sigma-ideal (namely the $\tau$-pushforward of the $\@O_\sigma$-meager sigma-ideal on $G$).
We say \defn{weakly componentwise comeager} for the dual notion; this implies componentwise dense,\break since \cref{eq:orbtop-orb} is a continuous surjection.
Imitating \cref{def:orbtop}, we write $\subseteq^{**}_G, =^{**}_G$ for containment and equality mod weakly componentwise meager; thus
\begin{eqaligned*}
A \subseteq^{**}_G B  &\coloniff  \tau^{-1}(A) \subseteq^*_\sigma \tau^{-1}(B)  \iff  \sigma^{-1}(A) \subseteq^*_\tau \sigma^{-1}(B).
\end{eqaligned*}

Clearly, being componentwise $\*\Sigma^0_2$, i.e., $\*\Sigma^0_2$ with respect to the componentwise topology $\@O_G(X)$, implies being weakly componentwise $\*\Sigma^0_2$, by continuity of \cref{eq:orbtop-orb}; and componentwise meager implies weakly componentwise meager, since \cref{eq:orbtop-orb} is category-preserving.
The weak and ordinary notions agree if $G$ has $\*\Pi^0_2$ hom-sets, by \cref{thm:qpol-openquot-bor} and Kuratowski--Ulam \labelcref{thm:kuratowski-ulam}.
\end{definition}

\begin{remark}
\label{rmk:bocqper}
When $G = E$ is a Borel-overt classwise quasi-Polish equivalence relation $E \subseteq X^2$, we also call $\@O_E(X), \@O_E(E)$ defined as in \labelcref{def:orbtop} the \defn{classwise topologies}.

In this case, $\tau, \mu$ from \cref{rmk:orbtop-openquot} become topological coverings, with \cref{eq:orbtop-orb} a homeomorphism
\begin{eqaligned*}
\tau : \sigma^{-1}(x) = \set{x} \times [x]_E \cong ([x]_E, \@O_E([x]_E)).
\end{eqaligned*}
It follows that $\@O_E(X)$ determines the $\sigma$-fiberwise topology $\@O_\sigma(E)$ in this case: namely, $U \subseteq E$ is in $\@O_\sigma(E)$ iff for every $x \in X$, the fiber $U_y = \set{y | (x,y) \in U}$ is in $\@O_E(X)$.
In other words,
\begin{eqenumerate}
\item \label{it:bocqper-fibtop}
$\@O_\sigma(E)$ is the subspace topology on $E \subseteq X^2$ induced by the product topology of the discrete topology $\@P(X)$ and the classwise topology $\@O_E(X)$; similarly for $\@O_\tau(E)$.
\end{eqenumerate}
We may thus equivalently define a \defn{Borel-overt classwise quasi-Polish equivalence relation} to be a Borel equivalence relation equipped with a classwise quasi-Polish topology $\@O_E(X)$, such that the corresponding $\sigma$-fiberwise topology $\@O_\sigma(E)$ derived via \cref{it:bocqper-fibtop} is Borel-overt fiberwise quasi-Polish.
This recovers the definition given in the introduction (\cref{intro:def:bocper}).

As for $\@O_E(E)$: we have that $U \subseteq E$ is in $\@O_E(E)$ iff
$\mu^{-1}(U) = \set{((y,z),(x,y)) \in E^2 | (x,z) \in U}$
is in $\@O_\sigma(E) \otimes_X^\infty \@O_\tau(E)$, which means by \cref{it:bocqper-fibtop} that for each $y \in X$, $U \cap [y]_E^2 \in \@O_E([y]_E)^2$.
So
\begin{eqenumerate}
\item \label{it:bocqper-orbtop-mor}
$\@O_E(E)$ is the subspace topology on $E \subseteq X^2$ induced by the product of $\@O_E(X)$ with itself.
\end{eqenumerate}

Note that clearly, a Borel-overt classwise quasi-Polish equivalence relation $E$ has $\*\Pi^0_2$ hom-sets; thus the preceding remark and definition are irrelevant in this case.
\end{remark}

\begin{remark}
\label{rmk:orbtop-bor}
Returning to the groupoid case, as in \cref{rmk:fibqpolgpd-bor}, it follows from \cref{def:orbtop} and Kunugi--Novikov that the Borel sets that are open in the componentwise topologies are given by
\begin{align}
\label{eq:orbtop-obj-bor}
\@{BO}_G(X) &= (\tau^{-1})^{-1}(\@{BO}_\sigma(G)) = (\sigma^{-1})^{-1}(\@{BO}_\tau(G)), \\
\label{eq:orbtop-mor-bor}
\@{BO}_G(G) &= (\mu^{-1})^{-1}(\@{BO}_\sigma(G) \otimes_X \@{BO}_\tau(G)).
\end{align}
In particular, in conjunction with \cref{rmk:fibqpolgpd-bor}, we get
\begin{align}\label{eq:orbtop-fibtop}
\@{BO}_G(G) \subseteq \@{BO}_\sigma(G) \cap \@{BO}_\tau(G).
\end{align}
We also clearly have
\begin{align}\label{eq:orbtop-inv}
\@{BO}_G(G)^{-1} = \@{BO}_G(G).
\end{align}
Moreover, the following laws hold:
\begin{align}\label{eq:orbtop-st}
\sigma^{-1}(\@{BO}_G(X)) &\subseteq \@{BO}_G(G), &
\tau^{-1}(\@{BO}_G(X)) &\subseteq \@{BO}_G(G).
\end{align}
(By \cref{eq:orbtop-obj-bor,eq:orbtop-fibtop}, this means $\@{BO}_G(X) = (\sigma^{-1})^{-1}(\@{BO}_G(G)) = (\tau^{-1})^{-1}(\@{BO}_G(G))$.)
\begin{align}\label{eq:orbtop-id}
\iota^{-1}(\@{BO}_G(G)) &\subseteq \@{BO}_G(X).
\end{align}
\end{remark}

\begin{proof}[Proof of \cref{eq:orbtop-st}]
By \cref{eq:orbtop-mor-bor} and
\begin{eqaligned*}[b][\qedhere]
\mu^{-1}(\sigma^{-1}(\@{BO}_G(X)))
&= G \otimes_X \sigma^{-1}(\@{BO}_G(X))
    &&\text{since $\sigma \circ \mu = \sigma \circ \pi_2$} \\
&\subseteq G \otimes_X \@{BO}_\tau(G)
    &&\text{by definition of $\@O_G(X)$} \\
&\subseteq \@{BO}_\sigma(G) \otimes_X \@{BO}_\tau(G).
\end{eqaligned*}
\end{proof}

\begin{proof}[Proof of \cref{eq:orbtop-id}]
By \cref{eq:orbtop-obj-bor} and
\begin{eqaligned*}[b][\qedhere]
\tau^{-1}(\iota^{-1}(\@{BO}_G(G)))
&= (\id_G, \nu)^{-1}(\mu^{-1}(\@{BO}_G(G)))
    &&\text{since $1_{\tau(g)} = gg^{-1}$} \\
&\subseteq (\id_G, \nu)^{-1}(\@{BO}_\sigma(G) \otimes_X \@{BO}_\tau(G))
    &&\text{by \cref{eq:orbtop-mor-bor}} \\
&= \@{BO}_\sigma(G) \ocap \nu^{-1}(\@{BO}_\tau(G))
= \@{BO}_\sigma(G).
\end{eqaligned*}
\end{proof}

\begin{example}
\label{ex:qpolgpd-orbtop}
Every open (quasi-)Polish groupoid $(X,G)$ induces an underlying Borel-overt fiberwise (quasi-)Polish groupoid, by forgetting the global topology and remembering only the Borel structure and the fiberwise restriction of the topology.
By Kunugi--Novikov \cref{eq:fib-bor-qpol-basis},
\begin{align*}
\@{BO}_\sigma(G) &= \sigma^{-1}(\@B(X)) \ocap \@O(G), &
\@{BO}_\tau(G) &= \tau^{-1}(\@B(X)) \ocap \@O(G).
\end{align*}
We will show in \cref{thm:ucomqpolgpd-action-orbtop} that also in this case,
\begin{align*}
\@{BO}_G(G) &= \sigma^{-1}(\@{BO}_G(X)) \ocap \@O(G) = \tau^{-1}(\@{BO}_G(X)) \ocap \@O(G).
\end{align*}
By the Becker--Kechris theorem \cite[4.3.12]{chen:beckec}, $\@{BO}_G(X)$ consists of all sets $A \subseteq X$ that can be made open in some finer quasi-Polish topology on $X$ for which the trivial action $G \actson X$ (\cref{ex:action-trivial}) remains continuous, whence the action groupoid $G \ltimes X$ yields a finer groupoid topology on $G$; the union of all such finer topologies on $G$ is then $\@{BO}_G(G)$.

Our definition of the componentwise topologies $\@{BO}_G(X), \@{BO}_G(G)$ in a general fiberwise quasi-Polish groupoid is motivated by this special case.
Indeed, our main results below (under additional assumptions; see \cref{thm:ucomqpolgpd-qpol}) will show that $\@{BO}_G(X), \@{BO}_G(G)$ consist in general of precisely all open sets in \emph{some} open quasi-Polish groupoid topology inducing the fiberwise topology.
\end{example}

We introduce one more sigma-topology, which is a common refinement of $\@{BO}_\sigma, \@{BO}_\tau$ in contrast to $\@{BO}_G$ which is a common coarsening.
Unlike the others, this sigma-topology only makes sense in the Borel context, and does not consist of the Borel sets in a fiberwise topology.
(Its closure under arbitrary unions is the hom-set-wise topology given by copies of the unique Polish group topology in each component as in \cref{rmk:homtop}, by Pettis's \cref{intro:thm:pettis}.)

\begin{definition}
\label{def:boxtop}
Put
\begin{eqaligned*}
\@{BO}_\Box(G) := (\mu^{-1})^{-1}(\@B(G) \otimes_X \@B(G)).
\end{eqaligned*}
That is, $U \in \@B(G)$ is in $\@{BO}_\Box(G)$ iff $\mu^{-1}(U) \subseteq G \times_X G$ is a countable union of Borel rectangles $V_i \times_X W_i$, where $V_i, W_i \in \@B(G)$.
By \cref{rmk:fibqpolgpd-bor},
\begin{align*}
\@{BO}_\sigma(G) &\subseteq \@{BO}_\Box(G), &
\@{BO}_\tau(G) &\subseteq \@{BO}_\Box(G).
\end{align*}
\end{definition}

\Cref{fig:topologies} from the introduction summarizes the Borel sigma-topologies we have defined, starting from a Borel-overt fiberwise quasi-Polish groupoid $(X,G)$.
(See \cref{thm:orbtop-exists*} below for the behavior of the Baire category quantifiers $\exists^*$.)

\subsection{Topologies on actions}
\label{sec:actions}

In this subsection, we specialize the topologies on a groupoid introduced above to the case of an action groupoid $G \ltimes M$ over an existing fiberwise topological groupoid $G$.
We will then discuss one of the main motivating examples for all of these notions: Borel actions of Polish groups.

Fix as before a standard Borel-overt fiberwise quasi-Polish groupoid $(X,G)$.

\begin{definition}
\label{def:fibqpolgpd-action}
Let $p : M -> X$ be a standard Borel $G$-space (\cref{def:gpd-top}), with the action map denoted $\alpha : G \times_X M -> M$.
Consider the action groupoid $G \ltimes M = G \times_X M$ (\cref{def:action}); recall that its source map is the projection $\sigma_{G \ltimes M} = \pi_2 : G \times_X M -> M$, while its target map is the action $\tau_{G \ltimes M} = \alpha$.

The \defn{lifted fiberwise topology} on $G \ltimes M$ will mean the $\sigma$-fiberwise topology $\@O_\sigma(G \ltimes M)$ given by pulling back (\cref{def:fib-pb}) the $\sigma$-fiberwise topology $\@O_\sigma(G)$ on $\sigma_G : G -> X$ along $p : M -> X$.
In other words, the cocycle associated to the action \cref{diag:action-gpd} becomes a $\sigma$-fiberwise homeomorphism
\begin{eqtikzcd*}
(G \ltimes M, \@O_\sigma(G \ltimes M)) \dar["\sigma_{G \ltimes M}"'] \rar["\phi"] &
(G, \@O_\sigma(G)) \dar["\sigma_G"]
\\
M \rar["p"] &
X
\end{eqtikzcd*}
By \cref{rmk:fib-bor-qpol-pb}, $\@O_\sigma(G \ltimes M)$ is a Borel-overt $\sigma_{G \ltimes M}$-fiberwise quasi-Polish topology, with
\begin{align*}
\@{BO}_\sigma(G \ltimes M) = \@{BO}_\sigma(G) \otimes_X \@B(M).
\end{align*}
It is also easily seen to be right-translation-invariant.
Thus, $G \ltimes M$ becomes a Borel-overt fiberwise quasi-Polish groupoid in its own right.

The \defn{$\alpha$-fiberwise topology} on $G \ltimes M = G \times_X M$ will always refer to the corresponding $\tau$-fiberwise topology $\@O_\tau(G \ltimes M) = \@O_\sigma(G \ltimes M)^{-1}$, as in \cite[4.1.8]{chen:beckec} (the ``twist involution'' $\dagger$ from \cite[4.1.5]{chen:beckec} being the inversion map $\nu$ of $G \ltimes M$).
We call the componentwise topologies on $G \ltimes M$ (\cref{def:orbtop}) the \defn{orbitwise topologies}, denoted via the shorthands
\begin{eqaligned*}
\@O_G(M) &:= \@O_{G \ltimes M}(M), \\
\@O_G(G \ltimes M) &:= \@O_{G \ltimes M}(G \ltimes M).
\end{eqaligned*}
We similarly abbreviate ``componentwise meager'' to ``\defn{orbitwise meager}'', and write $\subseteq^*_G, =^*_G$ instead of $\subseteq^*_{G \ltimes M}, =^*_{G \ltimes M}$; similarly for the ``weak'' notions introduced in \cref{def:orbtop-weak}.
In this notation, \cref{eq:orbtop-obj-bor} and \labelcref{def:orbtop-weak} say, for $A \subseteq M$:
\begin{zalign}
\label{eq:orbtop-action}
A \in \@{BO}_G(M)  &\iff  \alpha^{-1}(A) \in \@{BO}_\sigma(G) \otimes_X \@B(M)  &&\iff  G \times_X A \in \@{BO}_\alpha(G \times_X M), \hspace{-1em} \\
\label{eq:orbtop-meager-action}
A \text{ weakly orbwise mgr}  &\iff  \alpha^{-1}(A) \text{ $\pi_2$-fiberwise meager}  &&\iff  G \times_X A \text{ $\alpha$-fiberwise mgr}.\hspace{-2em}
\end{zalign}
\end{definition}

\begin{remark}
Under additional assumptions on $G$, we moreover have
\begin{align*}
\@{BO}_G(G \ltimes M) = \@{BO}_G(G) \otimes_X \@{BO}_G(M);
\end{align*}
see \cref{thm:ucomqpolgpd-action-orbtop}.
In particular, this holds when the fiberwise topology on $G$ is the restriction of an open quasi-Polish groupoid topology (as in the following example).
\end{remark}

\begin{example}
\label{ex:polgrp-action}
Let $G$ be a Polish group, regarded as a one-object open Polish groupoid.
Let $M$ be a standard Borel $G$-space.
Then the lifted $\sigma$-fiberwise topology on $G \ltimes M = G \times M$ defined above is the $\sigma = \pi_2$-fiberwise topology consisting of a copy of the topology of $G$ on each $\sigma$-fiber
\begin{align*}
\sigma^{-1}(a) = \pi_2^{-1}(a) = G \times \set{a}.
\end{align*}
Per \cref{rmk:orbtop-openquot}, the orbitwise topology $\@O_G(M)$ is the disjoint union of the topologies on each orbit $G \cdot a$ induced by the continuous open surjection
\begin{eqaligned*}
\alpha : \sigma^{-1}(a) = G \times \set{a} -->> G \cdot a.
\end{eqaligned*}
By the preceding remark, the orbitwise topology on morphisms $\@O_G(G \ltimes M)$ is the product topology on $G \times M$ where $M$ has the orbitwise topology $\@O_G(M)$.
This may be seen directly again using \cref{rmk:orbtop-openquot}, according to which $\@O_G(G \ltimes X)$ is the disjoint union of the topologies on the subspace of morphisms $G \times Ga$ within each orbit $Ga$ induced by the continuous open surjection
\begin{eqaligned*}
\mu_{G \ltimes M} : \sigma^{-1}(a) \times \tau^{-1}(a) = (G \times \set{a}) \times (G \times \set{a})^{-1} -->> G \times Ga.
\end{eqaligned*}
A basic open set on the left is $(U \times \set{a}) \times (V \times \{a\})^{-1}$ for $U, V \in \@O(G)$, whose image under $\mu_{G \ltimes M}$ is
$\set{(gh^{-1},ha) | g \in U,\, h \in V}$.
Using continuity of multiplication in $G$, such a set is easily seen to be a union of sets of the form
$U'V'^{-1} \times V'a \in \@O(G) \otimes \@O_G([a]_G)$,
for $U' \subseteq U$ and $V' \subseteq V$.
\end{example}

\begin{example}
\label{ex:polgrp-orbit}
Consider again a Borel action of a Polish group $G \actson M$.
Now instead of taking the action groupoid, assume the orbit equivalence relation $E := \#E_{G \ltimes M} \subseteq M^2$ is Borel.
We have a classwise topology $\@O_E(M)$ as above, where each $[a]_E = Ga$ is equipped with the quotient topology via the action $G ->> Ga$; thus for each $U \in \@O(G)$, we have an open set $Ua \subseteq Ga$.
Via \cref{rmk:bocqper}, the corresponding $\sigma$-fiberwise topology on $E$ is generated by the fiberwise open sets
\begin{eqaligned*}
\~U := \bigcup_{a \in M} (\set{a} \times Ua) = \set{(a,b) \in E | \exists g \in U\, (ga = b)}.
\end{eqaligned*}
By Becker--Kechris \cite[7.1.2]{becker-kechris:polgrp}, this set is Borel provided $E$ is; and clearly $\sigma(\~U) = M$ for $U \ne \emptyset$.
We thus have a Borel-overt classwise Polish topology on $E$, with a countable fiberwise basis given by $\~U$ for any countably many basic $U \in \@O(G)$.

(The same construction works for orbit equivalence relations of open Polish groupoid actions, using \cite[5.2.2]{lupini:polgpd}.
For the generalization to open quasi-Polish groupoids, see \cref{thm:ucomqpolgpd-erquot}.)
\end{example}

\subsection{Vaught transforms}
\label{sec:vaught}

In this subsection, we again fix a standard Borel-overt fiberwise quasi-Polish groupoid $(X,G)$.
We now consider the natural generalization, to fiberwise quasi-Polish groupoids, of the following standard tool in the study of Polish group(oid) actions (see \cite[\S3.2]{gao:idst}, \cite[\S2.10]{lupini:polgpd}).

\begin{definition}
\label{def:vaught}
Let $p : M -> X$ be a standard Borel $G$-space, with the action map denoted $\alpha : G \times_X M -> M$, and consider the action groupoid $G \ltimes M = G \times_X M$ with the lifted fiberwise topology (\cref{def:fibqpolgpd-action}).
For $U \subseteq G$ and $A \subseteq M$, define the \defn{Vaught transform}
\begin{eqaligned*}
U * A
:=& \exists^*_{\tau_{G \ltimes M}}(U \times_X A) = \exists^*_\alpha(U \times_X A) \\
=& \set{b \in M | \exists^* g \in \tau^{-1}(p(b))\, (g \in U \AND b \in gA)}.
\end{eqaligned*}
(This binary $*$ notation is from \cite{chen:beckec}; the original notation from \cite{vaught:invariant} is $A^{\triangle U^{-1}}$.)

Following \cref{def:ostar}, for collections of sets $\@U \subseteq \@B(G)$ and $\@A \subseteq \@B(M)$, we write $\@U \oast \@A$ for the collection of all countable unions $\bigcup_i (U_i * A_i)$, where $U_i \in \@U$ and $A_i \in \@A$.
\end{definition}

\begin{remark}
\label{thm:vaught-props}
The Vaught transform satisfies the following algebraic laws.
These are copied from \cite[\S4.2]{chen:beckec}, with the privoso that not all laws from there carry over, since that paper assumed an open quasi-Polish groupoid $G$, whereas here we only have a fiberwise quasi-Polish groupoid.

Using \cref{eq:fib-baire-borel:exists*}, we have for any $U \subseteq G$ and $A \subseteq M$ that
\begin{eqenumerate}

\item \label{it:vaught-im}
$U * A \subseteq U \cdot A$, with equality if $U$ is $\tau$-fiberwise open and $A$ is orbitwise open;

\end{eqenumerate}
for countably many $U_i \subseteq G$ and $A_j \subseteq M$,
\begin{align}
\label{it:vaught-union}
(\bigcup_i U_i) * (\bigcup_j A_j) &= \bigcup_{i,j} (U_i * A_j);
\end{align}
and for $U \subseteq G$, for $A, B \subseteq M$, and for any countable fiberwise basis $\@W \subseteq \@O_\tau(G)$,
\begin{align}
\label{it:vaught-diff}
U * (A \setminus B) &= \bigcup_{W \in \@W} \paren[\big]{((W \cap U) * A) \setminus ((W \cap U) * B)}.
\end{align}
We have the following consequences of \defn{Frobenius reciprocity} \cref{it:fib-baire-frob}: for any $B \subseteq X$, $U \subseteq G$, $A \subseteq M$, and $G$-invariant $C \subseteq M$,
\begin{align}
\label{it:vaught-frob}
(\tau^{-1}(B) \cap U) * A &= p^{-1}(B) \cap (U * A), &
U * (A \cap C) &= (U * A) \cap C.
\end{align}
The proofs of these are straightforward and the same as in \cite[\S4.2]{chen:beckec}.

We call the following laws \defn{Pettis's theorem}.
(Recall \cref{def:orbtop-weak} of ``weakly'' below.)
\begin{eqenumerate}

\item \label{it:vaught-meager}
If $U \subseteq G$ is $\tau$-fiberwise meager, or $A \subseteq M$ is weakly orbitwise meager, then
$U * A = \emptyset$.

\item \label{it:vaught-pettis}
If $U \subseteq^*_\tau V \subseteq G$ and $A \subseteq^{**}_G B \subseteq M$, then $U * A \subseteq V * B$.

\end{eqenumerate}
By the $\tau$-fiberwise Baire property for $G$ (\cref{thm:fib-baire-borel}), and the weak orbitwise Baire property for $M$ (see \cref{thm:orbtop-bp} below), it follows that for any $U \subseteq G$ and $A \subseteq M$,
\begin{align}
\label{it:vaught-bp}
\@B(G) * A &= \@{BO}_\tau(G) * A, &
U * \@B(M) &= U * \@{BO}_G(M).
\end{align}
In particular, $\@B(G) * \@B(M) = \@{BO}_\tau(G) * \@{BO}_G(M)$, i.e., any Vaught transform $U * A$ is equal to one where $U$ is $\tau$-fiberwise open and $A$ is orbitwise open, which is just $U \cdot A$ by \cref{it:vaught-im}.
\end{remark}

\begin{proof}[Proof of \cref{it:vaught-meager,it:vaught-pettis}]
If $U \subseteq G$ is $\tau_G$-fiberwise meager, then $(U \times_X M)^{-1} = U^{-1} \times_X M$ is $\sigma_{G \ltimes M} = \pi_2$-fiberwise meager, hence $U \times_X M \subseteq G \ltimes M$ is $\tau_{G \ltimes M}$-fiberwise meager, hence $U * A \subseteq U * M = \emptyset$.
If $A \subseteq M$ is weakly orbitwise meager, then $G \times_X A$ is $\tau_{G \ltimes M} = \alpha$-fiberwise meager by \cref{eq:orbtop-meager-action}.
To get \cref{it:vaught-pettis}, apply \cref{it:vaught-union} to $(V \cup (U \setminus V)) * (B \cup (A \setminus B))$.
\end{proof}

The following was essentially part of \cite[4.3.7]{chen:beckec} (for open quasi-Polish $G$):

\begin{lemma}
\label{thm:vaught-fibact}
Let $p : M -> X$ be a Borel bundle of quasi-Polish spaces, equipped with a Borel action of $G$ such that each $g : x -> y \in G$ acts via a homeomorphism $p^{-1}(x) -> p^{-1}(y)$.
Then
\begin{eqaligned*}
\@B(G) \oast \@{BO}_p(M) &\subseteq \@{BO}_p(M).
\end{eqaligned*}
\end{lemma}
\begin{proof}
By fiberwise continuity and Kunugi--Novikov, we have (recall $\nu_{G \ltimes M}$ from \cref{def:action})
\begin{eqaligned*}
\nu_{G \ltimes M}(\@B(G) \otimes_X \@{BO}_p(M))
&= (\nu_G \pi_1, \alpha)^{-1}(\@B(G) \otimes_X \@{BO}_p(M)) \\
&\subseteq \@B(G) \otimes_X \@{BO}_p(M),
\end{eqaligned*}
thus
\begin{eqaligned*}[b][\qedhere]
\@B(G) \oast \@{BO}_p(M)
&= \exists^*_{\tau_{G \ltimes M}}(\@B(G) \otimes_X \@{BO}_p(M)) \\
&= \exists^*_{\sigma_{G \ltimes M}}(\nu_{G \ltimes M}(\@B(G) \otimes_X \@{BO}_p(M))) \\
&\subseteq \exists^*_{\pi_2}(\@B(G) \otimes_X \@{BO}_p(M)) \\
&= \@{BO}_p(M)
    \quad \text{by Frobenius \cref{it:fib-baire-frob}}.
\end{eqaligned*}
\end{proof}

\begin{example}
\label{ex:vaught-mul}
We are particularly interested in the Vaught transform for the left translation action $\alpha = \mu : G \times_X G -> G$ (\cref{ex:action-transl}).
In that case, we have fiberwise homeomorphisms
\begin{eqtikzcd}[column sep=6em][\label{diag:gpd-2simp-edge}]
(G \times_X G, \@O_{\pi_1})
    \drar["\pi_1"']
    &
(G \times_X G, \@O_\mu)
    \dar["\mu"]
    \lar[<->, "{(g,h) |-> (gh,h^{-1})}"']
    \rar[<->, "{(g,h) |-> (g^{-1},gh)}"]
    &
(G \times_X G, \@O_{\pi_2})
    \dlar["\pi_2"]
\\
& G
\end{eqtikzcd}
where the right homeomorphism is by \cref{def:fibqpolgpd-action} of $\@O_\mu(G \times_X G) = \@O_\alpha(G \times_X G)$, and the left homeomorphism is because its composite with the right homeomorphism is the fiberwise homeomorphism $(g,h) |-> (h^{-1}g^{-1},g)$ from $\@O_{\pi_1}(G \times_X G)$ to $\@O_{\pi_2}(G \times_X G)$ (because $h |-> h^{-1}g^{-1}$ is a homeomorphism from $\@O_\tau(G)$ to $\@O_\sigma(G)$).
Thus the \defn{$\mu$-fiberwise topology} on $G \times_X G$ is symmetric, in that it could have equivalently been defined via the \emph{right} translation action $G \actedby G$.

It follows that the definition of Vaught transform $U * V$ for $U, V \subseteq G$ is symmetric as well:
\begin{align}
\label{eq:vaught-inv}
U * V = (V^{-1} * U^{-1})^{-1}.
\end{align}
It is also easy to see that the orbitwise topology on objects \cref{eq:orbtop-action} for the left action $G \actson G$ coincides with the $\sigma$-fiberwise topology $\@O_\sigma(G)$.
(Note that $\@O_\sigma(G)$ is $T_0$, obviating \cref{rmk:orbtop-t0}; hence the ``weak'' notions in \cref{def:orbtop-weak} become redundant.)
Thus, the laws listed in \cref{thm:vaught-props} form dual pairs, swapping the roles of $\tau, \sigma$.
We also get \cref{thm:vaught-fibact} and its dual:
\begin{align}
\label{eq:vaught-fibtop}
\@B(G) \oast \@{BO}_\tau(G) &\subseteq \@{BO}_\tau(G), &
\@{BO}_\sigma(G) \oast \@B(G) &\subseteq \@{BO}_\sigma(G).
\end{align}
\end{example}

\begin{definition}
Consider now the ``vertex projections'' $\tau\pi_1$, $\sigma\pi_1 = \tau\pi_2$, and $\sigma\pi_2 : G \times_X G -> X$.
These are related via fiberwise bijections (cf.\ \labelcref{diag:gpd-2simp-edge}; here we are projecting to $X$ instead of $G$):
\begin{eqtikzcd}[column sep=6em][\label{diag:gpd-2simp-vertex}]
(G \times_X G, \@O_{\tau\pi_1})
    \drar["\tau\pi_1"']
    &
(G \times_X G, \@O_{\sigma\pi_1}{=}\@O_{\tau\pi_2})
    \dar["\sigma\pi_1=\tau\pi_2\mathstrut"{anchor=center,fill=white,inner sep=0pt}]
    \lar[<->, "{(g,h) |-> (g^{-1},gh)}"']
    \rar[<->, "{(g,h) |-> (gh,h^{-1})}"]
    &
(G \times_X G, \@O_{\sigma\pi_2})
    \dlar["\sigma\pi_2"]
\\
& X
\end{eqtikzcd}
By the \defn{$\sigma\pi_1 = \tau\pi_2$-fiberwise topology} on $G \times_X G$, we will mean the fiberwise product topology of $\@O_\sigma(G)$ and $\@O_\tau(G)$, which is a Borel-overt fiberwise quasi-Polish topology on the bundle
$\sigma\pi_1 = \tau\pi_2 : G \times_X G -> X$
(by \cref{rmk:fib-bor-qpol-pb}).
We then transfer this fiberwise topology via the above fiberwise bijections to the \defn{$\tau\pi_1$-fiberwise} and \defn{$\sigma\pi_2$-fiberwise topologies} on $G \times_X G$.

By the \defn{Kuratowski--Ulam theorem} \labelcref{thm:kuratowski-ulam}, for Borel $W \subseteq G \times_X G$ we have
\begin{align}
\label{eq:vaught-ku-middle}
\exists^*_\sigma(\exists^*_{\pi_1}(W)) = \exists^*_{\sigma\pi_1}(W) = \exists^*_{\tau\pi_2}(W) = \exists^*_\tau(\exists^*_{\pi_2}(W)).
\end{align}
Transferring this from the middle bundle in \cref{diag:gpd-2simp-vertex} to the two side bundles yields
\begin{align}
\label{eq:vaught-ku}
\exists^*_\tau(\exists^*_{\pi_1}(W)) &= \exists^*_{\tau\pi_1}(W) = \exists^*_\tau(\exists^*_\mu(W)), &
\exists^*_\sigma(\exists^*_\mu(W)) &= \exists^*_{\sigma\pi_2}(W) = \exists^*_\sigma(\exists^*_{\pi_2}(W)).
\end{align}
In particular, for $W = U \times_X V$ where $U, V \subseteq G$, we have (by Frobenius reciprocity \labelcref{it:fib-baire-frob})
\begin{align}
\label{eq:vaught-fib}
U * \exists^*_\tau(V) = \exists^*_\tau(U \cap \sigma^{-1}(\exists^*_\tau(V))) &= \exists^*_\tau(U * V), &
\exists^*_\sigma(U * V) &= \exists^*_\sigma(\tau^{-1}(\exists^*_\sigma(U)) \cap V)
\end{align}
where here $U * \exists^*_\tau(V)$ refers to the Vaught transform for the trivial action $G \actson X$ (\cref{ex:action-trivial}).

More generally, we have a Kuratowski--Ulam associativity law for Vaught transforms for arbitrary Borel actions $G \actson M$ as in \cite[4.2.20]{chen:beckec}: taking $G$ above to be $G \ltimes M$ and $W$ in \cref{eq:vaught-ku} to be $U \times_X V \times_X A$ for $U, V \subseteq G$ and $A \subseteq M$ yields
\begin{align}
U * (V * A) = \exists^*_\alpha(U \times_X (V * A)) = \exists^*_\alpha((U * V) \times_X A) = (U * V) * A.
\end{align}
In particular this holds for the left translation action $G \actson G$ and Borel $U, V, A \subseteq G$.
\end{definition}

\begin{remark}
The above fiberwise topologies may be generalized by considering the simplicial nerve $(G_n)_{n \in \#N}$ of $G$; see \cite[\S1.4]{cisinski:hoalg}.
This is the (symmetric) simplicial set with
\begin{itemize}
\item  vertices $G_0 := X$;
\item  edges $G_1 := G$;
\item  triangles $G_2 := G \times_X G =$ composable pairs $(g,h) \cong$ commuting triangles \smash{\begin{tikzcd}[row sep=1em, column sep=1em, every cell/.append style={inner sep=1pt}]
    \cdot && \cdot \ar[ll,"gh"'] \dlar["h"] \\
    & \cdot \ular["g"]
\end{tikzcd}};
\item  3-cells $G_3 := G \times_X G \times_X G =$ composable triples $\cong$ commuting tetrahedra;
\end{itemize}
etc.
For each face map $\partial_s : G_n ->> G_m$, $m \le n$, corresponding to an injection between standard simplices $s : \{0,\dotsc,m\} `-> \{0,\dotsc,n\}$, we get a canonical $\partial_s$-fiberwise topology on $G_n$, given by identifying an $n$-simplex with an $m$-simplex together with $(n-m)$-many morphisms with common source (or target), and then taking the $(n-m)$-fold fiberwise product topology of $\@O_\sigma$ (or $\@O_\tau$).
\end{remark}

\subsection{Baire category for derived topologies}

We now apply Vaught transforms to further study the various topologies on a Borel-overt fiberwise quasi-Polish groupoid $(X,G)$ introduced in \cref{sec:fibqpolgpd,sec:orbtop}.

\begin{lemma}
\label{thm:orbtop-exists*}
The componentwise topologies (\cref{def:orbtop}) may be obtained via
\begin{align*}
\@{BO}_G(X) &= \exists^*_\tau(\@{BO}_\sigma(G)) = \exists^*_\sigma(\@{BO}_\tau(G)), \\
\@{BO}_G(G) &= \exists^*_\mu(\@{BO}_\sigma(G) \otimes_X \@{BO}_\tau(G)) = \@{BO}_\sigma(G) \oast \@{BO}_\tau(G). \\
\intertext{Thus, for a Borel $G$-space $p : M -> X$, the orbitwise topology (\cref{def:fibqpolgpd-action}) may be obtained via}
\@{BO}_G(M) &= \exists^*_\alpha(\@{BO}_\sigma(G \ltimes M)) = \@{BO}_\sigma(G) \oast \@B(M).
\end{align*}
\end{lemma}
(See \cref{fig:topologies}.
Note that this is the Baire-categorical analog of \cref{rmk:orbtop-openquot}.)
\begin{proof}
$\subseteq$ follows from \cref{eq:orbtop-obj-bor,eq:orbtop-mor-bor} and the fact that $\tau, \sigma, \mu$ are surjective and fiberwise quasi-Polish, hence fiberwise nonmeager (see \labelcref{it:fib-baire-surj}).

For $\supseteq$, imitating the argument in \cref{rmk:orbtop-openquot}, we have for any $U \in \@{BO}_\sigma(G)$ that
\begin{eqaligned}
\label{eq:fibtop-sat}
\tau^{-1}(\exists^*_\tau(U))
&= \set{g \in G | \exists^* h \in \tau^{-1}(\tau(g))\, (h \in U)}
= U * G
\in \@{BO}_\sigma(G)
\end{eqaligned}
by \cref{eq:vaught-fibtop}, whence $\exists^*_\tau(U) \in \@{BO}_G(X)$ by \cref{eq:orbtop-obj-bor}.
And for any $U \in \@{BO}_\sigma(G)$ and $V \in \@{BO}_\tau(G)$,
\begin{eqaligned*}
\mu^{-1}(U * V)
&= \set*{(h,k) \in G \times_X G | \exists^* h' \in \tau^{-1}(\tau(h))\, (h' \in U \AND hk \in h'V)} \\
&= \set*{(h,k) \in G \times_X G | \exists^* g \in \tau^{-1}(\sigma(h))\, (h \in Ug^{-1} \AND k \in gV)}
\end{eqaligned*}
using the $\@O_\tau(G)$-homeomorphism $h^{-1} \cdot (-) : \tau^{-1}(\tau(h)) \cong \tau^{-1}(\sigma(h))$ taking $h'$ to $g$; but this last set is $G * (U \times_G V)$ for the fiberwise continuous action $g \cdot (h,k) := (hg^{-1},gk)$ on the ``middle vertex'' bundle $\sigma\pi_1 = \tau\pi_2 : G \times_X G -> X$ equipped with the fiberwise product topology of $\@O_\sigma(G)$ and $\@O_\tau(G)$, and so again by \cref{thm:vaught-fibact} we have $\mu^{-1}(U * V) \in \@{BO}_\sigma(G) \times_X \@{BO}_\tau(G)$ whence $U * V \in \@{BO}_G(G)$ by \cref{eq:orbtop-mor-bor}.
\end{proof}

\begin{remark}
\label{rmk:boxtop-vaught}
We have the following analogous, but weaker, Baire-categorical description of the sigma-topology $\@{BO}_\Box(G)$ introduced in \cref{def:boxtop}:
\begin{eqaligned*}
\@{BO}_\Box(G) \subseteq \@B(G) \oast \@B(G) = \@{BO}_\tau(G) \oast \@{BO}_\sigma(G).
\end{eqaligned*}
The $\subseteq$ follows as above from surjectivity of $\mu$ \cref{it:fib-baire-surj}; the equality follows from Pettis's theorem \cref{it:vaught-bp}.
(The $\subseteq$ becomes equality under additional assumptions; see \cref{thm:ucomqpolgpd}.)
\end{remark}

We will also need the following analog of \cref{thm:orbtop-exists*} for the second level of the Borel hierarchy.
We recall from \cref{def:orbtop-weak} that the ``weakly'' in the following results can be dropped if $G$ has $\*\Pi^0_2$ hom-sets, e.g., if $G$ is an equivalence relation or an open quasi-Polish groupoid.

\begin{lemma}
For any Borel $\sigma$-fiberwise $\*\Sigma^0_2$ set $F \subseteq G$, $\exists^*_\tau(F) \subseteq X$ is weakly componentwise $\*\Sigma^0_2$.
\end{lemma}
\begin{proof}
By Saint Raymond's uniformization theorem (\cref{thm:saint-raymond}), it suffices to assume $F = U \setminus V$ for $U, V \in \@{BO}_\sigma(G)$.
As in \cref{eq:fibtop-sat}, we have
\begin{eqaligned*}
\tau^{-1}(\exists^*_\tau(U \setminus V))
&= (U \setminus V) * G
= \bigcup_{W \in \@W} ((U * W) \setminus (V * W))
\end{eqaligned*}
for any countable Borel $\sigma$-fiberwise basis $\@W \subseteq \@{BO}_\sigma(G)$, by the dual of \cref{it:vaught-diff}.
This set is $\sigma$-fiberwise $\*\Sigma^0_2$ by \cref{eq:vaught-fibtop}, thus $\exists^*_\tau(U \setminus V)$ is weakly componentwise $\*\Sigma^0_2$ by \cref{def:orbtop-weak}.
\end{proof}

The following is the groupoid analog of \cite[2.2(i)]{solecki-srivastava:lpolgrp}:

\begin{lemma}
\label{thm:fibtop-meager-exists*}
If $U \subseteq G$ is Borel $\sigma$-fiberwise meager, then $\exists^*_\tau(U) \subseteq X$ is weakly componentwise meager.
Dually, if $U$ is Borel $\tau$-fiberwise meager, then $\exists^*_\sigma(U)$ is weakly componentwise meager.
\end{lemma}
\begin{proof}
Suppose $U$ is $\sigma$-fiberwise meager.
Then
\begin{eqaligned*}
\exists^*_\sigma(\tau^{-1}(\exists^*_\tau(U)))
&= \exists^*_\sigma(U * G)
    &&\text{by \cref{eq:fibtop-sat}} \\
&= \exists^*_\sigma(\tau^{-1}(\exists^*_\sigma(U)))
    &&\text{by Kuratowski--Ulam \cref{eq:vaught-fib}} \\
&= \emptyset
    &&\text{since $U$ is $\sigma$-fiberwise meager},
\end{eqaligned*}
i.e., $\tau^{-1}(\exists^*_\tau(U))$ is $\sigma$-fiberwise meager,
fulfilling \cref{def:orbtop-weak}.
\end{proof}

\begin{corollary}
\label{thm:orbtop-meager-bp}
Any Borel weakly componentwise meager $A \subseteq X$ is contained in a Borel weakly componentwise meager $\*\Sigma^0_2$ set.
\end{corollary}
\begin{proof}
Since $A$ is weakly componentwise meager, $\tau^{-1}(A) \subseteq G$ is $\sigma$-fiberwise meager, hence contained in a Borel $\sigma$-fiberwise meager $\@F_\sigma$ set $F \subseteq G$ by \cref{thm:fib-baire-borel}; then $\exists^*_\tau(F)$ works by the preceding two lemmas.
\end{proof}

\begin{corollary}[weak componentwise Baire property]
\label{thm:orbtop-bp}
For any Borel $A \subseteq X$, there is a Borel weakly componentwise comeager $\*\Pi^0_2$ $Y \subseteq X$ and Borel componentwise open $A'$ with $Y \cap A = Y \cap A'$.
\end{corollary}
Note that this holds in particular for an action groupoid $G \ltimes M$, yielding a \emph{weak orbitwise Baire property} for Borel $G$-spaces $M$.
This can be seen as a version of Hjorth's ``orbit continuity lemma'' \cite[3.17]{hjorth:book} (the name was coined in \cite[2.5]{lupini-panagiotopoulos:games}).
\begin{proof}
By the $\sigma$-fiberwise Baire property (\cref{thm:fib-baire-borel}), find $\tau^{-1}(A) =^*_\sigma U \in \@{BO}_\sigma(G)$, i.e., $\tau^{-1}(A) \triangle U$ is $\sigma$-fiberwise meager; then $A = \exists^*_\tau(\tau^{-1}(A)) \subseteq \exists^*_\tau(U) \cup \exists^*_\tau(\tau^{-1}(A) \triangle U)$, and similarly, $\exists^*_\tau(U) \subseteq A \cup \exists^*_\tau(\tau^{-1}(A) \triangle U)$, so $A = A'$ on the weakly componentwise comeager $X \setminus \exists^*_\tau(\tau^{-1}(A) \triangle U)$, which contains a weakly componentwise comeager $\*\Pi^0_2$ set by the preceding corollary.
\end{proof}

\subsection{Idealisticity}

\begin{definition}[{see \cite[5.4.9]{gao:idst}}]
\label{def:er-idl}
An equivalence relation $E \subseteq X^2$ on a standard Borel space with Borel equivalence classes is \defn{idealistic} if there exists a family $(\@I_C)_{C \in X/E}$ of sigma-ideals of Borel sets $\@I_C \subseteq \@B(C)$ on each equivalence class $C \in X/E$, such that defining the \defn{$\@I$-quantifier}
\begin{align*}
\exists^\@I_\sigma : \@B(E) &--> \@P(X) \\
U &|--> \set{x \in X | U_x \notin \@I_{[x]_E}},
\end{align*}
we have that
\begin{enumerate}[roman]
\item  \label{def:er-idl:borel}
$\exists^\@I_\sigma$ maps Borel sets to Borel sets (i.e., $\exists^\@I_\sigma : \@B(E) -> \@B(X)$);
\item  \label{def:er-idl:retr}
each $\@I_C$ is a proper ideal (i.e., $\exists^\@I_\sigma(E) = X$, or equivalently $\exists^\@I_\sigma(\sigma^{-1}(A)) = A$ for all $A \in \@B(X)$).
\end{enumerate}
\end{definition}

\begin{remark}
The $\@I$-quantifier $\exists^\@I_\sigma$ contains essentially the same information as the family of sigma-ideals $(\@I_C)_C$: conversely, given any countable union-preserving map $e : \@B(E) -> \@B(X)$ which obeys the Frobenius reciprocity condition \cref{it:fib-baire-frob} with respect to $\sigma$ and is ``right-translation-invariant'' (i.e., agrees on $\sigma$-fibers in the same $E$-class, in the sense that $x \in e(U) \iff y \in e(V)$ whenever $U, V \in \@B(E)$ and $x \mathrel{E} y$ with $U_x = V_y$), then the family of sigma-ideals
\begin{align*}
\@I_{[x]_E} := \set{D \subseteq [x]_E | x \notin e(\set{x} \times D)}
\end{align*}
recovers $e = \exists^\@I_\sigma$.
(See \cite[\S2.6]{chen:beckec} for a general discussion of this correspondence.)
\end{remark}

The following important application of idealisticity is known as (one version of) the ``large section uniformization theorem''; see \cite[18.6$^*$]{kechris:cdst-corrections}, \cite[5.4.11]{gao:idst}.

\begin{theorem}[Kechris]
\label{thm:largeunif}
Let $E \subseteq X^2$ be an idealistic equivalence relation on a standard Borel space $X$ which is \defn{smooth}, i.e., is the kernel $E = \ker(f) = \set{(x,x') | f(x) = f(x')}$ of a Borel map $f : X -> Y$ to a standard Borel space $Y$.
Then $X/E \cong f(X) \subseteq Y$ is standard Borel, and there exists a Borel section $g : f(X) `-> X$ of $f$, i.e., a Borel map such that $f \circ g = \id_{f(X)}$.
\end{theorem}

Clearly, every Borel-overt classwise quasi-Polish equivalence relation $E \subseteq X^2$ is idealistic, with $\exists^\@I_\sigma := \exists^*_\sigma$ which corresponds to the meager ideal $\@I_C$ in the classwise topology on each $C \in X/E$.
More generally, we have

\begin{proposition}
\label{thm:fibqpolgpd-er-idl}
The connectedness relation $\#E_G$ of every Borel-overt fiberwise quasi-Polish groupoid $(X,G)$ is idealistic.
\end{proposition}
\begin{proof}
Let $\@I_C$ be the weakly componentwise meager ideal (\cref{def:orbtop-weak}) on each $C \in X/G$.
Then $U_x \subseteq [x]_G$ is in $\@I_{[x]_G}$ iff $\tau^{-1}(U_x) \subseteq \sigma^{-1}(x)$ is $\sigma$-fiberwise meager, so that for Borel $U \subseteq \#E_G$,
\begin{equation*}
\exists^\@I_{\sigma_{\#E_G}}(U) = \exists^*_{\sigma_G}\paren[\big]{((\sigma_G, \tau_G) : G ->> \#E_G)^{-1}(U)}.
\end{equation*}

In order to verify that each component $C \subseteq X$ is indeed Borel, we use a result from later: by \cref{thm:homtop-dense}, $C$ contains an $x$ whose automorphism group $G(x,x)$ is Polish; then $C = \tau(\sigma^{-1}(x))$ is the quotient of the standard Borel space $\sigma^{-1}(x)$ by the right translation action of $G(x,x)$, hence Borel since Polish group actions induce idealistic equivalence relations (see \cite[5.4.10]{gao:idst}).
\end{proof}

\begin{remark}
There are several variations of the definition of idealisticity in the literature.
The original definition from \cite{kechris-louveau:hypersmooth} requires, in addition to \labelcref{def:er-idl}\cref{def:er-idl:borel,def:er-idl:retr} above, that each of the sigma-ideals $\@I_C$ obey the countable chain condition (ccc); this strengthening was called \emph{ccc-idealistic} in \cite{derancourt-miller:ctblunion}.
The latter paper also defined another strengthening, \emph{strongly idealistic}, to mean that
(in an equivalent formulation) condition \labelcref{def:er-idl}\cref{def:er-idl:borel} above also holds for every pullback of the bundle $\sigma : E -> X$ along an arbitrary Borel map $Z -> X$ (equipped with the obvious pullback sigma-ideals on each fiber).
The above proof easily adapts to show that $\#E_G$ for a Borel-overt fiberwise quasi-Polish groupoid $G$ is strongly ccc-idealistic.
\end{remark}

\begin{corollary}
\label{thm:fibqpolgpd-idl}
If the connectedness relation of a Borel-overt fiberwise quasi-Polish groupoid, or more generally the orbit equivalence relation of a Borel action of such a groupoid, is smooth, then its quotient is standard Borel and admits a Borel selector of a unique point in each class.
\qed
\end{corollary}

\begin{remark}
\label{rmk:idlgpd}
It seems natural to define more generally a \defn{Borel-idealistic groupoid} $(X,G)$ to mean a standard Borel groupoid equipped with a right-translation-invariant family $(\@I_x)_{x \in X}$ of proper sigma-ideals of Borel sets $\@I_x \subseteq \sigma^{-1}(x)$, such that the corresponding quantifier $\exists^\@I_\sigma : \@B(G) -> \@P(X)$ preserves Borel sets.
As above, it follows that $\#E_G \subseteq X^2$ is idealistic.

As noted in the introduction (\cref{sec:intro-misc}), a natural extension to the family of conjectures surrounding Polish group actions and idealistic equivalence relations would be that every (strongly/ccc) Borel-idealistic groupoid is Borel equivalent to a Polish group action.
Equivalently by \cref{thm:fibqpolgpd-polact}, it would suffice to have an equivalence to a Borel-overt fiberwise quasi-Polish groupoid.
\end{remark}

\subsection{Equivalences of groupoids}
\label{sec:gpd-equiv}

In this subsection, we consider how a Borel-overt fiberwise quasi-Polish topology interacts with various kinds of ``injective'' and ``surjective'' functors.

\begin{definition}[see \cite{maclane:cats}]
\label{def:ftr-equiv}
Let $F : (X,G) -> (Y,H)$ be a functor between groupoids.
\begin{itemize}

\item
We will call $F$ \defn{injective}, \defn{surjective}, or \defn{bijective} respectively if it is so on objects.

\item
$F$ is \defn{faithful} if for each $x_1,x_2 \in X$, $F$ restricts to an injection $G(x_1,x_2) `-> H(F(x_1),F(x_2))$.

\item
$F$ is \defn{full} if for each $x_1,x_2 \in X$, $F$ restricts to a surjection $G(x_1,x_2) ->> H(F(x_1),F(x_2))$.

\item
$(X,G)$ is a \defn{subgroupoid} of $(Y,H)$ if $X \subseteq Y$, $G \subseteq H$, and the groupoid operations are the same.
The inclusion of a subgroupoid is a faithful injective functor; conversely, every faithful injective functor is an isomorphism onto its image which is a subgroupoid of the codomain.

\item
A \defn{full subgroupoid} of $(Y,H)$ is a subgroupoid for which the inclusion is full, i.e., of the form $(X, \eval{H}_X)$ for some $X \subseteq Y$, where $\eval{H}_X := \sigma_H^{-1}(X) \cap \tau_H^{-1}(X)$.
(In other words, the underlying directed multigraph is the induced subgraph on a subset of objects.)

The \defn{full image} of $F$ is the full subgroupoid on $F(X) \subseteq Y$.

\item
$F$ is \defn{essentially surjective} if the image saturation $[F(X)]_H \subseteq Y$ is all of $Y$, i.e., $F$ descends to a surjection on the quotient set of connected components $X/G ->> Y/H$.

The full subgroupoid on $[F(X)]_H \subseteq Y$ is called the \defn{essential image} (or \emph{replete image}) of $F$.

\item
$F$ is an \defn{equivalence of groupoids} if it is full, faithful, and essentially surjective, or equivalently it has an \defn{inverse-up-to-isomorphism} $F^{-1} : H -> G$, i.e., a functor such that $F \circ F^{-1}$ and $F^{-1} \circ F$ admit natural transformations to the identity.

\item
Recall from \cref{def:action} that $F$ is a \defn{discrete fibration} if for any morphism in $H$ and any lift of its source in $X$, there is a unique lift of the entire morphism with that source.

If only the existence of the lift holds, but not necessarily uniqueness, then $F$ is a \defn{fibration}.

\end{itemize}
\end{definition}

\begin{remark}
The above notions have well-known specializations in the case of a homomorphism between equivalence relations (especially in the invariant descriptive set-theoretic context):
a \emph{full} functor is a \emph{reduction} between equivalence relations;
an \emph{equivalence of groupoids} is a \emph{bireduction};
a \emph{fibration} is a \emph{class-surjective homomorphism}; and
a \emph{discrete fibration} is a \emph{class-bijective homomorphism}.
See \cite[\S2.3]{chen-kechris:structurable}.
\end{remark}

We now consider these notions in the Borel context.
The following is essentially \cite[2.8]{chen:polgpdrep}, and generalizes \cite[5.2.3]{gao:idst} (when $G$ is an action groupoid and $H$ is an equivalence relation):

\begin{proposition}
\label{thm:fibqpolgpd-equiv}
Let $F : (X,G) -> (Y,H)$ be a Borel full functor from a Borel-overt fiberwise quasi-Polish groupoid to an arbitrary standard Borel groupoid.
Then the image saturation $[F(X)]_H \subseteq Y$ is Borel, and there is a Borel map $Y \ni y |-> (h_y,x_y) \in H \times_Y X$ taking each object $y \in Y$ to a preimage-up-to-isomorphism $x_y \in X$ with $h_y : F(x_y) -> y \in H$.

If moreover $F$ is an equivalence of groupoids, then $F$ has a Borel inverse-up-to-isomorphism, as witnessed by Borel natural isomorphisms to the identity.
\end{proposition}
It follows that the phrase ``Borel equivalence of groupoids'' has an unambiguous meaning when used between Borel-overt fiberwise quasi-Polish groupoids.
\begin{proof}
Let $G$ act on the bundle $\pi_2 : H \times_Y X -> X$ via
\begin{equation*}
(g : x -> x') \cdot (h, x) := (hF(g)^{-1}, x').
\end{equation*}
The orbit equivalence relation of this action is precisely the kernel of $\tau_H \pi_1 : H \times_Y X -> Y$: indeed, it is clearly contained in said kernel; conversely, if $(h,x), (h',x') \in H \times_Y X$ with $\tau(h) = \tau(h')$, then by fullness of $F$ there exists $g : x -> x' \in G$ with $F(g) = h'^{-1}h : F(x) -> F(x')$; then $g \cdot (h,x) = (h',x')$.
Hence by \cref{thm:fibqpolgpd-idl}, the desired Borel section $y |-> (h_y,x_y)$ of $\tau_H \pi_1$ exists.

If $F$ is an equivalence, then the standard formula for an inverse-up-to-isomorphism (see e.g., \cite[IV~\S4~Theorem~1]{maclane:cats}) is easily Borel, using the above selector $y |-> (h_y,x_y)$: namely $F^{-1}(y) := x_y$ for $y \in Y$, and for a morphism $h : y -> y' \in H$, $F^{-1}(h) := (\eval{F}_{G(x_y,x_{y'})})^{-1}(h_{y'}^{-1}hh_y)$.
The natural isomorphisms are $y |-> h_y : FF^{-1} \cong \id_H$ and $x |-> (\eval{F}_{G(x_{F(x)},x)})^{-1}(h_{F(x)}) : F^{-1}F \cong \id_G$.
\end{proof}

\begin{remark}
\label{rmk:hograph}
The above construction may be understood groupoid-theoretically as follows: the space $H \times_Y X$, when equipped with the pullback groupoid of $G$ along $\pi_2^2 : (H \times_Y X)^2 -> X^2$, is the \emph{homotopy graph} (also known as \emph{comma category} $(F \mathbin\down H)$ \cite[II~\S6]{maclane:cats}, \emph{pathspace fibration}, or \emph{mapping cocylinder} \cite[\S7.3]{may:algtop}) of the functor $F$, and yields a canonical factorization
\begin{equation*}
(X,G) \overset{\simeq}{`-->} H \times_Y X -->> (Y,H)
\end{equation*}
of $F$ into an injective equivalence of groupoids with a canonical retraction ($\pi_2$), followed by a fibration ($\tau_H \pi_1$).
Finding a section-up-to-isomorphism of $F$ amounts to finding a section-on-the-nose of $\tau_H \pi_1$.
The above idealistic action groupoid of $G \actson H \times_Y X$ is the kernel of the functor $\tau_H \pi_1$, which is a normal subgroupoid whose quotient recovers $(Y,H)$ (see \cite{brown-hardy:topgpd}).
\end{remark}

One way to get an equivalent groupoid is to pass to the full subgroupoid on a complete section of objects (i.e., a subset of objects meeting each connected component).
For fiberwise quasi-Polish groupoids, this may be done as follows.

\begin{remark}
\label{rmk:fibqpolgpd-full}
If $(X,G)$ is a Borel-overt fiberwise quasi-Polish groupoid, and $Y \subseteq X$ is Borel and weakly componentwise $\*\Pi^0_2$ (\cref{def:orbtop-weak}), then $\tau^{-1}(Y) \subseteq G$ is $\sigma$-fiberwise $\*\Pi^0_2$.
So we get an induced fiberwise quasi-Polish topology on $\eval{G}_Y$, which however may not be Borel-overt in general.

If $Y \subseteq X$ is moreover componentwise dense, then $\eval{G}_Y \subseteq G$ is $\sigma$-fiberwise $\*\Pi^0_2$ dense over $Y$, by \cref{rmk:orbtop-openquot}.
Hence in this case, the full subgroupoid $(Y, \eval{G}_Y)$ with the fiberwise subspace topology becomes a Borel-overt fiberwise quasi-Polish groupoid in its own right, by \cref{thm:fib-bor-subsp-dense}.

Assume given such Borel weakly componentwise $\*\Pi^0_2$ dense $Y \subseteq X$.
By \cref{thm:fib-bor-subsp},
\begin{align*}
\@{BO}_\sigma(\eval{G}_Y) &= \eval*{\@{BO}_\sigma(G)}_Y, &
\@{BO}_\tau(\eval{G}_Y) &= \eval*{\@{BO}_\tau(G)}_Y.
\end{align*}
Clearly, for any $U \subseteq G$, we have by Baire category
\begin{align*}
\exists^*_\sigma(\eval*{U}_Y) &= \exists^*_\sigma(U) \cap Y, &
\exists^*_\tau(\eval*{U}_Y) &= \exists^*_\tau(U) \cap Y.
\end{align*}
Thus by \cref{thm:orbtop-exists*},
\begin{eqaligned*}
\@{BO}_{\eval{G}_Y}(Y)
&= \exists^*_\tau(\@{BO}_\sigma(\eval{G}_Y))
= \exists^*_\tau(\eval*{\@{BO}_\sigma(G)}_Y)
= \exists^*_\tau(\@{BO}_\sigma(G)) \cap Y
= \@{BO}_G(X) \cap Y.
\end{eqaligned*}
Also in the diagram \cref{diag:gpd-2simp-edge} used to define Vaught transforms in $G$, $\eval{G}_Y \times_Y \eval{G}_Y \subseteq G \times_X G$ is $\pi_2$-fiberwise dense over $\eval{G}_Y \subseteq G$, hence also $\mu$-fiberwise dense; thus for Borel $U, V \subseteq G$,
\begin{align*}
\eval*{U}_Y *_{\eval{G}_Y} \eval*{V}_Y = \eval*{(U *_G V)}_Y,
\end{align*}
and so by \cref{thm:orbtop-exists*} we similarly have
\begin{eqaligned*}
\@{BO}_{\eval{G}_Y}(\eval{G}_Y) = \eval*{\@{BO}_G(G)}_Y.
\end{eqaligned*}
In short, all of the important topological data on $\eval{G}_Y$ is just the restriction of that on $G$.
\end{remark}

Under stronger hypotheses, we will be able to say more about general fiberwise quasi-Polish (not necessarily fiberwise dense or full) subgroupoids, as well as ``quotient groupoids''; see \cref{sec:misc}.

\section{Componentwise quasi-Polish groupoids}
\label{sec:comgpd}

In this section, we put the theory of Borel-overt fiberwise quasi-Polish groupoids developed above to use, by proving our main representation theorems and several other related results.
To do so, we must first introduce the following strengthened notion.

\subsection{(Uniform) continuity of differences}
\label{sec:comqpolgpd}

\begin{definition}
\label{def:comqpolgpd}
Let $(X,G)$ be a Borel-overt fiberwise quasi-Polish groupoid.
We say $G$ has \defn{fiberwise continuous differences}, or that it is a \defn{Borel-overt \emph{componentwise} quasi-Polish groupoid}, if the following equivalent conditions hold:
\begin{enumerate}[roman]
\item  \label{def:comqpolgpd:diff1}
For any $x \in X$, the following map is continuous with respect to the $\sigma$-fiberwise topology:
\begin{eqaligned*}
\sigma^{-1}(x) \tensor[_\tau]\times{_\tau} \sigma^{-1}(x) &--> G(x,x) \subseteq \sigma^{-1}(x) \\
(g,h) &|--> g^{-1} h.
\end{eqaligned*}
\item  \label{def:comqpolgpd:diff2}
For any $x, y \in X$, the following map is continuous with respect to the $\sigma$-fiberwise topologies:
\begin{eqaligned*}
\sigma^{-1}(y) \tensor[_\tau]\times{_\tau} \sigma^{-1}(x) &--> G(x,y) \subseteq \sigma^{-1}(x) \\
(g,h) &|--> g^{-1} h.
\end{eqaligned*}
\item  \label{def:comqpolgpd:mul}
For any $x, y \in X$, the multiplication map
\begin{eqaligned*}
\mu : \tau^{-1}(y) \times_X \sigma^{-1}(x) &--> G(x,y) \subseteq \sigma^{-1}(x)
\end{eqaligned*}
is continuous (with respect to the product of the $\tau$- and $\sigma$-fiberwise topologies).
\end{enumerate}
(Recall our notation for fiber products from \cref{def:fib-pb}.)
Conditions \cref{def:comqpolgpd:diff1,def:comqpolgpd:diff2} are related by a homeomorphism $(-)k : \sigma^{-1}(y) \cong \sigma^{-1}(x)$, for any $k \in G(x,y)$ (if $G(x,y) \ne \emptyset$; otherwise \cref{def:comqpolgpd:diff2} is trivial),
while \cref{def:comqpolgpd:diff2,def:comqpolgpd:mul} are related by a homeomorphism $\nu : \sigma^{-1}(y) \cong \tau^{-1}(y)$.
\end{definition}

Clearly, every (globally) open quasi-Polish groupoid has fiberwise continuous differences.
In other words, given a Borel-overt fiberwise quasi-Polish groupoid $(X,G)$, in order for there to exist a compatible global open quasi-Polish groupoid topology restricting to the given fiberwise topology, the fiberwise topology must necessarily have fiberwise continuous differences.
In \cref{ex:fibqpolgpd-disctscocy}, we will construct a fiberwise Polish groupoid \emph{without} fiberwise continuous differences, which thus does not admit a compatible global quasi-Polish topology.
This is purely a phenomenon of topological groupoids, with no descriptive set-theoretic content: the groupoid in question is countable, hence Borelness is not an issue.

We will show in \cref{thm:comqpolgpd-qpol-connected} that conversely, every Borel-overt componentwise quasi-Polish groupoid, \emph{on each component}, admits a compatible open quasi-Polish groupoid topology; hence the name.
However, we do not know if this can always be done uniformly over all components.
This motivates the following ``uniform'' version of the above definition.
In contrast to \cref{ex:fibqpolgpd-disctscocy}, here we have no example separating the two versions; by \cref{thm:comqpolgpd-qpol-connected}, any such example must necessarily be due to definability issues, rather than purely topological groupoid-theoretic.

\begin{definition}
\label{def:ucomqpolgpd}
Let $(X,G)$ be a Borel-overt componentwise quasi-Polish groupoid.
Note that \labelcref{def:comqpolgpd}\cref{def:comqpolgpd:mul} can be restated as
\begin{eqaligned*}
\mu^{-1}(\@O_\sigma(G)) \subseteq \@O_\tau(G) \otimes^\infty_X \@O_\sigma(G).
\end{eqaligned*}
Indeed, the right-hand side (recall \labelcref{def:ostar}) is the disjoint union topology on $\bigsqcup_{x,y \in X} (\tau^{-1}(y) \times_X \sigma^{-1}(x))$.

We say that $G$ has \defn{\emph{uniformly} fiberwise continuous differences}, or that it is a \defn{Borel-overt \emph{uniformly} componentwise quasi-Polish groupoid}, if in fact,
\begin{align}
\tag{$*$}
\label{def:ucomqpolgpd:mu-sigma}
\mu^{-1}(\@{BO}_\sigma(G)) \subseteq \@{BO}_\tau(G) \otimes_X \@{BO}_\sigma(G).
\end{align}
When $G \subseteq X^2$ is an equivalence relation, we call it a \defn{Borel-overt uniformly classwise quasi-Polish equivalence relation}.
(Note that the non-uniform condition is trivial in this case.)
\end{definition}

We now have the core result of the paper:

\begin{theorem}
\label{thm:ucomqpolgpd}
Let $(X,G)$ be a Borel-overt uniformly componentwise quasi-Polish groupoid.
Then
\begin{enumeqno}[label=(\roman*), series=thm:ucomqpolgpd]
\begin{gather}
\label{thm:ucomqpolgpd:mu-fib}
\begin{aligned}
\mu^{-1}(\@{BO}_\sigma(G)) &\subseteq \@{BO}_G(G) \otimes_X \@{BO}_\sigma(G), &\quad
\mu^{-1}(\@{BO}_\tau(G)) &\subseteq \@{BO}_\tau(G) \otimes_X \@{BO}_G(G),
\end{aligned} \\
\label{thm:ucomqpolgpd:mu-orb}
\mu^{-1}(\@{BO}_G(G)) \subseteq \@{BO}_G(G) \otimes_X \@{BO}_G(G).
\end{gather}
Hence, $(X,G)$ becomes an open sigma-topological groupoid when equipped with the componentwise sigma-topologies $\@{BO}_G(X), \@{BO}_G(G)$.
Moreover, these sigma-topologies are compatible with the Borel structure; and $\@{BO}_G(G)$ induces the other topologies on $G$:
\begin{gather}
\label{thm:ucomqpolgpd:fibtop}
\begin{aligned}
\@{BO}_\sigma(G) &= \sigma^{-1}(\@B(X)) \ocap \@{BO}_G(G), &\qquad
\@{BO}_\tau(G) &= \tau^{-1}(\@B(X)) \ocap \@{BO}_G(G),
\end{aligned} \\
\label{thm:ucomqpolgpd:boxtop}
\begin{aligned}[t]
\@{BO}_\Box(G)
&= \tau^{-1}(\@B(X)) \ocap \@{BO}_\sigma(G)
= \sigma^{-1}(\@B(X)) \ocap \@{BO}_\tau(G) \\
&= \sigma^{-1}(\@B(X)) \ocap \tau^{-1}(\@B(X)) \ocap \@{BO}_G(G) \\
&= \@B(G) \oast \@B(G)
= \@{BO}_\tau(G) \oast \@{BO}_\sigma(G).
\end{aligned}
\end{gather}
\end{enumeqno}
\end{theorem}
(Recall the meanings of $\ocap, \oast$ from \cref{def:ostar,def:vaught}, $\@{BO}_\Box$ from \cref{def:boxtop}, and \emph{compatible sigma-topology} from \cref{def:stop}.)
\begin{proof}
First, from \cref{def:ucomqpolgpd}, we get
\begin{enumeqno}[resume*=thm:ucomqpolgpd]
\begin{eqaligned}[b][\label{thm:ucomqpolgpd:inv}]
\@{BO}_\tau(G)
&= (\nu, \iota\tau)^{-1}(\mu^{-1}(\@{BO}_\sigma(G)))
    &&\text{since $g^{-1} = g^{-1} \cdot 1_{\tau(g)}$, i.e., $\nu = \mu \circ (\nu, \iota\tau)$} \\
&\subseteq (\nu, \iota\tau)^{-1}(\@{BO}_\tau(G) \otimes_X \@{BO}_\sigma(G))
    &&\text{by \labelcref{def:ucomqpolgpd}\cref{def:ucomqpolgpd:mu-sigma}} \\
&\subseteq \@{BO}_\sigma(G) \ocap \tau^{-1}(\@B(X)).
\end{eqaligned}
Now we show \cref{thm:ucomqpolgpd:boxtop}:
\begin{eqaligned*}
\@{BO}_\Box(G)
&\subseteq \@B(G) \oast \@B(G) = \@{BO}_\tau(G) \oast \@{BO}_\sigma(G)
    &&\text{by \cref{rmk:boxtop-vaught}} \\
&\subseteq (\tau^{-1}(\@B(X)) \ocap \@{BO}_\sigma(G)) \oast (\@{BO}_\tau(G) \ocap \sigma^{-1}(\@B(X)))
    &&\text{by \cref{thm:ucomqpolgpd:inv} and its dual} \\
&= \tau^{-1}(\@B(X)) \ocap (\@{BO}_\sigma(G) \oast \@{BO}_\tau(G)) \ocap \sigma^{-1}(\@B(X))
    &&\text{by Frobenius reciprocity \cref{it:vaught-frob}} \\
&= \tau^{-1}(\@B(X)) \ocap \@{BO}_G(G) \ocap \sigma^{-1}(\@B(X))
    &&\text{by \cref{thm:orbtop-exists*}} \\
&\subseteq (\tau^{-1}(\@B(X)) \ocap \@{BO}_\sigma(G)) \cap (\sigma^{-1}(\@B(X)) \ocap \@{BO}_\tau(G))
    &&\text{by \cref{eq:orbtop-fibtop}} \\
&\subseteq \@{BO}_\Box(G)
    &&\text{by \cref{def:boxtop}};
\end{eqaligned*}
thus all of these sigma-topologies are equal.

Next, we show \cref{thm:ucomqpolgpd:mu-fib}: the idea is to use the associativity axiom $(gh)k = g(hk)$, i.e., $\mu \circ (\mu \times \id_G) = \mu \circ (\id_G \times \mu)$, and apply $\mu$-preimage and $\exists^*_\mu$ while switching the order of multiplication.
\begin{zalign*}
\mu^{-1}(\@{BO}_\sigma(G))
&\subseteq \exists^*_{\id_G \times \mu}((\id_G \times \mu)^{-1}(\mu^{-1}(\@{BO}_\sigma(G))))
    &\quad&\text{by \cref{it:fib-baire-surj}} \\
&= \exists^*_{\id_G \times \mu}((\mu \times \id_G)^{-1}(\mu^{-1}(\@{BO}_\sigma(G))))
    &&\text{by associativity} \\
&\subseteq \exists^*_{\id_G \times \mu}((\mu \times \id_G)^{-1}(\@{BO}_\sigma(G) \otimes_X \@B(G)))
    &&\text{by \cref{rmk:fibqpolgpd-bor}} \\
&\subseteq \exists^*_{\id_G \times \mu}(\@{BO}_\sigma(G) \otimes_X \@B(G) \otimes_X \@B(G))
    &&\text{by \cref{rmk:fibqpolgpd-bor}} \\
&= \@{BO}_\sigma(G) \otimes_X (\@B(G) \oast \@B(G)) \\
&= \@{BO}_\sigma(G) \otimes_X (\tau^{-1}(\@B(X)) \ocap \@{BO}_\sigma(G))
    &&\text{by \cref{thm:ucomqpolgpd:boxtop}} \\
\yesnumber \label{thm:ucomqpolgpd:mu-ss}
&\subseteq \@{BO}_\sigma(G) \otimes_X \@{BO}_\sigma(G)
    &&\text{since $\tau\pi_2 = \sigma\pi_1$};
\\[1ex]
\mu^{-1}(\@{BO}_\sigma(G))
&\subseteq \exists^*_{\mu \times \id_G}((\mu \times \id_G)^{-1}(\mu^{-1}(\@{BO}_\sigma(G))))
    &&\text{by \cref{it:fib-baire-surj}} \\
&= \exists^*_{\mu \times \id_G}((\id_G \times \mu)^{-1}(\mu^{-1}(\@{BO}_\sigma(G))))
    &&\text{by associativity} \\
&\subseteq \exists^*_{\mu \times \id_G}((\id_G \times \mu)^{-1}(\@{BO}_\sigma(G) \otimes_X \@{BO}_\sigma(G)))
    &&\text{by \cref{thm:ucomqpolgpd:mu-ss}} \\
&\subseteq \exists^*_{\mu \times \id_G}(\@{BO}_\sigma(G) \otimes_X \@{BO}_\tau(G) \otimes_X \@{BO}_\sigma(G))
    &&\text{by \labelcref{def:ucomqpolgpd}\cref{def:ucomqpolgpd:mu-sigma}} \\
&= (\@{BO}_\sigma(G) \oast \@{BO}_\tau(G)) \otimes_X \@{BO}_\sigma(G) \\
&= \@{BO}_G(G) \otimes_X \@{BO}_\sigma(G)
    &&\text{by \cref{thm:orbtop-exists*}}.
\end{zalign*}
From this, we get \cref{thm:ucomqpolgpd:mu-orb}:
\begin{eqaligned*}
\mu^{-1}(\@{BO}_G(G))
&= \exists^*_{\id_G \times \mu}(\mu^{-1}(\@{BO}_\sigma(G)) \otimes_X \@{BO}_\tau(G))
    &&\text{by associativity as above} \\
&\subseteq \exists^*_{\id_G \times \mu}(\@{BO}_G(G) \otimes_X \@{BO}_\sigma(G) \otimes_X \@{BO}_\tau(G))
    &&\text{by \cref{thm:ucomqpolgpd:mu-fib}} \\
&= \@{BO}_G(G) \otimes_X \@{BO}_G(G)
    &&\text{by \cref{thm:orbtop-exists*} as above}.
\end{eqaligned*}
\Cref{thm:ucomqpolgpd:fibtop} also follows:
\begin{eqaligned*}
\@{BO}_\sigma(G)
&\subseteq \@{BO}_\sigma(G) \oast \@{BO}_\sigma(G)
    &&\text{by \cref{thm:ucomqpolgpd:mu-ss,it:fib-baire-surj}} \\
&\subseteq \@{BO}_\sigma(G) \oast (\@{BO}_\tau(G) \ocap \sigma^{-1}(\@B(X)))
    &&\text{by the dual of \cref{thm:ucomqpolgpd:inv}} \\
&\subseteq (\@{BO}_\sigma(G) \oast \@{BO}_\tau(G)) \ocap \sigma^{-1}(\@B(X))
    &&\text{by Frobenius \cref{it:vaught-frob}} \\
&= \@{BO}_G(G) \ocap \sigma^{-1}(\@B(X))
    &&\text{by \cref{thm:orbtop-exists*}};
\end{eqaligned*}
\end{enumeqno}
whereas $\supseteq$ is clear (by \labelcref{eq:orbtop-fibtop}).

By \cref{thm:ucomqpolgpd:mu-orb}, $\mu : G \times_X G -> G$ is continuous between the sigma-topologies
$\@{BO}_G(G) \otimes_X \@{BO}_G(G)$ and $\@{BO}_G(G)$.
We already verified in \cref{rmk:orbtop-bor} that the other 4 groupoid operations are also continuous between the Borel componentwise sigma-topologies, and also $\sigma(\@{BO}_G(G)) = \exists^*_\sigma(\@{BO}_G(G)) \subseteq \@{BO}_X(G)$ by fiberwise Baireness and \cref{thm:orbtop-exists*}.
Hence, $(X,G)$ equipped with $\@{BO}_G(X)$ and $\@{BO}_G(G)$ is an open sigma-topological groupoid.

Finally, to show compatibility of $\@{BO}_G(X)$, we apply \cref{thm:sierpinski-exists*-stop} to the Borel-overt bundle $\tau : (G,\@{BO}_\tau(G)) ->> X$ and the compatible sigma-topology
$\@S(G) := \@{BO}_\sigma(G)$
on $G$.
We have
\begin{eqaligned*}
\tau^{-1}(\exists^*_\tau(\@{BO}_\sigma(G)))
= \tau^{-1}(\@{BO}_G(X))
\subseteq \@{BO}_\sigma(G)
\end{eqaligned*}
by \cref{thm:orbtop-exists*}, and $\@{BO}_\sigma(G)$ contains a countable Borel fiberwise basis for a finer $\tau$-fiberwise topology than $\@O_\tau(G)$ by \cref{thm:ucomqpolgpd:inv}:
namely, take any Borel fiberwise basis of $U_i \in \@{BO}_\tau(G)$, and write each
$U_i = \bigcup_j (\tau^{-1}(B_j) \cap V_{ij})$
where $V_{ij} \in \@{BO}_\sigma(G)$; then the $V_{ij}$'s work.

Similarly, to show compatibility of $\@{BO}_G(G)$, we apply \cref{thm:sierpinski-exists*-stop} to $\mu : G \times_X G ->> G$ with the $\mu$-fiberwise topology (\cref{ex:vaught-mul}) and the compatible sigma-topology
$\@S(G \times_X G) := \@{BO}_\sigma(G) \otimes_X \@{BO}_\tau(G)$, whose image under $\exists^*_\mu$ is $\@{BO}_G(G) \subseteq (\mu^{-1})^{-1}(\@S(G \times_X G))$ again by \cref{thm:orbtop-exists*}.
To find a countable Borel fiberwise basis for a finer topology than $\@O_\mu$ contained in $\@S(G \times_X G)$: for any Borel fiberwise basis of $U_i \in \@{BO}_\sigma(G)$, the sets $U_i \times_X G$ form a $\pi_2$-fiberwise basis for $G \times_X G$, hence their $G \ltimes G$-inverses $U_i^{-1} \times_X G$ (recall \cref{def:action}) form a $\mu$-fiberwise basis.
By \cref{thm:ucomqpolgpd:inv},
\begin{eqaligned*}
U_i^{-1} = \bigcup_j (\tau^{-1}(B_j) \cap V_{ij})
\end{eqaligned*}
for $B_j \in \@B(X)$ and $V_{ij} \in \@{BO}_\sigma(G)$.
Then the sets
\begin{eqaligned*}
(\tau^{-1}(B_j) \cap V_{ij}) \times_X G = \mu^{-1}(\tau^{-1}(B_j)) \cap (V_{ij} \times_X G),
\end{eqaligned*}
hence also $V_{ij} \times_X G \in \@{BO}_\sigma(G) \otimes_X \@{BO}_\tau(G)$, generate a finer fiberwise topology than $\@O_\mu(G \times_X G)$.
\end{proof}

\Cref{thm:ucomqpolgpd} is the essence of a global topological realization result for Borel-overt uniformly componentwise quasi-Polish groupoids, except that it yields a sigma-topological groupoid, rather than a quasi-Polish groupoid.
From here, a standard $\omega$-iteration argument shows that arbitrarily fine quasi-Polish groupoid topologies may be found within the orbitwise sigma-topologies:

\begin{lemma}
\label{thm:stopgpd}
Let $(X,G)$ be a standard Borel groupoid, equipped with compatible sigma-topologies $\@S(X) \subseteq \@B(X)$ and $\@S(G) \subseteq \@B(G)$ making all of the groupoid operations continuous (and $\sigma, \tau$ open).
Then a club of pairs $(\@O(X), \@O(G))$ of quasi-Polish topologies $\@O(X) \subseteq \@S(X)$ and $\@O(G) \subseteq \@S(G)$ turn $(X,G)$ into an (open) quasi-Polish groupoid.
\end{lemma}
(Recall the notion of a \emph{club of quasi-Polish topologies} from \cref{rmk:stop-club};
a \emph{club of pairs of quasi-Polish topologies} is the obvious generalization to the product lattice of pairs of topologies.)
\begin{proof}
Clearly such pairs $(\@O(X), \@O(G))$ are closed under countable increasing joins, so it suffices to verify unboundedness.
Let $\@O_0(X) \subseteq \@S(X)$ and $\@O_0(G) \subseteq \@S(G)$ be any quasi-Polish topologies.
Given $\@O_n(X), \@O_n(G)$, take refinements $\@O_{n+1}(X), \@O_{n+1}(G)$ containing their preimages under all the groupoid operations (and $\sigma, \tau$-images; for preimage under $\mu$, for each basic $U \in \@O_n(G)$, include into $\@O_{n+1}(G)$ an arbitrary countable family of sets in $\@S(G)$ witnessing $\mu^{-1}(U) \in \@S(G) \otimes_X \@S(G)$).
Let $\@O(X), \@O(G)$ be the topologies generated by $\bigcupup_n \@O_n(X), \bigcupup_n \@O_n(G)$.
\end{proof}

\begin{theorem}
\label{thm:ucomqpolgpd-qpol}
Let $(X,G)$ be a Borel-overt uniformly componentwise quasi-Polish groupoid.
Then there exist compatible quasi-Polish topologies $\@O(X) \subseteq \@{BO}_G(X)$ and $\@O(G) \subseteq \@{BO}_G(G)$ turning $(X,G)$ into an open quasi-Polish groupoid, such that $\@O(G)$ restricts to the original fiberwise topology.

Moreover, such pairs $(\@O(X), \@O(G))$ form a club below $(\@{BO}_G(X), \@{BO}_G(G))$, i.e., any countably many Borel componentwise open $A_i \in \@{BO}_G(X)$ and $U_i \in \@{BO}_G(G)$ may be included in $\@O(X), \@O(G)$.
\end{theorem}
\begin{proof}
By \cref{thm:ucomqpolgpd} and \cref{thm:stopgpd}, with countably many $U_i \in \@{BO}_G(G)$ forming a countable fiberwise basis for $\@O_\sigma(G)$ by \labelcref{thm:ucomqpolgpd}\cref{thm:ucomqpolgpd:fibtop} included in the resulting $\@O(G)$ in order to ensure that $\@O(G)$ fiberwise restricts to the originally given $\@O_\sigma(G)$.
\end{proof}

\begin{remark}
It is not possible to find an open quasi-Polish groupoid topology $(\@O(X),\@O(G))$ restricting \emph{componentwise} to $(\@{BO}_G(X),\@{BO}_G(G))$, unless $\#E_G$ is smooth; see \cref{thm:ucomqpolgpd-smooth}.
\end{remark}

\begin{corollary}
\label{thm:ucomqpolgpd-homtop}
Every Borel-overt uniformly componentwise quasi-Polish groupoid $(X,G)$ has $\*\Pi^0_2$ hom-sets.
Hence, hom-sets within each component are homeomorphic copies of the same Polish group; the componentwise topologies $\@O_G(X), \@O_G(G)$ are quasi-Polish on each component; and every weakly componentwise meager or $\*\Sigma^0_2$ set $A \subseteq X$ is strongly so.
\end{corollary}
(The ``uniformly'' assumption can be removed; see \cref{thm:comqpolgpd-homtop}.)
\begin{proof}
Clearly, hom-sets of a quasi-Polish groupoid are $\*\Pi^0_2$; now apply \cref{thm:ucomqpolgpd-qpol}.
The listed consequences follow from the discussion in \cref{rmk:homtop,rmk:orbtop-t0,def:orbtop-weak}.
\end{proof}

The following is a simple example application of \cref{thm:ucomqpolgpd-qpol}:

\begin{example}
Let $(X,G)$ be a \defn{fiberwise countable Borel groupoid} (also known as \emph{locally countable Borel groupoid}, or sometimes simply as \emph{countable Borel groupoid}), i.e., each fiber $\sigma^{-1}(x)$ is countable, or equivalently each connected component contains countably many morphisms.
We may equip $G$ with the discrete $\sigma$-fiberwise topology $\@{BO}_\sigma(G) := \@B(G)$, which is clearly right-invariant, and Borel-overt by the Lusin--Novikov uniformization theorem \cite[18.10]{kechris:cdst}.
Uniform fiberwise continuity of differences is trivial, since
\begin{equation*}
\mu^{-1}(\@{BO}_\sigma(G)) \subseteq \@B(G) \otimes_X \@B(G) = \@{BO}_\tau(G) \otimes_X \@{BO}_\sigma(G).
\end{equation*}
We thus have a Borel-overt uniformly componentwise quasi-Polish groupoid, which by \cref{thm:ucomqpolgpd-qpol} may be turned into an open quasi-Polish groupoid, such that each fiber $\sigma^{-1}(x)$ is discrete, i.e., the diagonal in each $\sigma^{-1}(x) \times \sigma^{-1}(x)$ is open.
We may in fact make the diagonal in $G \tensor[_\sigma]\times{_\sigma} G$ open, i.e., make $G$ into an \defn{étale groupoid}: indeed, the diagonal is in $\@B(G) \tensor[_\sigma]\otimes{_\sigma} \@B(G) = \@{BO}_\sigma(G) \tensor[_\sigma]\otimes{_\sigma} \@{BO}_\sigma(G)$ (by Kunugi--Novikov, or more simply by Lusin--Novikov), hence is a countable union of rectangles $U_i \times_X V_i$ where $U_i, V_i \in \@{BO}_\sigma(G)$; by \cref{thm:ucomqpolgpd-qpol} these may be included in the topology $\@O(G)$.

This recovers a result from \cite[\S7.7]{lupini:polgpd} (proved there using more elementary methods) as a formal consequence of \cref{thm:ucomqpolgpd-qpol}.
\end{example}

Finally, we give the promised example of a countable (hence trivially Borel-overt) fiberwise Polish groupoid \emph{without} fiberwise continuous differences, to show that this extra assumption is indeed necessary for the global topological realization \cref{thm:ucomqpolgpd-qpol} to hold.

\begin{example}
\label{ex:fibqpolgpd-disctscocy}
Fix some nontrivial countable discrete abelian group $\Gamma$, e.g., $\Gamma = \#Z/2\#Z$.
Consider the product group $\Gamma^\#N$ with the product topology, which is zero-dimensional Polish, as well as the direct sum $\Gamma^{\oplus \#N}$ with the discrete topology acting on $\Gamma^\#N$ by translation.
Let $\-{\#N} = \#N \cup \{\infty\}$ be the one-point compactification of the discrete space $\#N$, which is a countable Stone space.
Put
\begin{eqaligned*}
Y := \-{\#N} \times \Gamma^\#N,
\end{eqaligned*}
with the product topology, and let $\Gamma^{\oplus \#N}$ act on the second coordinate of $Y$.
For each $\→\gamma \in \Gamma^{\oplus \#N}$, put
\begin{eqaligned*}
C_{\→\gamma} := \set{(i,\→\delta) \in Y | \forall j < i\, (\delta_j = \gamma_j)},
\end{eqaligned*}
a closed set.
Note that if we adjoin each of these sets to the topology of $Y$, as well as the singletons $\set{(i,\→\gamma)}$ for all $i \in \#N$ and $\→\gamma \in \Gamma^{\oplus \#N}$, then the action $\Gamma^{\oplus \#N} \actson Y$ remains continuous.
Now let
\begin{eqaligned*}
Z := \-{\#N} \times \Gamma^{\oplus \#N}
= \set[\big]{(i,\→\gamma) | i \in \#N,\, \→\gamma \in \Gamma^{\oplus \#N}} \cup \paren[\big]{(\set{\infty} \times \Gamma^\#N) \cap \bigcup_{\→\gamma \in \Gamma^{\oplus \#N}} C_{\→\gamma}} \subseteq Y
\end{eqaligned*}
equipped with the subspace topology induced by the aforementioned finer topology on $Y$, in which $Z$ is a $\*\Pi^0_2$ subspace, hence zero-dimensional Polish; and $\Gamma^{\oplus \#N}$ still acts continuously on $Z$.

Define the groupoid $(X,G)$, where $X := \-{\#N}$ and
\begin{eqaligned*}
G := X \times Z = \-{\#N} \times \-{\#N} \times \Gamma^{\oplus \#N}
\end{eqaligned*}
equipped with the product topology on $X \times Z$.
The source and target maps are the first and second projections $\sigma, \tau : G -> X$; note that $\sigma, \tau$ are both continuous.
The multiplication is
\begin{eqaligned*}
(j, k, \→\delta) \cdot (i, j, \→\gamma) := (i, k, \→\gamma + \→\delta).
\end{eqaligned*}
In other words, disregarding the topology, $G$ is just the product of the indiscrete equivalence relation $\-{\#N}^2$ on $\-{\#N}$ and the group $\Gamma^{\oplus \#N}$.
Note that this is continuous for fixed $(i,j,\→\gamma)$; thus, the $\sigma$-fiberwise restriction of the product topology on $G = X \times Z$ yields a countable fiberwise Polish groupoid.

We now show that there is no global topology on $G$ with the same $\sigma$-fiberwise restriction which renders all of the groupoid operations continuous.
If there were, then the difference map (\labelcref{def:comqpolgpd}\cref{def:comqpolgpd:diff1})
\begin{eqaligned*}
\sigma^{-1}(0) \tensor[_\tau]\times{_\tau} \sigma^{-1}(0) &--> G(0,0) \\
(g,h) = ((0,i,\→\gamma), (0,i,\→\delta)) &|--> g^{-1} \cdot h = (0,0,\→\delta-\→\gamma)
\end{eqaligned*}
would be continuous.
However, the preimage of the isolated point $(0,0,\→0)$ is the diagonal, which is not a neighborhood of $((0,\infty,\→0), (0,\infty,\→0))$: a basic neighborhood of $(0,\infty,\→0) \in G$ is of the form
\begin{equation*}
\set{0} \times (([n,\infty] \times \Gamma^{\oplus \#N}) \cap C_{\→0})
\end{equation*}
for some large $n \in \#N$;
and any such set contains two distinct $(0,n,\→\gamma), (0,n,\→\delta)$, where $\→\gamma, \→\delta$ are both $0$ in the first $n$ terms but differ at some point after that,
which means that the product of two such sets cannot intersect $\sigma^{-1}(0) \tensor[_\tau]\times{_\tau} \sigma^{-1}(0)$ only in the diagonal.
\end{example}

\begin{remark}
The above groupoid, despite not having fiberwise continuous differences, nonetheless obeys \labelcref{thm:ucomqpolgpd}\cref{thm:ucomqpolgpd:fibtop}.
We do not know whether there exists a Borel-overt fiberwise quasi-Polish groupoid failing \labelcref{thm:ucomqpolgpd}\cref{thm:ucomqpolgpd:fibtop}.
\end{remark}

\subsection{Action groupoids and the Becker--Kechris theorem}

Let $(X,G)$ be a Borel-overt fiberwise quasi-Polish groupoid.
Recall from \cref{def:fibqpolgpd-action} that for a Borel $G$-space $p : M -> X$, we may lift the fiberwise topology on $G$ to the action groupoid $G \ltimes M$.

\begin{lemma}
\label{thm:comqpolgpd-action}
If $G$ has (uniformly) fiberwise continuous differences, then so does the lifted fiberwise topology on $G \ltimes M$, for any Borel $G$-space $p : M -> X$.
\end{lemma}
\begin{proof}
The non-uniform version is straightforward (each of the maps in \cref{def:comqpolgpd} for the action groupoid is a homeomorphic copy of the original map for $G$).
For the uniform version: recalling $\mu_{G \ltimes M}, \nu_{G \ltimes M}$ from \cref{def:action} and $\@{BO}_\sigma(G \ltimes M), \@{BO}_\tau(G \ltimes M)$ from \cref{def:fibqpolgpd-action},
\begin{eqaligned*}
\mu_{G \ltimes M}^{-1}(\@{BO}_\sigma(G \ltimes M))
&= (\mu_G \times M)^{-1}(\@{BO}_\sigma(G) \otimes_X \@B(M)) \\
&\subseteq \@{BO}_\tau(G) \otimes_X \@{BO}_\sigma(G) \otimes_X \@B(M) \\
&\cong (\@{BO}_\tau(G) \otimes_X M) \tensor[_{\pi_2}]\otimes{_\alpha} (\@{BO}_\sigma(G) \otimes_X \@B(M)) \\
&\subseteq \nu_{G \ltimes M}^{-1}(\@{BO}_\sigma(G) \otimes_X \@B(M)) \tensor[_{\pi_2}]\otimes{_\alpha} (\@{BO}_\sigma(G) \otimes_X \@B(M)) \\
&= \@{BO}_\tau(G \ltimes M) \otimes_M \@{BO}_\sigma(G \ltimes M)
\end{eqaligned*}
using in the second-last step that $(U \times_X M)^{-1} = U^{-1} \times_X M$.
\end{proof}

We may simplify the description of the orbitwise topology in this case (as in \cref{ex:polgrp-action}):

\begin{lemma}
\label{thm:ucomqpolgpd-action-orbtop}
If $G$ has uniformly fiberwise continous differences, then
\begin{align*}
\@{BO}_G(G \ltimes M)
&= \@{BO}_G(G) \otimes_X \@{BO}_G(M) \\
&= \@O(G) \otimes_X \@{BO}_G(M),
\end{align*}
for any compatible open quasi-Polish groupoid topology $\@O(G)$ as in \cref{thm:ucomqpolgpd-qpol}.
\end{lemma}
By taking $M = X$, this shows in particular that for an open quasi-Polish groupoid $G$,
\begin{align*}
\@{BO}_G(G) = \sigma^{-1}(\@{BO}_G(X)) \ocap \@O(G) = \tau^{-1}(\@{BO}_G(X)) \ocap \@O(G),
\end{align*}
as mentioned in \cref{ex:qpolgpd-orbtop} (see also \cref{fig:topologies}).
\begin{proof}
$\supseteq$ (this does not use continuity of differences):
Recalling again \cref{def:action,def:fibqpolgpd-action},
\begin{eqaligned*}
\mu_{G \ltimes M}^{-1}(\@{BO}_G(G) \otimes_X \@{BO}_G(M))
&= \mu_G^{-1}(\@{BO}_G(G)) \otimes_X \@{BO}_G(M) \\
&\subseteq \@{BO}_\sigma(G) \otimes_X \@{BO}_\tau(G) \otimes_X \@{BO}_G(M) \\
&\cong (\@{BO}_\sigma(G) \otimes_X M) \tensor[_{\pi_2}]\otimes{_\alpha} (\@{BO}_\tau(G) \otimes_X \@{BO}_G(M)) \\
&\subseteq (\@{BO}_\sigma(G) \otimes_X \@B(M)) \tensor[_{\pi_2}]\otimes{_\alpha} \nu_{G \ltimes M}^{-1}(\@{BO}_\sigma(G) \otimes_X \@B(M)) \\
&= \@{BO}_\sigma(G \ltimes M) \otimes_M \@{BO}_\tau(G \ltimes M)
\end{eqaligned*}
using in the second-last step that $(U \times_X A)^{-1} = (U^{-1} \times_X M) \cap \alpha^{-1}(A)$.

$\subseteq$:
It suffices to prove the statement for $\@O(G)$, by \cref{thm:ucomqpolgpd-qpol}.
Note that since $\@O(G)$ fiberwise restricts to $\@{BO}_\sigma(G)$, we may simplify \cref{def:fibqpolgpd-action} to
\begin{eqaligned*}
\@{BO}_\sigma(G \ltimes M)
= \@{BO}_\sigma(G) \otimes_X \@B(M)
= (\@O(G) \ocap \sigma^{-1}(\@B(X))) \otimes_X \@B(M)
= \@O(G) \otimes_X \@B(M).
\end{eqaligned*}
Similarly to in the proof of \cref{thm:ucomqpolgpd}, we have
\begin{eqaligned*}[b][\tag{$*$}]
\alpha^{-1}(\@{BO}_G(M))
&\subseteq \exists^*_{\id_G \times \alpha}((\id_G \times \alpha)^{-1}(\alpha^{-1}(\@{BO}_G(M))))
    &&\text{by \cref{it:fib-baire-surj}} \\
&= \exists^*_{\id_G \times \alpha}((\mu \times \id_M)^{-1}(\alpha^{-1}(\@{BO}_G(M))))
    &&\text{by associativity} \\
&\subseteq \exists^*_{\id_G \times \alpha}((\mu \times \id_M)^{-1}(\@O(G) \otimes_X \@B(M)))
    &&\text{by \cref{eq:orbtop-action} and above} \\
&\subseteq \exists^*_{\id_G \times \alpha}(\@O(G) \otimes_X \@O(G) \otimes_X \@B(M))
    &&\text{by $\@O(G)$-continuity of $\mu$} \\
&= \@O(G) \otimes_X (\@O(G) \oast \@B(M)) \\
&= \@O(G) \otimes_X \@{BO}_G(M)
    &&\text{by \cref{thm:orbtop-exists*}}.
\end{eqaligned*}
Now by \cref{thm:orbtop-exists*,def:fibqpolgpd-action} and the above,
\begin{zalign*}
\@{BO}_G(G \ltimes M)
&= \@{BO}_\sigma(G \ltimes M) \oast \@{BO}_\tau(G \ltimes M) \\
&= \@{BO}_\sigma(G \ltimes M) \oast \nu_{G \ltimes M}^{-1}(\@{BO}_\sigma(G \ltimes M)) \\
&= (\@O(G) \otimes_X \@B(M)) \oast \nu_{G \ltimes M}^{-1}(\@O(G) \otimes_X \@B(M)); \\
\intertext{by the orbitwise Baire property (\labelcref{thm:orbtop-bp}; recall also \labelcref{thm:ucomqpolgpd-homtop}), every set in $\@B(M)$ is equal mod orbitwise meager to some set in $\@{BO}_G(M)$, whence their $\sigma_{G \ltimes M} = \pi_2$-preimages are equal mod $\tau_{G \ltimes M}$-fiberwise meager, and dually for the $\tau_{G \ltimes M}$-preimages; now by Pettis's theorem \cref{it:vaught-pettis} (for the Vaught transform of the left translation action $(G \ltimes M) \actson (G \ltimes M)$, \emph{not} the action $G \actson M$), we get}
&= \mathrlap{(\@O(G) \otimes_X \@{BO}_G(M)) \oast \nu_{G \ltimes M}^{-1}(\@O(G) \otimes_X \@{BO}_G(M))} \\
&= \mathrlap{\exists^*_{\mu_{G \ltimes M}}((\@O(G) \otimes_X \@{BO}_G(M)) \tensor[_{\pi_2}]\otimes{_\alpha} \nu_{G \ltimes M}^{-1}(\@O(G) \otimes_X \@{BO}_G(M))} \\
\tag{$\dagger$}
&= \mathrlap{\exists^*_{\mu_G \times M}(\@O(G) \otimes_X (\alpha^{-1}(\@{BO}_G(M)) \ocap (\nu_G^{-1}(\@O(G)) \otimes_X M)))} \\
&\subseteq \exists^*_{\mu_G \times M}(\@O(G) \otimes_X (\@O(G) \otimes_X \@{BO}_G(M)))
    &&\text{by ($*$)} \\
&= \@O(G) \otimes_X \@{BO}_G(M),
\end{zalign*}
where step ($\dagger$) is by transporting the sets across the isomorphism
$(G \times_{X} M) \tensor[_{\pi_2}]\times{_\alpha} (G \times_{X} M) \cong G \times_{X} G \times_{X} M$
in \cref{def:action} of $\mu_{G \ltimes M}$.
\end{proof}

Similarly to \cref{thm:ucomqpolgpd}, \cref{thm:ucomqpolgpd-action-orbtop} may be regarded as an abstract, sigma-topological form of the Becker--Kechris topological realization theorem, and easily implies the concrete result:

\begin{corollary}[{Becker--Kechris theorem for open quasi-Polish groupoid actions \cite[4.3.12]{chen:beckec}}]
\label{thm:beckec}
\leavevmode\par\noindent
Let $(X,G)$ be an open quasi-Polish groupoid, $p : M -> X$ be a Borel $G$-space.
Then for a club of compatible quasi-Polish topologies $\@O(M) \subseteq \@{BO}_G(M)$, the action $G \actson (M,\@O(M))$ is continuous.
\end{corollary}
\begin{proof}
Note that a quasi-Polish topology $\@O(M) \subseteq \@{BO}_G(M)$ makes the action $G \actson M$ continuous iff
$\@O(G) \otimes_X \@O(M)$ is a groupoid topology on $G \ltimes M$ (since the action map $\alpha : G \times_X M -> M$ is $\tau_{G \ltimes M}$).
The collection of all such topologies $\@O(M)$ is the preimage, under the sigma-complete join-semilattice homomorphism $\@O(G) \otimes_X (-)$ which is cofinal by \cref{thm:ucomqpolgpd-action-orbtop}, of the club of groupoid topologies from \cref{thm:ucomqpolgpd-qpol}, hence is itself a club.
\end{proof}

Thus in some sense (modulo the computation in \cref{thm:ucomqpolgpd-action-orbtop}), the topological realization theorem \labelcref{thm:ucomqpolgpd-qpol} for groupoids ``subsumes'' the Becker--Kechris topological realization theorem for actions (which of course largely inspired \cref{thm:ucomqpolgpd-qpol}).

\subsection{Comeager subgroupoids}

We now extend a weakening of \cref{thm:ucomqpolgpd-qpol} to all Borel-overt \emph{fiberwise} quasi-Polish groupoids $(X,G)$: even though such groupoids need not have fiberwise continuous differences by \cref{ex:fibqpolgpd-disctscocy}, we will show that they can always be replaced with a Borel-equivalent groupoid that does.

Recall from \cref{rmk:fibqpolgpd-full} that one way to get such an equivalent groupoid is to restrict to the full subgroupoid $\eval{G}_Y$ on a Borel weakly componentwise $\*\Pi^0_2$ dense subset of objects $Y \subseteq X$, and that all of the Borel componentwise topologies, category quantifiers, and Vaught transforms we have considered in \cref{sec:fibgpd}, when evaluated in such a componentwise comeager full subgroupoid, are just the restrictions of the corresponding notions from the ambient groupoid $(X,G)$.

\begin{lemma}
\label{thm:ucomqpolgpd-dense}
Let $(X,G)$ be a Borel-overt fiberwise quasi-Polish groupoid.
There is a Borel weakly componentwise $\*\Pi^0_2$ dense $Y \subseteq X$ such that $(Y,\eval{G}_Y)$ has uniformly fiberwise continuous differences.
\end{lemma}
\begin{proof}
Let $\{U_i\}_i \subseteq \@{BO}_\sigma(G)$ be a countable Borel fiberwise basis.
Consider as in the proof of \cref{thm:ucomqpolgpd} the quaternary multiplication $\mu_4 : G \times_X G \times_X G \times_X G -> G$, written using associativity as either $(g,h,k,l) |-> ((gh)k)l$ or $(gh)(kl)$.
By \cref{rmk:fibqpolgpd-bor} applied thrice, write each
\begin{eqaligned*}
\mu_4^{-1}(U_i) = \bigcup_{j \in \#N} (V_{ij} \times_X A_{ij} \times_X B_{ij} \times_X C_{ij})
\end{eqaligned*}
for $V_{ij} \in \@{BO}_\sigma(G)$ and $A_{ij}, B_{ij}, C_{ij} \in \@B(G)$.
By the fiberwise Baire property (\labelcref{thm:fib-baire-borel}), find
\begin{align*}
A_{ij} &=^*_\tau A_{ij}' \in \@{BO}_\tau(G), &
B_{ij} &=^*_\sigma B_{ij}' \in \@{BO}_\sigma(G).
\end{align*}
Thus $A_{ij} \triangle A_{ij}'$ and $B_{ij} \triangle B_{ij}'$ are $\tau$- and $\sigma$-fiberwise meager respectively.
By \cref{thm:fibtop-meager-exists*},
\begin{eqaligned*}[][]
\bigcup_{i,j} (
    \exists^*_\sigma(A_{ij} \triangle A_{ij}') \cup
    \exists^*_\tau(B_{ij} \triangle B_{ij}')
) \subseteq X
\end{eqaligned*}
is weakly componentwise meager, hence by \cref{thm:orbtop-meager-bp} disjoint from some Borel weakly componentwise $\*\Pi^0_2$ dense set $Y$.
This means that after restricting to the full subgroupoid $\eval{G}_Y$, we\nobreak\space have
\begin{align*}
A_{ij} &=^*_\sigma A_{ij}', &
B_{ij} &=^*_\tau B_{ij}',
\end{align*}
thus by Pettis's theorem \cref{it:vaught-pettis} and \cref{eq:vaught-fibtop},
\begin{eqaligned*}
\mu^{-1}(U_i)
= \exists^*_{\mu \times \mu}(\mu_4^{-1}(U_i))
= \bigcup_j ((V_{ij} * A_{ij}') \times_Y (B_{ij}' * C_{ij}'))
\in \@{BO}_\tau(\eval{G}_Y) \otimes_Y \@{BO}_\sigma(\eval{G}_Y).
\end{eqaligned*}
By Kunugi--Novikov \cref{eq:fib-bor-qpol-basis}, it follows that
$\mu^{-1}(\@{BO}_\sigma(\eval{G}_Y)) \subseteq \@{BO}_\tau(\eval{G}_Y) \otimes_Y \@{BO}_\sigma(\eval{G}_Y)$.
\end{proof}

\begin{theorem}
\label{thm:fibqpolgpd-realiz-equiv}
Every Borel-overt fiberwise quasi-Polish groupoid $(X,G)$ admits a Borel equivalence of groupoids to an open quasi-Polish groupoid, namely a Borel weakly componentwise $\*\Pi^0_2$ dense full subgroupoid $(Y,\eval{G}_Y)$ equipped with a quasi-Polish topology restricting to its fiberwise topology.
\end{theorem}
\begin{proof}
Apply \cref{thm:ucomqpolgpd-qpol} to the full subgroupoid $(Y,\eval{G}_Y)$ from \cref{thm:ucomqpolgpd-dense}.
\end{proof}

\begin{example}
\label{ex:qpolgrp-realiz}
In the case of a group $(X,G) = (*,G)$, the above says that a standard Borel group $G$ equipped with a compatible quasi-Polish topology which is translation-invariant on one side must already be a Polish group.
This is a weak form of a result of Solecki--Srivastava \cite{solecki-srivastava:lpolgrp} (assuming the stronger hypothesis of joint Borelness of $\mu$, albeit with metrizability replaced by quasi-Polish).
\end{example}

Using the above results, we may also clean up some loose ends for groupoids under our earlier weaker hypotheses.
Recall from \cref{rmk:homtop} the question of whether hom-sets in a Borel-overt fiberwise quasi-Polish groupoid are always Polish, equivalently $\*\Pi^0_2$ in the $\sigma$- or $\tau$-fiberwise topologies.

\begin{corollary}
\label{thm:homtop-dense}
Let $(X,G)$ be a Borel-overt fiberwise quasi-Polish groupoid.
There is a Borel weakly componentwise $\*\Pi^0_2$ dense $Y \subseteq X$ such that for every $x \in Y$, the automorphism group $G(x,x)$ is a Polish group in both the $\sigma$- and $\tau$-fiberwise topologies.
Thus, every hom-set $G(x,y)$ is Polish in the $\sigma$-fiberwise topology if $y \in Y$, and in the $\tau$-fiberwise topology if $x \in Y$; and these topologies agree if both $x, y \in Y$.
Moreover, every automorphism group $G(x,x)$ for $x \in X$ is Polishable.
\end{corollary}
\begin{proof}
By \cref{thm:ucomqpolgpd-dense}, \cref{thm:ucomqpolgpd-homtop}, and the discussion in \cref{rmk:homtop}.
\end{proof}

The following shows that there is no distinction between the uniform and non-uniform notions of \emph{Borel-overt componentwise quasi-Polish groupoid} locally on each component.
Thus if a distinction exists, it must be due to a failure to witness continuous differences uniformly over all components.

\begin{proposition}
\label{thm:comqpolgpd-qpol-connected}
Let $(X,G)$ be a Borel-overt componentwise quasi-Polish \emph{connected} groupoid.
Then $G$ automatically has uniformly fiberwise continuous differences.
Thus, the componentwise topologies $\@O_G(X), \@O_G(G)$ turn $(X,G)$ into an open quasi-Polish groupoid, with $\@O_G(G)$ restricting to the original fiberwise topology.
\end{proposition}
(In fact the topology is unique in this case, by the Effros \cref{thm:effros} below.)
\begin{proof}
By \cref{thm:ucomqpolgpd-qpol}, it suffices to verify uniform fiberwise continuity of differences.
By \cref{thm:homtop-dense}, fix any $x_0 \in X$ such that $G(x_0,x_0)$ is a Polish group in both the $\sigma$- and $\tau$-fiberwise topologies.
Using \cref{thm:fibqpolgpd-idl} (for $G \actson \sigma^{-1}(x_0)$), pick a Borel family $k = (k_x : x_0 -> x)_{x \in X} : X -> \sigma^{-1}(x_0) \subseteq G$ of morphisms to every other object $x \in X$.
We must show that the multiplication
\begin{align*}
\tag{$*$}
\mu : (G \times_X G, \@{BO}_\tau(G) \otimes_X \@{BO}_\sigma(G)) &--> (G, \@{BO}_\sigma(G)) \\
(g,h) &|--> gh = k_{\tau(g)}(k_{\tau(g)}^{-1}g)(hk_{\sigma(h)})k_{\sigma(h)}^{-1} \\
\intertext{is continuous between these sigma-topologies.
By fiberwise continuity of differences \labelcref{def:comqpolgpd}\cref{def:comqpolgpd:mul},}
\mu : \tau^{-1}(x_0) \times_X \sigma^{-1}(x_0) &--> G(x_0,x_0) \\
\shortintertext{is continuous.
Hence}
G \times_X G &--> G(x_0,x_0) \\
(g,h) &|--> (k_{\tau(g)}^{-1}g)(hk_{\sigma(h)}) \\
\intertext{is continuous from the sigma-topology $\@{BO}_\tau(G) \otimes_X \@{BO}_\sigma(G)$ on $G \times_X G$ (using \cref{rmk:fibqpolgpd-bor}).
Now}
\mu : \sigma^{-1}(x_0) \times G(x_0,x_0) &--> \sigma^{-1}(x_0) \\
\intertext{is continuous from $\@O_\sigma(\sigma^{-1}(x_0)) \otimes \@O(G(x_0,x_0))$ to $\@O_\sigma(\sigma^{-1}(x_0))$, since it is a Borel action of a Polish group on a quasi-Polish space via homeomorphisms \cite[9.16]{kechris:cdst}, \cite[3.3.5]{chen:beckec}.
It follows that}
G \times_X G &--> \sigma^{-1}(x_0) \\
(g,h) &|--> k_{\tau(g)}(k_{\tau(g)}^{-1}g)(hk_{\sigma(h)})
\end{align*}
is continuous from $\@{BO}_\tau(G) \otimes_X \@{BO}_\sigma(G)$ to $\@O_\sigma(\sigma^{-1}(x_0))$, which easily then implies that ($*$) is continuous, by translating on the right by $k_{\sigma(h)}^{-1}$ again using \cref{rmk:fibqpolgpd-bor}.
\end{proof}

\begin{remark}
It follows that for a Borel-overt (not necessarily uniformly) componentwise quasi-Polish groupoid $(X,G)$, the global (non-Borel) componentwise topologies $\@O_G(X), \@O_G(G)$ form an open groupoid topology, namely the disjoint union of the quasi-Polish groupoid topologies on each component.
To see this, apply \cref{thm:comqpolgpd-qpol-connected} to the full subgroupoid on each component (which is Borel by \cref{thm:fibqpolgpd-er-idl}).
\end{remark}

\begin{corollary}
\label{thm:comqpolgpd-homtop}
Every Borel-overt componentwise quasi-Polish groupoid has $\*\Pi^0_2$ hom-sets.
\end{corollary}
\begin{proof}
By \cref{thm:ucomqpolgpd-homtop} applied to each component.
\end{proof}

\subsection{Polish topologies}

In this subsection, we characterize those componentwise quasi-Polish groupoids which admit a global \emph{Polish} topology.
The naive requirement, that the fiberwise topology be fiberwise Polish, turns out to be insufficient; see \cref{ex:fibpolgpd-nonpol}.

Instead, Polishability of groupoids is closely related to ``global uniformizability'' of the topology.
Recall that every topological group has a uniformizable, hence regular, topology; the basic fact needed to show this is the existence of rapidly decreasing ``fundamental sequences'' of identity neighborhoods.
In a groupoid $(X,G)$, while the same is easily shown around each individual identity morphism $1_x \in G$, the ability to find arbitrarily small neighborhoods of the set $\iota(X)$ of \emph{all} identities is a nontrivial result of Ramsay \cite[pp.~361--362]{ramsay:polgpd}, shown using a strong form of regularity.

We give a version of Ramsey's argument here, as we will need it later.
Recall that a Hausdorff space is \defn{paracompact} if for every open cover $\@U$, there exists an open cover $\@V$ which \defn{star-refines} $\@U$, meaning that for every $V \in \@V$, the union of all $V' \in \@V$ intersecting $V$ (called the \defn{$\@V$-star} of $V$) is contained in a single $U \in \@U$.
All metrizable spaces are paracompact.
See e.g., \cite[20.8, 20.15]{willard:topology}.

\begin{lemma}[Ramsay]
\label{thm:pcptgpd-idnbhd}
Let $(X,G)$ be a paracompact Hausdorff groupoid.
\begin{enumerate}[alph]
\item \label{thm:pcptgpd-idnbhd:halfnbhd}
For any $\iota(X) \subseteq W \in \@O(G)$, there is a (symmetric) $\iota(X) \subseteq W' \in \@O(G)$ with $W' \cdot W' \subseteq W$.
\item \label{thm:pcptgpd-idnbhd:basis}
The family $\set{W \cap \sigma^{-1}(Y) | \iota(X) \subseteq W \in \@O(G),\, Y \in \@O(X)}$ includes a neighborhood basis for every identity morphism $1_x \in G$.
\end{enumerate}
\end{lemma}
\begin{proof}
\cref{thm:pcptgpd-idnbhd:halfnbhd}
Since $g^{-1}g = 1_{\sigma(g)} \in W$ for all $g \in G$, by continuity,
\begin{eqaligned*}
\@U := \set{U \in \@O(G) | U^{-1} \cdot U \subseteq W}
\end{eqaligned*}
forms an open cover of $G$.
Let $\@V$ be a star-refinement, and put
\begin{eqaligned*}
W' := \bigcup_{V \in \@V} (V \cap V^{-1} \cap \sigma^{-1}(\iota^{-1}(V)) \cap \tau^{-1}(\iota^{-1}(V))).
\end{eqaligned*}
Then every identity $1_x$ is in some $V \in \@V$, whence
$1_x \in V \cap V^{-1} \cap \sigma^{-1}(\iota^{-1}(V)) \cap \tau^{-1}(\iota^{-1}(V)) \subseteq W'$.
For $g, h \in W'$ with $\sigma(g) = \tau(h)$, there are some $V, V' \in \@V$ with
$g \in V^{-1} \cap \sigma^{-1}(\iota^{-1}(V))$
and
$h \in V' \cap \tau^{-1}(\iota^{-1}(V'))$,
whence
$\iota(\sigma(g)) = \iota(\tau(h)) \in V \cap V'$;
since $\@V$ star-refines $\@U$, $V \cup V' \subseteq U$ for some $U \in \@U$, whence
$gh \in V^{-1} \cdot V' \subseteq U^{-1} \cdot U \subseteq W$.

\cref{thm:pcptgpd-idnbhd:basis}
Given any $1_x \in U \in \@O(X)$, take $\@U$ above to be the cover $\set{U, \neg \set{1_x}}$.
Then for $\@V, W'$ as above, for any $1_x \in V \in \@V$, we claim
$1_x \in W' \cap \sigma^{-1}(\iota^{-1}(V)) \subseteq U$.
Indeed, for any
$g \in W' \cap \sigma^{-1}(\iota^{-1}(V))$,
there is some $V' \in \@V$ with
$g \in V' \cap \sigma^{-1}(\iota^{-1}(V'))$,
whence
$\iota(\sigma(g)) \in V \cap V'$;
thus $W' \cap \sigma^{-1}(\iota^{-1}(V))$ is contained in the $\@V$-star of $V$, which must be contained in $U \in \@U$ since $1_x \in V$.
\end{proof}

\begin{remark}
If $G$ above is metrizable, then we may find a ``fundamental sequence'' $G = W_0 \supseteq W_1 \supseteq \dotsb \supseteq \iota(X)$, $W_{n+1}W_{n+1} \subseteq W_n$, whose restrictions to $\sigma^{-1}(\@O(X))$ form a neighborhood basis for every identity morphism, as in \cref{thm:pcptgpd-idnbhd:basis}.
(Apply the above inductively, and also take the $\@U$ at each stage to refine a countable basis of uniform open covers, e.g., balls of radius $2^{-n}$.)
This may in turn be used to prove a version of the Birkhoff--Kakutani metrization theorem; see \cite{buneci:gpdmet}, \cite[9.2]{chen:polgpdrep}.
\end{remark}

\begin{remark}
The converse of \cref{thm:pcptgpd-idnbhd} is false: there is a non-Hausdorff open quasi-Polish groupoid satisfying the conclusions.
Namely, take any discrete (equality) equivalence relation on a non-Hausdorff space.
However, the converse is true for open quasi-Polish connected groupoids, by \cref{thm:ucompolgpd} below and the Effros \cref{thm:effros}.
(See also \cite{buneci:gpdmet} for related results.)
\end{remark}

We now show that the analogous conditions to \labelcref{thm:pcptgpd-idnbhd}\cref{thm:pcptgpd-idnbhd:halfnbhd,thm:pcptgpd-idnbhd:basis} for \emph{componentwise} quasi-Polish groupoids yield a characterization of global Polishability.

\begin{lemma}[{cf.~\cite[3.6.4]{chen:beckec}}]
\label{thm:fibqpolgpd-idnbhd-reg}
Let $(X,G)$ be a Borel-overt fiberwise quasi-Polish groupoid, $U, V, W \in \@{BO}_G(G)$ with $\sigma(V) = \sigma(W) = X$.
Then the $\@O_G(G)$-closure of $U$ is contained in $V^{-1}VUW^{-1}W$, as witnessed by the Borel componentwise closed set
\begin{eqaligned*}
U \subseteq \neg (V^{-1} * \neg (VUW^{-1}) * W) \subseteq V^{-1}VUW^{-1}W.
\end{eqaligned*}
\end{lemma}
\begin{proof}
Note that $VUW^{-1} = V * U * W^{-1}$ is indeed Borel, by \cref{it:vaught-im},
and $\neg (V^{-1} * \neg (VUW^{-1}) * W)$ is componentwise closed, by \cref{eq:vaught-fibtop,thm:orbtop-exists*}.
The first containment follows from
\begin{eqaligned*}
\emptyset &= (VUW^{-1}) \cap \neg (VUW^{-1}) \\
\implies  \emptyset
&= U \cap (V^{-1} \cdot \neg (VUW^{-1}) \cdot W) \\
&\supseteq U \cap (V^{-1} * \neg (VUW^{-1}) * W)
    &&\text{by \cref{it:vaught-im}},
\end{eqaligned*}
while the second containment follows from
\begin{eqaligned*}[b][\qedhere]
G &= V^{-1}GW = V^{-1} * G * W
    &&\text{by $\sigma(V) = \sigma(W) = X$ and \cref{it:vaught-im}} \\
&= (V^{-1} * (VUW^{-1}) * W) \cup (V^{-1} * \neg (VUW^{-1}) * W)
    &&\text{by \cref{it:vaught-union}} \\
&= V^{-1}VUW^{-1}W \cup (V^{-1} * \neg (VUW^{-1}) * W)
    &&\text{by \cref{it:vaught-im}}.
\end{eqaligned*}
\end{proof}

\begin{theorem}
\label{thm:ucompolgpd}
Let $(X,G)$ be a Borel-overt uniformly componentwise quasi-Polish groupoid.
The following are equivalent:
\begin{enumerate}[roman]
\item \label{thm:ucompolgpd:reg}
$\@{BO}_G(G)$ is a regular sigma-topology (hence so is $\@{BO}_G(X)$).
\item \label{thm:ucompolgpd:dblnbhd}
The family $\set{W^{-1} W | W \in \@{BO}_G(G),\, \sigma(W) = X} \ocap \sigma^{-1}(\@{BO}_G(X))$ includes, for any $U \in \@{BO}_G(G)$, a subset of $U$ containing $U \cap \iota(X)$.
\item \label{thm:ucompolgpd:halfnbhd}
For any $\iota(X) \subseteq W \in \@{BO}_G(G)$, there is a (symmetric) $\iota(X) \subseteq W' \in \@{BO}_G(G)$ with $W' \cdot W' \subseteq W$; and the sigma-topology generated by all such $W$ together with $\sigma^{-1}(\@{BO}_G(X))$ includes, for any $U \in \@{BO}_G(G)$, a subset of $U$ containing $U \cap \iota(X)$.
\item \label{thm:ucompolgpd:pol}
There are compatible Polish topologies $\@O(X) \subseteq \@{BO}_G(X)$ and $\@O(G) \subseteq \@{BO}_G(G)$ turning $(X,G)$ into an open Polish groupoid, such that $\@O(G)$ restricts to the original fiberwise topology.
\item \label{thm:ucompolgpd:club}
A club of compatible open quasi-Polish groupoid topologies $\@O(G) \subseteq \@{BO}_G(G)$ restricting to the original fiberwise topology as in \cref{thm:ucomqpolgpd-qpol} are Polish.
\end{enumerate}
If these hold, we call $(X,G)$ a \defn{Borel-overt uniformly componentwise \emph{Polish} groupoid}.
\end{theorem}

(The generation conditions in \cref{thm:ucompolgpd:dblnbhd,thm:ucompolgpd:halfnbhd} mean, morally speaking, that the sets $W \cap \sigma^{-1}(Y)$ for $Y \in \@{BO}_G(X)$ form a neighborhood basis for each identity morphism $1_x$ as in \cref{thm:pcptgpd-idnbhd}\cref{thm:pcptgpd-idnbhd:basis}; but here we must say this ``uniformly over all $x$'' as $\@{BO}_G(G)$ is not second-countable.)

\begin{proof}
Note that $\iota, \sigma$ exhibit $X$ as a continuous retract of $G$; thus regularity of $\@{BO}_G(G)$ (respectively $\@O(G)$) implies that of $\@{BO}_G(X)$ (respectively $\@O(X)$).

\cref{thm:ucompolgpd:reg}$\implies$\cref{thm:ucompolgpd:club}
by \cref{thm:ucomqpolgpd-qpol,thm:stop-pol}.

\cref{thm:ucompolgpd:club}$\implies$\cref{thm:ucompolgpd:pol}
is obvious.

\cref{thm:ucompolgpd:club}$\implies$\cref{thm:ucompolgpd:halfnbhd}
follows immediately from \cref{thm:pcptgpd-idnbhd}, by making the given sets in $\@{BO}_G(G)$ open.

\cref{thm:ucompolgpd:halfnbhd}$\implies$\cref{thm:ucompolgpd:dblnbhd}
is straightforward (take $W$ in \cref{thm:ucompolgpd:dblnbhd} to be $W' \cap W^{\prime-1}$ for $W'$ from \cref{thm:ucompolgpd:halfnbhd}).

\cref{thm:ucompolgpd:pol}$\implies$\cref{thm:ucompolgpd:dblnbhd}
follows similarly to the preceding two parts, using \cref{thm:ucomqpolgpd-action-orbtop} which shows that in \cref{thm:ucompolgpd:dblnbhd} it is enough to consider $U = V \cap \sigma^{-1}(Y) \in \@{BO}_G(G)$ where $V \in \@O(G)$ and $Y \in \@{BO}_G(X)$;
by \cref{thm:pcptgpd-idnbhd} we may write
$V \cap \iota(X) \subseteq \bigcup_i (W_iW_i \cap \sigma^{-1}(Z_i)) \subseteq V$
for symmetric $\iota(X) \subseteq W_i \in \@O(G)$, whence
$U \cap \iota(X) \subseteq \bigcup_i (W_i^{-1}W_i \cap \sigma^{-1}(Z_i \cap Y)) \subseteq U$.

\cref{thm:ucompolgpd:dblnbhd}$\implies$\cref{thm:ucompolgpd:reg}:
Let $U \in \@{BO}_G(G)$; we must show that $U$ is a countable union of $U'_i \in \@{BO}_G(G)$ which are contained in some Borel componentwise closed sets contained in $U$.
By \labelcref{thm:ucomqpolgpd}\cref{thm:ucomqpolgpd:mu-orb},
\begin{eqaligned*}
\set{(h,g,k) \in G \times_X G \times_X G | hgk \in U} &= \bigcup_i (A_i \times_X U_i \times_X B_i)
\end{eqaligned*}
for countably many $A_i, U_i, B_i \in \@{BO}_G(G)$.
By \cref{thm:ucompolgpd:dblnbhd} and its dual, we may find
\begin{align*}
A_i \cap \iota(X) &\subseteq \bigcup_j (V_j^{-1} V_j \cap \sigma^{-1}(Y_j)) \subseteq A_i, \\
B_i \cap \iota(X) &\subseteq \bigcup_k (W_k^{-1} W_k \cap \tau^{-1}(Z_k)) \subseteq B_i
\end{align*}
where $V_j, W_k \in \@{BO}_G(G)$ with $\sigma(V_j) = \sigma(W_k) = X$ and $Y_j, Z_k \in \@{BO}_G(X)$.
It follows that
\begin{eqaligned*}
\set{(1_y,g,1_z) | g : z -> y \in U}
&\subseteq \bigcup_{i,j,k} \paren[\big]{
        (V_j^{-1} V_j \cap \sigma^{-1}(Y_j))
        \otimes_X U_i
        \otimes_X (W_k^{-1} W_k \cap \tau^{-1}(Z_k))
    } \\
&= \bigcup_{i,j,k} \paren[\big]{
        V_j^{-1} V_j \otimes_X (\tau^{-1}(Y_j) \cap U_i \cap \sigma^{-1}(Z_k)) \otimes_X W_k^{-1} W_k
    } \\
&\subseteq \set{(h,g,k) \in G \times_X G \times_X G | hgk \in U},
\end{eqaligned*}
whence projecting to the middle coordinate, and also multiplying, yields
\begin{eqaligned*}
U
&\subseteq \bigcup_{i,j,k} (\tau^{-1}(Y_j) \cap U_i \cap \sigma^{-1}(Z_k))
\subseteq \bigcup_{i,j,k} \paren[\big]{
        V_j^{-1} V_j (\tau^{-1}(Y_j) \cap U_i \cap \sigma^{-1}(Z_k)) W_k^{-1} W_k
    }
\subseteq U.
\end{eqaligned*}
By the preceding lemma, each $U'_{ijk} := \tau^{-1}(Y_j) \cap U_i \cap \sigma^{-1}(Z_k) \in \@{BO}_G(G)$ is contained in a Borel componentwise closed set which is contained in $V_j^{-1} V_j U'_{ijk} W_k^{-1} W_k \subseteq U$, as desired.
\end{proof}

\begin{remark}
Note the non-obvious implication \labelcref{thm:ucompolgpd}\cref{thm:ucompolgpd:pol}$\implies$\cref{thm:ucompolgpd:club}: if $(X,G)$ admits at least one compatible open Polish groupoid topology, then cofinally many topologies are Polish.

This implication may be shown directly using the version of the Becker--Kechris theorem for open Polish groupoids due to Lupini \cite{lupini:polgpd}.
Conversely, we may easily deduce this version of Becker--Kechris from \cref{thm:ucompolgpd}, thereby showing that as in the quasi-Polish context (\cref{thm:beckec}), topological realization for Polish groupoids ``subsumes'' that for actions:
\end{remark}

\begin{lemma}
\label{thm:ucompolgpd-action}
If $(X,G)$ is a Borel-overt uniformly componentwise Polish groupoid, then so is $G \ltimes M$ with the lifted fiberwise topology, for any Borel $G$-space $p : M -> X$.
\end{lemma}
\begin{proof}
By \cref{thm:comqpolgpd-action,thm:ucomqpolgpd-action-orbtop}, $G \ltimes M$ is Borel-overt uniformly componentwise quasi-Polish, with $\@{BO}_G(G \ltimes M) = \@{BO}_G(G) \otimes_X \@{BO}_G(M)$; it is easily seen that for the $W \in \@{BO}_G(G)$ witnessing \labelcref{thm:ucompolgpd}\cref{thm:ucompolgpd:dblnbhd} for $G$, the corresponding $W \times_X M \in \@{BO}_G(G \ltimes M)$ witness \labelcref{thm:ucompolgpd}\cref{thm:ucompolgpd:dblnbhd} for $G \ltimes M$.
\end{proof}

\begin{corollary}[{Becker--Kechris theorem for open Polish groupoid actions \cite[4.1.1]{lupini:polgpd}}]
\label{thm:beckec-pol}
\leavevmode\par\noindent
Let $(X,G)$ be an open Polish groupoid, $p : M -> X$ be a Borel $G$-space.
Then for a club of compatible Polish topologies $\@O(M) \subseteq \@{BO}_G(M)$, the action $G \actson (M,\@O(M))$ is continuous.
\end{corollary}
\begin{proof}
As in \cref{thm:beckec}, but take preimage under $\@O(G) \otimes_X (-)$ of the club of Polish topologies on $G \ltimes M$ from \cref{thm:ucompolgpd} instead.
\end{proof}

If we allow the freedom to pass to componentwise comeager subgroupoids as in \cref{thm:fibqpolgpd-realiz-equiv}, then rather surprisingly, the distinction between quasi-Polishability and Polishability completely disappears, yielding a significant strengthening of that earlier result.

\begin{theorem}
\label{thm:fibqpolgpd-pol}
Let $(X,G)$ be a Borel-overt fiberwise quasi-Polish groupoid.
Then there is a Borel weakly componentwise dense $\*\Pi^0_2$ set $Y \subseteq X$ such that the full subgroupoid $(Y,\eval{G}_Y)$ (with the fiberwise subspace topology) is Borel-overt uniformly componentwise Polish.
\end{theorem}
\begin{proof}
By \cref{thm:fibqpolgpd-realiz-equiv}, we may assume $G$ is uniformly componentwise quasi-Polish.
The idea is to ``realize \labelcref{thm:pcptgpd-idnbhd}\cref{thm:pcptgpd-idnbhd:halfnbhd,thm:pcptgpd-idnbhd:basis}'' on a further comeager set of objects.

\cref{thm:pcptgpd-idnbhd:basis}
Given any $U \in \@{BO}_G(G)$, clearly
$W := U \cup \sigma^{-1}(X \setminus \iota^{-1}(U)) \subseteq G$
has
$\iota(X) \subseteq W$
and
$W \cap \sigma^{-1}(\iota^{-1}(U))
= U \cap \sigma^{-1}(\iota^{-1}(U))
\supseteq U \cap \iota(X)$.
Now $X \setminus \iota^{-1}(U) \subseteq X$ might not be componentwise open;
but by the componentwise Baire property
(\cref{thm:orbtop-bp}; we may also just take the componentwise interior, i.e., $\exists^*_\tau$ of the complement of the $\sigma$-fiberwise closure of $\tau^{-1}(\iota^{-1}(U))$ as in \cref{eq:fib-bor-closure}),
there is a Borel componentwise $\*\Pi^0_2$ dense $Y \subseteq X$ on which $X \setminus \iota^{-1}(U)$ becomes componentwise open,
hence $W$ becomes componentwise open when restricted to $\eval{G}_Y$.

\cref{thm:pcptgpd-idnbhd:halfnbhd}
Given any $\iota(X) \subseteq W \in \@{BO}_G(G)$, by continuity of $\mu$ with respect to $\@{BO}_G(G)$ (\labelcref{thm:ucomqpolgpd}\cref{thm:ucomqpolgpd:mu-orb}),
\begin{eqaligned*}
G = \set{g \in G | g^{-1}g \in W} = \bigcup_{i \in \#N} U_i
\end{eqaligned*}
for countably many $U_i \in \@{BO}_G(G)$ such that $U_i^{-1} \cdot U_i \subseteq W$.
By the proof of \cref{thm:pcptgpd-idnbhd}\cref{thm:pcptgpd-idnbhd:halfnbhd}, if $\@V$ is a star-refinement of the cover $\set{U_i}_i$, then the union of the symmetric unital parts of $V \in \@V$ yields a smaller neighborhood $W' \supseteq \iota(X)$ with $W' \cdot W' \subseteq W$.
In fact, that argument did not require $\@V$ to form a star-refinement, but only that any two sets in $\@V$ containing the same identity are contained in a single $U_i$; moreover it did not require $\@V$ to cover $G$, but only $\iota(X)$.
So put
\begin{eqgathered*}
V_i := U_i \setminus \bigcup_{j < i} \sigma^{-1}(\iota^{-1}(U_j)), \\
W' := \bigcup_{i \in \#N} (V_i \cap V_i^{-1} \cap \sigma^{-1}(\iota^{-1}(V_i)) \cap \tau^{-1}(\iota^{-1}(V_i)));
\end{eqgathered*}
then $\iota(X) \subseteq W' = W^{\prime-1}$ and $W' \cdot W' \subseteq W$ as in the proof of \cref{thm:pcptgpd-idnbhd}\cref{thm:pcptgpd-idnbhd:halfnbhd}.
As above, we may pass to a Borel componentwise $\*\Pi^0_2$ dense $Y \subseteq X$ on which each $X \setminus \bigcup_{j < i} \iota^{-1}(U_j)$ becomes componentwise open, hence $W'$ becomes componentwise open in $\eval{G}_Y$.

Now fix by \cref{thm:ucomqpolgpd-qpol} some open quasi-Polish groupoid topology $(\@O(X),\@O(G))$ inducing the given fiberwise topology, and fix a countable basis $\@U \subseteq \@O(G)$.
For each $U \in \@U$, find as in step \cref*{thm:pcptgpd-idnbhd:basis} above some
$\iota(X) \subseteq W_U \subseteq G$,
and then find as in step \cref*{thm:pcptgpd-idnbhd:halfnbhd} some symmetric
$\iota(X) \subseteq W'_U \subseteq G$
with
$W'_U \cdot W'_U \subseteq W_U$,
such that $W_U, W'_U$ become Borel componentwise open and
$W_U \ocap \sigma^{-1}(\@{BO}_G(Y))$
includes a subset of $U$ containing the same identity morphisms after restricting to a single Borel componentwise $\*\Pi^0_2$ dense $Y \subseteq X$.
Then we may directly verify \labelcref{thm:ucompolgpd}\cref{thm:ucompolgpd:dblnbhd} for $\eval{G}_Y$ by imitating the proof of \labelcref{thm:ucompolgpd}\cref{thm:ucompolgpd:pol}$\implies$\cref{thm:ucompolgpd:dblnbhd}.
Indeed, a basic Borel componentwise open subset of $\eval{G}_Y$ is (by \cref{rmk:fibqpolgpd-full}) the restriction of one in $G$, hence by \cref{thm:ucomqpolgpd-action-orbtop} (applied to $G$) is of the form
\begin{eqaligned*}
\eval{G}_Y \cap U \cap \sigma^{-1}(Z) \in \@{BO}_G(\eval{G}_Y)
\end{eqaligned*}
for some $U \in \@U$ and $Z \in \@{BO}_G(X)$.
By the above,
\begin{eqaligned*}
U \cap \iota(Y)
&\subseteq (\eval{G}_Y \cap W'_U)^{-1}(\eval{G}_Y \cap W'_U) \cap \sigma^{-1}(Z_U)
\subseteq W_U \cap \sigma^{-1}(Z_U)
\subseteq U
\end{eqaligned*}
for some $Z_U \in \@{BO}_G(Y)$, whence
\begin{eqaligned*}
\eval{G}_Y \cap U \cap \sigma^{-1}(Z) \cap \iota(Y)
\subseteq (\eval{G}_Y \cap W'_U)^{-1}(\eval{G}_Y \cap W'_U) \cap \sigma^{-1}(Z \cap Z_U)
\subseteq \eval{G}_Y \cap U \cap \sigma^{-1}(Z)
\end{eqaligned*}
for $\iota(Y) \subseteq \eval{G}_Y \cap W'_U \in \@{BO}_G(\eval{G}_Y)$ and $Z \cap Z_U \in \@{BO}_G(Y)$ as required by \labelcref{thm:ucompolgpd}\cref{thm:ucompolgpd:dblnbhd}.
\end{proof}

\begin{theorem}
\label{thm:fibqpolgpd-realiz-equiv-pol}
Every Borel-overt fiberwise quasi-Polish groupoid $(X,G)$ admits a Borel equivalence of groupoids to an open Polish groupoid, namely a Borel weakly componentwise $\*\Pi^0_2$ dense full subgroupoid $(Y,\eval{G}_Y)$ equipped with a Polish topology restricting to its fiberwise topology.
\end{theorem}
\begin{proof}
By \cref{thm:fibqpolgpd-pol,thm:ucompolgpd}.
\end{proof}

\begin{corollary}
\label{thm:fibqpolgpd-polact}
Every Borel-overt fiberwise quasi-Polish groupoid $(X,G)$ admits a Borel equivalence of groupoids to an action groupoid of a Polish group action.
\end{corollary}
\begin{proof}
Combine the preceding result with \cite[1.2]{chen:polgpdrep}.
\end{proof}

\begin{corollary}
\label{thm:bocqper-polact}
Every Borel-overt classwise quasi-Polish equivalence relation $(X,E)$ is Borel bireducible with the orbit equivalence relation of a free Polish group action.
\end{corollary}
\begin{proof}
Being an equivalence relation transfers along an equivalence of groupoids; and an action groupoid is an equivalence relation iff the action is free.
\end{proof}

\begin{remark}
\label{rmk:fibpolgpd-realiz}
We sketch a simpler proof of \cref{thm:fibqpolgpd-realiz-equiv-pol}, hence also \cref{thm:fibqpolgpd-polact,thm:bocqper-polact}, not using the ``uniformization'' machinery of \cref{thm:ucompolgpd}, in the special case when the original fiberwise topology on $G$ is fiberwise Polish (as originally stated in the introduction).
By \cref{thm:fibqpolgpd-realiz-equiv}, we may assume $G$ is an open quasi-Polish groupoid.
For each $U_i \in \@O(G)$ in a countable basis, $U_i$ is $\sigma$-fiberwise Polish, hence by \cref{thm:fib-bor-pol-reg} is a countable union of Borel $\sigma$-fiberwise open $V_{ij}$ which are contained in Borel $\sigma$-fiberwise closed $F_{ij} \subseteq U_i$.
By Kunugi--Novikov \cref{eq:fib-bor-qpol-basis}, $\sigma^{-1}(B_{ijk})$ for countably many Borel $B_{ijk} \subseteq X$ need to be adjoined to $\@O(G)$ to make these $V_{ij},F_{ij}$ genuinely open and closed.
By the componentwise Baire property (\labelcref{thm:orbtop-bp}), there is a Borel componentwise $\*\Pi^0_2$ dense $Y \subseteq X$ on which these $B_{ijk}$ become componentwise open, hence become open in a finer compatible open quasi-Polish groupoid topology on $(Y,\eval{G}_Y)$.
Now repeat this procedure $\omega$ times; the resulting subgroupoid with a finer topology is easily seen to be open Polish.
\end{remark}

Finally, we give a counterexample to show that the conditions in \cref{thm:ucompolgpd} characterizing global Polishability of the \emph{original} groupoid cannot be simplified to ``fiberwise Polish'', and must necessarily involve interaction between the fiberwise (or componentwise) topology and the algebraic structure of the groupoid, as in \labelcref{thm:ucompolgpd}\cref{thm:ucompolgpd:dblnbhd,thm:ucompolgpd:halfnbhd}.

\begin{example}
\label{ex:fibpolgpd-nonpol}
Let $G$ be the set of partial bijections of $\#N$ with infinite domain and range, topologized by treating their graphs as points in the Sierpiński product space $\#S^{\#N \times \#N}$; thus
\begin{eqaligned*}
G = \set*{g \in \#S^{\#N \times \#N} | \begin{gathered}
    \forall i\, \forall j \ne k\, \paren[\big]{((i,j) \notin g \OR (i,k) \notin g) \AND ((k,i) \notin g \OR (k,j) \notin g)}, \\
    \forall i\, \exists j, k > i\, ((j,k) \in g)
\end{gathered}}
\end{eqaligned*}
is a $\*\Pi^0_2$ subspace of $\#S^{\#N \times \#N}$, hence quasi-Polish.
Let $X$ be the set of infinite subsets of $\#N$, which is a $\@G_\delta$ subspace of $\#S^\#N$.
Then $(X,G)$ with the usual domain $\sigma$, range $\tau$, composition, identity, and inverse operations is easily seen to be an open quasi-Polish groupoid; openness is because the $\sigma$-image of the basic open set of all $g \in G$ extending a finite partial bijection $\gamma$ is easily seen to be the set of all infinite subsets containing the domain of $\gamma$.

For any fixed infinite subset $x \in X$, the fiber $\sigma^{-1}(x)$ is the space of injections $x `-> \#N$.
The subspace topology induced from $G$ is the same as the zero-dimensional product topology on $\#N^x$, since for $g \in \sigma^{-1}(x)$, to say that $(i,j) \notin g$ is equivalently to say that $(i,k) \in g$ for some $k \ne j$.
Thus, the $\sigma$-fiberwise topology on $G$ is fiberwise zero-dimensional Polish.

However, there does not exist any open Polish groupoid topology on $G$ with the same $\sigma$-fiberwise restriction.
This is most easily seen using the Effros \cref{thm:effros} below: since $G$ is a connected groupoid, the unique open quasi-Polish groupoid topology on it inducing the same $\sigma$-fiberwise topology is its original topology defined above, which is not Polish.

For a direct argument, note that the componentwise topology $\@O_G(X)$ is the same as the original quasi-Polish topology on $X$: by \cref{rmk:orbtop-openquot}, for any $x \in X$, we have a continuous open surjection $\tau : \sigma^{-1}(x) ->> X$ from the $\sigma$-fiberwise topology to the componentwise topology $\@O_G(X)$; and the $\tau$-image of the basic open set of all $g$ with domain $x$ extending a finite $\gamma$ is again easily seen to be all infinite subsets containing the range of $\gamma$.
Thus, any compatible open Polish groupoid topology on $(X,G)$ inducing the same $\@O_\sigma(G)$ must have $\@O(X) \subseteq \@O_G(X)$.
But $\@O_G(X)$ is a non-$T_1$ topology; indeed, no $x \in X$ is a closed point (its closure consists of all infinite subsets $y \subseteq x$).
Thus no topology contained in $\@O_G(X)$ can be $T_1$ either.
\end{example}

\subsection{Miscellaneous results}
\label{sec:misc}

In the rest of the paper, we record the groupoid generalizations of some other standard results about Polish groups and their actions.
Most of these have appeared before in the literature on groupoids in some context, but not for open quasi-Polish groupoids.
Several of these results also have natural reformulations in the fiberwise framework of this paper.

\begin{remark}
\label{rmk:gpd-sub-closure}
Recall that every Polish (equivalently $\*\Pi^0_2$) subgroup $H \subseteq G$ of a Polish group $G$ is closed, by Baire category; see \cite[2.2.1]{gao:idst}.
Indeed, otherwise $H$ and a coset of it would be disjoint $\*\Pi^0_2$ sets dense in $\-H$, or alternatively, $H = H * H = \-H * \-H = \-H$ by Pettis's theorem (for $\-H$).

It follows that for a Borel-overt componentwise quasi-Polish groupoid $(X,G)$, for a fiberwise $\*\Pi^0_2$ subgroupoid $(Y,H) \subseteq (X,G)$, for every $y \in Y$, the automorphism group $H(y,y)$ is a closed subgroup of $G(y,y)$ (in the Polish group topology on $G(y,y)$ from \cref{rmk:homtop,thm:ucomqpolgpd-homtop}); hence more generally every hom-set $H(x,y) \subseteq G(x,y)$ is closed.
Note also that $H \subseteq G$ being fiberwise $\*\Pi^0_2$ implies that $Y \subseteq X$ is componentwise $\*\Pi^0_2$, by \cref{rmk:orbtop-openquot,thm:qpol-openquot-bor}.

We note that some natural strengthenings of these statements are false, even if $G$ is assumed to be an open quasi-Polish groupoid and $H$ is also assumed to be an open groupoid:
\begin{itemize}

\item
It is not necessarily the case that $H \subseteq G$ is $\sigma$-fiberwise closed: consider the indiscrete equivalence relation $G = X^2$ on a quasi-Polish space $X$, and $H = Y^2$ for an open $Y \subseteq X$.
For concreteness, say $H := (0,1)^2 \subseteq \#R^2 =: G$.
Note that this example is a full subgroupoid.

\item
Assuming alternatively that $Y = X$, it is again not necessary that $H \subseteq G$ is $\sigma$-fiberwise closed: consider $G = \#R^2$, and $H = (-\infty,0)^2 \cup \set{(0,0)} \cup (0,\infty)^2$.

\end{itemize}
\end{remark}

The following seems to be the best possible groupoid analog of the fact for Polish subgroups, and was shown by Johnstone \cite{johnstone:locgpd} for localic groupoids with essentially the same proof as below.

\begin{proposition}[closed subgroupoid theorem]
Let $(X,G)$ be a Borel-overt componentwise quasi-Polish groupoid, $(Y,H)$ be an arbitrary subgroupoid.
Let $\-H \subseteq G$ be the intersection of the $\sigma$- and $\tau$-fiberwise closures of $H$.
Then $(Y,\-H)$ is a subgroupoid of $(X,G)$.
If moreover $H \subseteq G$ is $\sigma$-fiberwise $\*\Pi^0_2$, then $H = \-H$.
\end{proposition}
(In fact the first claim holds for an arbitrary componentwise topological groupoid; neither Borel, nor overt, nor quasi-Polish is needed.
For the second claim, we need only assume that $G$ is fiberwise completely Baire.
We state the result in the restricted context of \cref{def:fibqpolgpd} for simplicity.)
\begin{proof}
For each $x, y \in X$, the multiplication $\mu : \-H \times_X \-H -> G$ lands in $\-H$: indeed, the restriction to each $\mu^{-1}(G(x,y))$ has domain contained in the product of fiberwise closures $\-{\tau_H^{-1}(y)} \times_X \-{\sigma_H^{-1}(x)}$, hence lands in $\-{\sigma_H^{-1}(x)}$ by fiberwise continuity of differences (see \labelcref{def:comqpolgpd}\cref{def:comqpolgpd:mul}), and dually lands in $\-{\tau_H^{-1}(y)}$, hence in $\-H$.
The rest of the groupoid operations are easily seen to also restrict to $(Y,\-H)$.

Now suppose that $H$ is fiberwise $\*\Pi^0_2$ (and $G$ is fiberwise completely Baire).
Then $H \subseteq \-H$ is fiberwise comeager, hence $H = H * H = \-H * \-H = \-H$ by Pettis's theorem \cref{it:vaught-pettis} (whose proof only uses that $\-H$ is fiberwise Baire, since it contains the fiberwise dense $\*\Pi^0_2$ $H$, using that $G$ is fiberwise completely Baire).
\end{proof}

The classical Effros theorem on orbits of Polish group actions \cite{effros:groups} (see also \cite[3.2.4]{gao:idst}) was generalized to open Polish groupoid actions by Ramsay \cite[\S3]{ramsay:polgpd} and Lupini \cite[\S3.1]{lupini:polgpd}, but does not appear to have been verified before for open quasi-Polish groupoids.
We do so here.

\begin{theorem}[Effros theorem]
\label{thm:effros}
Let $(X,G)$ be a Borel-overt componentwise quasi-Polish groupoid.
Then for each component $C \in X/G$, the componentwise topologies $\@O_G(C)$ and $\@O_G(\eval{G}_C)$ are the unique second-countable topologies $\@O(C) \subseteq \@B(C)$ and $\@O(\eval{G}_C) \subseteq \@B(G)$ turning $(C,\eval{G}_C)$ into an open topological groupoid, such that
\begin{enumerate}[roman]
\item  \label{thm:effros:fibtop}
$\@O(\eval{G}_C)$ induces the original $\sigma$-fiberwise topology $\@O_\sigma(\eval{G}_C)$;
\item  \label{thm:effros:compat}
$\@O(C)$ generates the Borel sigma-algebra $\@B(C)$; and
\item  \label{thm:effros:nonmeager}
$C$ is $\@O(C)$-nonmeager in itself.
\end{enumerate}
(For example, the first two conditions hold if $(X,G)$ is open quasi-Polish and $\@O(C), \@O(\eval{G}_C)$ are the subspace topologies, while the last two conditions hold if $\@O(C)$ is quasi-Polish.)
\end{theorem}
\begin{proof}
By restricting to the full subgroupoid $(C,\eval{G}_C)$, we may assume $C = X$, i.e., $G$ is connected.

First, note that $\@O_G(X), \@O_G(G)$ do indeed form an open quasi-Polish groupoid, by \cref{thm:comqpolgpd-qpol-connected}.
And any other open groupoid topology $\@O(X), \@O(G)$ obeying the above conditions must be contained in $\@O_G(X), \@O_G(G)$, since $\tau^{-1}(\@O(X)) \subseteq \@O(G) \subseteq \@O_\sigma(G)$ by \cref{thm:effros:fibtop}, and $\mu^{-1}(\@O(G)) \subseteq \@O(G) \otimes_X \@O(G) \subseteq \@O_\sigma(G) \otimes_X \@O_\tau(G)$.
So it remains to show $\@O_G(X) \subseteq \@O(X)$ and $\@O_G(G) \subseteq \@O(G)$.

We first show that any Borel $\@O(X)$-meager $A \subseteq X$ is $\@O_G(X)$-meager.
Indeed, $\tau^{-1}(A) \subseteq G$ is $\@O(G)$-meager since $\tau : (G,\@O(G)) -> (X,\@O(X))$ is continuous open.
Thus, $\exists^*_\sigma(\tau^{-1}(A)) \subseteq X$ is $\@O(X)$-meager by the Kuratowski--Ulam \cref{thm:kuratowski-ulam} and \cref{thm:effros:fibtop}.
But since $\tau^{-1}(A) \subseteq G$ is invariant under right multiplication, $\exists^*_\sigma(\tau^{-1}(A)) \subseteq X$ is $G$-invariant, hence must be $\emptyset$ by \cref{thm:effros:nonmeager}.
This means $\tau^{-1}(A)$ is $\sigma$-fiberwise meager, i.e., $A$ is $\@O_G(X)$-meager (\cref{def:orbtop-weak,thm:comqpolgpd-homtop}).

Now for $A \in \@{BO}_G(X)$,
\begin{eqaligned*}
\tau^{-1}(A) = \bigcup_i (\sigma^{-1}(B_i) \cap U_i)
\end{eqaligned*}
for $B_i \in \@B(X)$ and $U_i \in \@O(G)$ by \cref{eq:orbtop-obj-bor}, \cref{thm:effros:fibtop}, and Kunugi--Novikov.
By the Baire property and \cref{thm:effros:compat}, find $B_i =^*_{\@O(X)} B'_i \in \@O(X)$.
By the above, also $B_i =^*_{\@O_G(X)} B'_i$, so by \cref{def:orbtop-weak},
\begin{eqaligned*}
\tau^{-1}(A) =^*_\tau \bigcup_i (\sigma^{-1}(B'_i) \cap U_i) \in \@O(G),
\end{eqaligned*}
and so
$A = \exists^*_\tau(\tau^{-1}(A)) \in \exists^*_\tau(\@O(G)) = \tau(\@O(G)) = \@O(X)$.

Similarly for $U \in \@{BO}_G(G)$, write $\mu^{-1}(U) = \bigcup_i (V_i \times_X W_i)$ for $V_i \in \@{BO}_\sigma(G)$ and $\@{BO}_\tau(G)$, and replace them mod $\tau$- and $\sigma$-fiberwise meager respectively with $V'_i, W'_i \in \@O(G)$ as above; then $U = \bigcup_i (V_i * W_i) = \bigcup_i (V'_i \cdot W'_i) \in \@O(G)$ by Pettis \cref{it:vaught-pettis}.
\end{proof}

\begin{corollary}
\label{thm:effros-action}
For an open quasi-Polish groupoid $(X,G)$ and Borel $G$-space $p : M -> X$, the orbitwise topology $\@O_G(C)$ on each $C \in M/G$ is the unique second-countable topology $\@O(C)$ generating $\@B(C)$ and making $p : (C,\@O(C)) -> (X,\@O(X))$ and the action $\alpha : (G,\@O(G)) \times_X (C,\@O(C)) -> (C,\@O(C))$ continuous and $C$ nonmeager in itself.

Similarly, for a Borel-overt componentwise quasi-Polish $(X,G)$ and Borel $G$-space $p : M -> X$, the orbitwise topology $\@O_G(C)$ on each $C \in M/G$ is the unique second-countable topology $\@O(C)$ generating $\@B(C)$ and making $p : (C,\@O(C)) -> (X,\@O_G(X))$ and $\alpha : (G,\@O_G(G)) \times_X (C,\@O(C)) -> (C,\@O(C))$ continuous and $C$ nonmeager in itself.
\end{corollary}
\begin{proof}
Follows by applying \cref{thm:effros} to the action groupoid $G \ltimes M$, and noting that when the action on $C$ is continuous, then $\@O(C)$ together with $\@O(G) \otimes_X \@O(C)$, respectively $\@O_G(G) \otimes_X \@O(C)$, form the required open groupoid topology on $G \ltimes C$.
\end{proof}

An important application of the Effros theorem is to show that for groupoids or actions inducing a \emph{smooth} equivalence relation, the componentwise topologies may be realized simultaneously on each component.
(Smoothness is clearly necessary for this, by \cref{it:qpol-openquot}.)
The following result is one equivalent formulation of the quasi-Polish generalization of \cite[7.1.2]{becker-kechris:polgrp} for Polish group actions and \cite[5.2.2]{lupini:polgpd} for Polish groupoids (see \cref{thm:ucomqpolgpd-erquot} for the direct generalization).

\begin{theorem}
\label{thm:ucomqpolgpd-smooth}
Let $(X,G)$ be a Borel-overt uniformly componentwise quasi-Polish groupoid, and suppose its connectedness relation $\#E_G \subseteq X^2$ is smooth, i.e., $X/G$ is standard Borel (\cref{thm:fibqpolgpd-idl}).
Then the quotient map $\pi : X ->> X/G$ equipped with the fiberwise topology $\@{BO}_G(X)$ is a Borel-overt bundle of quasi-Polish spaces; similarly for $\pi\sigma = \pi\tau : G ->> X/G$ with $\@{BO}_G(G)$.

Thus, for a club of pairs $(\@O(X),\@O(G))$ of quasi-Polish topologies below $(\@{BO}_G(X),\@{BO}_G(G))$ forming an open quasi-Polish groupoid topology inducing the given fiberwise topology, as in \cref{thm:ucomqpolgpd-qpol}, we have that $\@O(X), \@O(G)$ restrict to the componentwise topologies on each component:
\begin{eqaligned*}
\@{BO}_G(X) &= \pi^{-1}(\@B(X/G)) \ocap \@O(X), \\
\@{BO}_G(G) &= \sigma^{-1}(\pi^{-1}(\@B(X/G))) \ocap \@O(G) = \tau^{-1}(\pi^{-1}(\@B(X/G))) \ocap \@O(G).
\end{eqaligned*}

Moreover, for any Borel transversal $Y \subseteq X$ of $\#E_G$, the surjection
\begin{eqaligned*}
\tau : \sigma^{-1}(Y) -->> X
\end{eqaligned*}
is ``uniformly open from $\@O_\sigma$ to $\@O_G$'', i.e., preserves Borelness of fiberwise open sets:
\begin{eqaligned*}
\tau(\@{BO}_\sigma(\sigma^{-1}(Y))) = \@{BO}_G(X);
\end{eqaligned*}
and there is a Borel section $\gamma : X `-> \sigma^{-1}(y)$ of $\tau : \sigma^{-1}(Y) -> X$, picking a unique morphism $\gamma(x) : y -> x$ from the unique $y \in Y \cap [x]_G$ to each $x \in X$.
\end{theorem}
\begin{proof}
We prove the second claim first.
Clearly pairs of topologies $(\@O(X),\@O(G))$ obeying the second claim are closed under countable increasing joins.
For cofinality, fix any countable separating family of Borel sets $B_i \subseteq X/G$.
Then by \cref{thm:ucomqpolgpd-qpol}, we may find $\@O(X), \@O(G)$ forming an open quasi-Polish groupoid topology inducing the given fiberwise topology on $G$, and also making open any desired countable family of sets in $\@{BO}_G(X), \@{BO}_G(G)$, as well as each of the sets $\pi^{-1}(B_i)$.
Then clearly, each $C \in X/G$ will be $\*\Pi^0_2$ in $\@O(X)$, hence the full subgroupoid on $C$ is quasi-Polish, which means by \cref{thm:effros} that $\@O(X), \@O(G)$ restrict to the componentwise topologies on $C$.
The claimed formulas for $\@{BO}_G(X), \@{BO}_G(G)$ then follow from Kunugi--Novikov \cref{eq:fib-bor-qpol-basis}, and imply that $X, G$ form Borel-overt fiberwise quasi-Polish bundles over $X/G$, with countable Borel fiberwise bases given by any countable bases for $\@O(X), \@O(G)$, proving the first claim.

For the third claim, it is clear by \cref{rmk:orbtop-openquot} that $\tau(\@{BO}_\sigma(\sigma^{-1}(Y))) \subseteq \@O_G(X)$; we must verify that Borelness is preserved.
We treat $\tau$ as a map between the Borel-overt bundles
\begin{eqtikzcd*}
(\sigma^{-1}(Y), \@{BO}_\sigma(\sigma^{-1}(Y))) \drar["\pi\sigma"'] \rar["\tau"] &
(X, \@{BO}_G(X)) \dar["\pi"]
\\
& X/G
\end{eqtikzcd*}
By \cref{thm:fib-open-borel-reg}, while fiberwise openness of $\tau$ might not guarantee that it preserves Borel fiberwise open sets, it must at least map each $V \in \@{BO}_\sigma(\sigma^{-1}(Y))$ to some $\tau(V) =^*_G V' \in \@{BO}_G(X)$.
Since $\tau(V)$ is componentwise open, it follows by \cref{it:vaught-im} and Pettis \cref{it:vaught-pettis} that
\begin{eqaligned*}
\@{BO}_G(G) \cdot \tau(V)
= \@{BO}_G(G) * \tau(V)
= \@{BO}_G(G) * V'
\subseteq \@{BO}_G(X)
\end{eqaligned*}
where $\cdot$ and $*$ refer to the trivial action $G \actson X$.
Thus
\begin{eqaligned*}
\tau(\@{BO}_\sigma(\sigma^{-1}(Y)))
&= \tau(\@{BO}_G(G) \odot \@{BO}_\sigma(\sigma^{-1}(Y)))
    &&\text{by \cref{thm:ucomqpolgpd}\cref{thm:ucomqpolgpd:mu-fib}} \\
&= \@{BO}_G(G) \odot \tau(\@{BO}_\sigma(\sigma^{-1}(Y)))
    &&\text{(since $\tau(gh) = \tau(g)$)} \\
&\subseteq \@{BO}_G(X)
    &&\text{by above}.
\end{eqaligned*}
It follows that $\@{BO}_G(G)$ restricts to a Borel-overt fiberwise quasi-Polish topology on $\tau : \sigma^{-1}(Y) ->> X$ (whose fibers are the hom-sets $G(y,x)$ for $y \in Y$, which are Polish by \cref{rmk:homtop}), which hence has a Borel section $\gamma$ by the large section uniformization theorem \labelcref{thm:largeunif}.
\end{proof}

\begin{remark}
\Cref{thm:ucomqpolgpd-smooth} says that in some sense, the structure of a smooth Borel-overt uniformly componentwise quasi-Polish groupoid may be ``fully uniformly'' decomposed into its components.
Namely, both $G$ and $X$ may be treated as Borel-overt bundles over $X/G$, and all $5$ of the groupoid operations $\sigma, \tau, \mu, \iota, \nu$ are Borel fiberwise continuous over $X/G$, with $\sigma, \tau, \mu$ also being uniformly fiberwise open (i.e., preserving Borel fiberwise open sets), yielding a ``Borel-overt bundle of connected open quasi-Polish groupoids''.

If one is not interested in preserving the componentwise topology, then it is even possible to decompose each component, into the set of objects together with the automorphism group.
Precisely, letting $\upsilon : X ->> Y$ map each $x$ to the unique $y \in [x]_G \cap Y$, we have an isomorphism
\begin{align*}
\smash{\origdisplaystyle
\bigsqcup_{y \in Y} ([y]_G^2 \times G(y,y))}
= \set[\big]{(y,x,z,h) \in X^3 \times G | y = \upsilon(x) = \upsilon(z) = \sigma(h) = \tau(h)} &--> G \\
(y,x,z,h) &|--> \gamma(z) h \gamma(x)^{-1}
\end{align*}
which encodes each morphism $g \in G$ as its component ($y$), its two endpoints $(x,z)$, and its conjugate into the automorphism group $G(y,y)$ by the selected morphisms $\gamma$ from \cref{thm:ucomqpolgpd-smooth}.
The groupoid structure on the left is the disjoint union of product groupoids of the indiscrete equivalence relation $[y]_G^2$ and the group $G(y,y)$ on each component $[y]_G$.
\end{remark}

We may reformulate \cref{thm:ucomqpolgpd-smooth} for actions:

\begin{corollary}
\label{thm:ucomqpolgpd-smooth-action}
Let $(X,G)$ be a Borel-overt uniformly componentwise quasi-Polish groupoid, $p : M -> X$ be a Borel $G$-space with smooth orbit equivalence relation.
Then the quotient map $\pi : M ->> X/G$ with the orbitwise topology $\@{BO}_G(M)$ is a Borel-overt bundle of quasi-Polish spaces.

If $G$ is an open quasi-Polish groupoid, then for a club of quasi-Polish topologies $\@O(M) \subseteq \@{BO}_G(M)$ making the action continuous as in \cref{thm:beckec}, we have that the orbit equivalence relation $\#E_{G \ltimes M}$ is $\*\Pi^0_2$, and $\@O(M)$ restricts to the orbitwise topology $\@{BO}_G(M)$ on each orbit.

Moreover, for any Borel transversal $Y \subseteq X$, we have a uniformly open equivariant surjection
\begin{eqgathered*}
\alpha : G \times_X Y -->> M, \\
\@{BO}_\sigma(G) \cdot Y = \@{BO}_G(M);
\end{eqgathered*}
and there is a Borel map $\gamma : M -> G$ choosing some $\gamma(a)$ such that $\gamma(a)^{-1} \cdot x \in Y$ for each $a \in M$.
\end{corollary}
\begin{proof}
Follows from applying \cref{thm:ucomqpolgpd-smooth} to $G \ltimes M$, in the quasi-Polish case recalling from \cref{thm:beckec} that a club of compatible groupoid topologies on $G \ltimes M$ are of the form $\@O(G) \otimes_X \@O(M)$ for some topology $\@O(M)$ on $M$ making the action continuous.
\end{proof}

As another application of the above results, we have a groupoid version of the open mapping theorem for Polish groups (see \cite[2.3.3]{gao:idst}), which shows that a Borel surjective homomorphism $f : G ->> H$ from a Polish group onto a standard Borel group is a continuous open map to the unique Polish group topology on $H$.

\begin{lemma}
\label{thm:ftr-cts}
Let $F : (X,G) -> (Y,H)$ be a Borel functor between Borel-overt componentwise quasi-Polish groupoids.
\begin{enumerate}[alph]
\item  \label{thm:ftr-cts:fibcts}
$F$ is $\sigma$-fiberwise continuous (i.e., $F^{-1}(\@O_\sigma(H)) \subseteq \@O_\sigma(G)$, or $F : \sigma_G^{-1}(x) -> \sigma_H^{-1}(F(x))$ is continuous between the $\sigma$-fiberwise topologies for each $x \in X$) iff $F$ is componentwise continuous on morphisms, and this implies componentwise continuity on objects.
\item  \label{thm:ftr-cts:fibopen}
If $F$ is $\sigma$-fiberwise open, then $F$ is componentwise open on objects as well as morphisms.
\end{enumerate}
\end{lemma}
\begin{proof}
\cref{thm:ftr-cts:fibcts}
$\Longleftarrow$ since $\@O_G(G)$ fiberwise restricts to $\@O_\sigma(G)$ by \labelcref{thm:ucomqpolgpd}\cref{thm:ucomqpolgpd:fibtop};
$\Longrightarrow$ since $\mu_G^{-1}(F^{-1}(\@O_H(H))) = (F \times F)^{-1}(\mu_H^{-1}(\@O_H(H))) \subseteq (F \times F)^{-1}(\@O_\sigma(H) \otimes_X \@O_\tau(H)) \subseteq \@O_\sigma(G) \otimes_X \@O_\tau(G)$, and similarly $\tau_G^{-1}(F^{-1}(\@O_H(Y))) \subseteq \@O_\sigma(G)$.

\cref{thm:ftr-cts:fibopen}
For any $x \in X$, using \cref{rmk:orbtop-openquot}, we have
$F(\@O_G([x]_G))
= F(\tau_G(\@O_\sigma(\sigma_G^{-1}(x))))
= \tau_H(F(\@O_\sigma(\sigma_G^{-1}(x))))
\subseteq \tau_H(\@O_\sigma(\sigma_H^{-1}(F(x))))
= \@O_H([F(x)]_H)$;
and similarly we have
$F(\@O_G(\eval{G}_{[x]_G})) \subseteq \@O_H(\eval{H}_{[F(x)]_H})$,
using that $F$ preserves multiplication.
\end{proof}

\begin{theorem}[open mapping theorem]
\label{thm:openmap}
Let $F : (X,G) -> (Y,H)$ be a Borel surjective full\nobreak\space functor from a Borel-overt uniformly componentwise quasi-Polish groupoid to a standard Borel groupoid.
\begin{enumerate}[alph]

\item \label{thm:openmap:bij}
If $F$ is bijective on objects, then
\begin{eqaligned*}
\@{BO}_\sigma(H) := F(\@{BO}_\sigma(G)) \subseteq \@B(H);
\end{eqaligned*}
and the \defn{quotient fiberwise topology} generated by this fiberwise basis turns $H$ into a Borel-overt uniformly componentwise quasi-Polish groupoid, such that $F$ is fiberwise continuous open on morphisms, hence also componentwise continuous open on objects as well as morphisms.

\item \label{thm:openmap:open}
If $H$ is a Borel-overt fiberwise quasi-Polish groupoid, and $F$ is fiberwise continuous (on morphisms) and componentwise open on objects, then for every $x \in X$,
\begin{eqaligned*}
F(\@O_\sigma(\sigma_G^{-1}(x))) = \@O_\sigma(\sigma_H^{-1}(F(x))),
\end{eqaligned*}
i.e., $F$ is fiberwise open, hence also componentwise open on morphisms.

Hence if $F$ is moreover bijective on objects, then $F$ is ``uniformly fiberwise open'':
\begin{eqaligned*}
F(\@{BO}_\sigma(G)) = \@{BO}_\sigma(H).
\end{eqaligned*}

\end{enumerate}
\end{theorem}
\begin{proof}
\cref{thm:openmap:bij}
We may assume $X = Y$ and $F$ is the identity on objects.
Then we have the left translation action of $G$ on $\tau_H : H -> X$, where $g \cdot h := F(g)h$, whose orbits are the $\sigma_H$-fibers since $F$ is full.
By applying \cref{thm:ucomqpolgpd-smooth} to the transversal consisting of identity elements $\iota_H(X) \subseteq H$, we get that $\@{BO}_\sigma(H) := F(\@{BO}_\sigma(G)) \subseteq \@B(H)$ is a Borel-overt $\sigma$-fiberwise quasi-Polish topology on $H$, namely the orbitwise topology of $G \actson H$, and that $F : (G,\@{BO}_\sigma(G)) -> (H,\@{BO}_\sigma(H))$ is $\sigma$-fiberwise continuous (and uniformly fiberwise open by definition).
It is easily seen that $\@{BO}_\sigma(H)$ is invariant under right translation by $H$, again using that $F$ is full; thus $H$ becomes a Borel-overt fiberwise quasi-Polish groupoid.
Uniform fiberwise continuity of differences easily descends from $G$:
\begin{eqaligned*}
\mu_H^{-1}(\@{BO}_\sigma(H))
&= (F \times F)((F \times F)^{-1}(\mu_H^{-1}(\@{BO}_\sigma(H)))) \\
&= (F \times F)(\mu_G^{-1}(F^{-1}(\@{BO}_\sigma(H)))) \\
&\subseteq (F \times F)(\mu_G^{-1}(\@{BO}_\sigma(G))) \\
&\subseteq (F \times F)(\@{BO}_\tau(G) \otimes_X \@{BO}_\sigma(G)) \\
&= \@{BO}_\tau(H) \otimes_X \@{BO}_\sigma(H).
\end{eqaligned*}
By the preceding lemma, $F$ is componentwise continuous open.

\cref{thm:openmap:open}
Consider the \defn{pullback groupoid}
\begin{eqaligned*}
F^*(H) := X \times_Y H \times_Y X = \set{(x',h,x) \in X \times H \times X | h : F(x) -> F(x')}
\end{eqaligned*}
on $X$, with the hom-sets lifted from $H$, where $\sigma_{F^*(H)}, \tau_{F^*(H)}$ are the third and first projections, and $\mu, \iota, \nu$ are defined as in $H$; then $F$ factors as the identity-on-objects full functor $\~F := (\tau_G,F,\sigma_G) : G -> F^*(H)$ followed by the full and faithful functor $\pi_2 : F^*(H) -> H$.
The $\sigma_{F^*(H)}$-fibers are
\begin{eqaligned*}
\sigma_{F^*(H)}^{-1}(x) \cong [x]_G \times_Y \sigma_H^{-1}(F(x)).
\end{eqaligned*}
By \cref{thm:openmap:bij}, we get a Borel-overt $\sigma$-fiberwise topology on $F^*(H)$ which is generated by
\begin{eqaligned*}
\~F(U) \cong (\tau_G,F)(U) \subseteq [x]_G \times_Y \sigma_H^{-1}(F(x))
\end{eqaligned*}
for $U \in \@O_\sigma(\sigma_G^{-1}(x))$, for each $x \in X$.

If now $H$ is equipped with a Borel-overt componentwise quasi-Polish groupoid topology, and $F$ is fiberwise continuous, then for each $x \in X$, the topology
$\@O_G([x]_G) \otimes_Y \@O_\sigma(\sigma_H^{-1}(F(x)))$
on $\sigma_{F^*(H)}^{-1}(x)$ makes $\pi_1 : \sigma_{F^*(H)}^{-1}(x) -> (X,\@O_G(X))$ and the left action of $G$ continuous, hence by \cref{thm:effros-action} must in fact equal the quotient topology defined above.
If $F$ is moreover componentwise open on objects, then
$\pi_2 : \sigma_{F^*(H)}^{-1}(x) -> (\sigma_H^{-1}(F(x)), \@O_\sigma(\sigma_H^{-1}(F(x))))$
is open, being a pullback of $F : [x]_G -> [F(x)]_H$; thus
\begin{eqaligned*}
\pi_2((\tau_G,F)(U)) = F(U) \subseteq \sigma_H^{-1}(F(x))
\end{eqaligned*}
is $\@O_\sigma(H)$-open for each $U \in \@O_\sigma(\sigma_G^{-1}(x))$.
\end{proof}

\begin{remark}
As in \cref{rmk:gpd-sub-closure}, some obvious attempts to strengthen the above result fail:
\begin{itemize}

\item
The assumption in \cref{thm:openmap:open} that $F$ is open on objects cannot be dropped: consider a continuous bijection from the same space $X$ with a finer to a coarser topology, regarded as a functor between the indiscrete equivalence relations $X^2$.

\item
The assumption in \cref{thm:openmap:open} that $F$ is continuous on morphisms cannot be dropped, even assuming continuity and openness on objects: let $X$ be a non-discrete Polish space, with a non-open set $A \subseteq X$, and consider the product of the indiscrete equivalence relation $X^2$ and the discrete group $\#Z$, as a groupoid $G = X^2 \times \#Z$ on $X$; we have a functor $F : G -> G$ which fixes the objects, and maps $(x,y,n) |-> (x,y,n)$ if $(x \in A \iff y \in A)$, else $(x,y,-n)$.

\item
The same examples show that it is difficult to formulate a general ``automatic continuity'' theorem for Borel functors between open Polish groupoids, that yields more than just Pettis's theorem for Polish groups which implies continuity on hom-sets.
(See \cite{rosendal:cocycle} for some recent results on automatic continuity for functors on special types of groupoids.)

\item
We cannot weaken the assumption that $F$ is a surjective full functor to $F$ being a fibration (recall \cref{def:ftr-equiv}) which descends to an injection on the quotient $X/G `-> Y/H$.
Let $H$ be a non-discrete Polish group, let $X$ be $H$ with a strictly finer topology and $G$ be the indiscrete equivalence relation $X^2$, and let $F : G -> H$ take $(g,g') |-> g'g^{-1}$.
(So $G$ is the action groupoid of the free transitive action $H \actson H$, but with a strictly finer topology.)
Then $F$ is continuous and trivially componentwise open on objects, but clearly not fiberwise open.

\item
It is easy to see that \cref{thm:openmap:bij} as stated does not hold if $F$ is not assumed to be bijective: consider a badly behaved (non-Borel-overt) continuous surjection between uncountable Polish spaces with the indiscrete equivalence relations.
A suitable generalization to non-bijective quotient groupoids needs to require $F$ to yield well-behaved open quotient topologies on each component, along the lines of \cref{it:qpol-openquot}.
We leave the details of such a generalization to future work.

\end{itemize}
\end{remark}

The following is the quasi-Polish version of \cite[7.1.2]{becker-kechris:polgrp}, \cite[5.2.2]{lupini:polgpd}:

\begin{corollary}
\label{thm:ucomqpolgpd-erquot}
Let $(X,G)$ be a Borel-overt uniformly componentwise quasi-Polish groupoid, and suppose its connectedness relation $\#E_G \subseteq X^2$ is Borel.
Then the componentwise topology $\@{BO}_G(X)$ turns $\#E_G$ into a Borel-overt uniformly classwise quasi-Polish equivalence relation, such that the quotient functor $(\sigma,\tau) : G ->> \#E_G$ is fiberwise as well as componentwise continuous uniformly open:
\begin{eqaligned*}
\@{BO}_\sigma(\#E_G) = F(\@{BO}_\sigma(G)) \subseteq \@B(\#E_G).
\end{eqaligned*}

Thus more generally, for a Borel-overt uniformly componentwise quasi-Polish groupoid $(X,G)$ and Borel $G$-space $p : M -> X$ whose orbit equivalence relation $E := \#E_{G \ltimes M} \subseteq M^2$ is Borel, the orbitwise topology $\@{BO}_G(M)$ turns $E$ into a Borel-overt uniformly classwise quasi-Polish equivalence relation, whose fiberwise topology $\@{BO}_\sigma(E)$ has a Borel fiberwise open basis consisting of
\begin{eqaligned*}
\~U := \bigcup_{a \in M} (\set{a} \times Ua) = \set{(a,b) \in E | \exists g \in U\, (ga = b)}
\end{eqaligned*}
for any Borel fiberwise open basis of $U \in \@{BO}_\sigma(G)$.
\end{corollary}
\begin{proof}
Follows from \cref{thm:openmap}\cref{thm:openmap:bij} for $H := \#E_G$, recalling from \cref{rmk:bocqper} that the fiberwise topology on an equivalence relation $E$ is determined by the classwise topology $\@O_E(X)$.
\end{proof}

\begin{remark}
We may reformulate this result more abstractly as in \cite[7.1.2]{becker-kechris:polgrp}, as follows.

For any standard Borel bundle of quasi-Polish spaces $p : X -> Y$, define the \defn{fiberwise lower powerspace} $\@F_p(X)$ to be the space of pairs $(y,F)$ where $y \in Y$ and $F \subseteq X_y$ is closed, equipped with the Borel sigma-algebra generated by the first projection $\@F_p(X) -> Y$ as well as the sets
\begin{eqaligned*}
\Dia_Y U := \set{(y,F) \in \@F_p(X) | F \cap U \ne \emptyset}
\end{eqaligned*}
for $U \in \@{BO}_p(X)$, and the $p$-fiberwise topology generated by the same sets as a fiberwise subbasis.
This is the Borel fiberwise topological version of the quasi-Polish fiberwise lower powerspace from \cite[2.5.7, 2.5.9]{chen:beckec}, which along with \cref{thm:fib-bor-qpol} shows that $\@F_p(X)$ is indeed a standard Borel bundle of quasi-Polish spaces over $Y$.
For $p$ to be Borel-overt means precisely that
\begin{eqaligned*}
p^* : Y &--> \@F_p(X) \\
y &|--> (y, X_y)
\end{eqaligned*}
is Borel; indeed, $(p^*)^{-1}(\Dia_Y U) = p(U)$.
(One often abuses notation and refers to an element of $\@F_p(X)$ as just a closed set $F$, instead of the pair $(y,F)$, when the basepoint $y$ is clear from context.)

Now \cref{thm:ucomqpolgpd-erquot} shows that for a Borel-overt uniformly componentwise quasi-Polish groupoid $(X,G)$ with $\#E_G$ Borel, the bundle $(\sigma,\tau) : G ->> \#E_G \subseteq X^2$ equipped with the hom-set-wise topology (the restriction of the $\sigma$- and $\tau$-fiberwise topologies; see \cref{thm:ucomqpolgpd-homtop}) is Borel-overt, i.e., the map
\begin{align*}
X^2 &--> \@F_{(\sigma,\tau)}(G) \\
(x,y) &|--> G(x,y) \text{ (i.e., $(x,y,G(x,y))$)}
\shortintertext{is Borel.
In particular,}
\Aut_G : X &--> \@F_{(\sigma,\tau)}(G) \\
x &|--> \Aut_G(x) = G(x,x)
\end{align*}
is Borel.
In the case of a Borel $G$-space $M$ with Borel orbit equivalence relation, these maps become
\begin{align*}
\begin{aligned}[t]
M^2 &--> \@F_{(\sigma,\tau)}(G) \\
(a,b) &|--> \set{g : p(a) -> p(b) | ga = b},
\end{aligned}
&&
\begin{aligned}[t]
\Stab_G : M &--> \@F_{(\sigma,\tau)}(G) \\
a &|--> \set{g : p(a) -> p(a) | ga = a}.
\end{aligned}
\end{align*}
Similarly, \labelcref{thm:openmap}\cref{thm:openmap:bij} shows that for a full identity-on-objects Borel functor $F : (X,G) ->> (X,H)$,
the ``kernel'' map
$X \ni x |-> F^{-1}(1_x) \in \@F_{(\sigma,\tau)}(G)$
is Borel.
\end{remark}

% remove MR number from references
\def\MR#1{}
\bibliographystyle{amsalpha}
\bibliography{refs}

\medskip
\noindent
Department of Mathematics\\
University of Michigan\\
Ann Arbor, MI, USA\\
\nolinkurl{ruiyuan@umich.edu}

\end{document}